\documentclass[reqno,11pt]{amsart}
\usepackage{amsmath,amssymb,graphicx,amsthm,fullpage,xcolor,mathrsfs}
\usepackage{stmaryrd}

\usepackage{subfigure}

\usepackage{xifthen,xcolor,tikz,setspace}
\usetikzlibrary{decorations.pathmorphing,decorations.pathreplacing,patterns,shapes,calc}
\usepackage{bm}
\usepackage[color]{}

\usepackage[normalem]{ulem}

\usepackage{prettyref}
\usepackage{bbm}
\usepackage{enumerate}

\usepackage[colorlinks=true]{hyperref}
\colorlet{DarkGreen}{green!50!black}
\colorlet{DarkGray}{gray!60!black}

\DeclareMathAlphabet{\mathcal}{OMS}{cmsy}{m}{n}

\newtheorem{theorem}{Theorem}[section]
\newtheorem*{theorem*}{Theorem}
\newtheorem{lemma}[theorem]{Lemma}
\newtheorem*{lemma*}{Lemma}
\newtheorem{claim}[theorem]{Claim}
\newtheorem{proposition}[theorem]{Proposition}

\newtheorem{corollary}[theorem]{Corollary}

\theoremstyle{definition}{

\newtheorem{definition}[theorem]{Definition}
\newtheorem*{definition*}{Definition}

}

\numberwithin{equation}{section}

\newcommand{\cC}{\mathcal{C}}

\newcommand{\cN}{\mathcal{N}}
\newcommand{\cE}{\mathcal{E}}

\newcommand{\fh}{\mathfrak{h}}
\newcommand{\fE}{\mathfrak{E}}

\newcommand{\fR}{\mathfrak{R}}

\newcommand{\sI}{\mathscr{I}}

\newcommand{\bmB}{{\sf{Tile}}}

\newcommand{\e}{\varepsilon}

\newcommand{\vse}{{v^*_{\south\east}}}
\newcommand{\vsw}{{v^*_{\south\west}}}
\newcommand{\ve}{{v^*_\east}}
\newcommand{\vw}{{v^*_\west}}
\newcommand{\we}{{w^*_\east}}
\newcommand{\ww}{{w^*_\west}}
\newcommand{\window}{{\sf W}}

\newcommand{\nbt}{{\mathsf{NoBT}}}

\newcommand{\tshape}{{\mathsf T}}
\newcommand{\senlarge}{{\mathsf S}}
\newcommand{\lshape}{{\mathsf L}}
\newcommand{\enlarge}{{\mathsf E}}
\newcommand{\corridor}{{\mathsf {Cor}}}

\newcommand{\llb }{\llbracket}
\newcommand{\rrb }{\rrbracket}
 
 \newcommand{\north}{{\textsc{n}}}
\newcommand{\south}{{\textsc{s}}}
\newcommand{\east}{{\textsc{e}}}
\newcommand{\west}{{\textsc{w}}}

\newcommand{\interior}{\mathsf{int}}

\DeclareMathOperator{\hgt}{ht}

\newcommand{\str}{{\mathsf{Strip}}}
\newcommand{\spiky}{{\mathsf{Spiky}}}
\newcommand{\s}{{\mathsf{stop}}}

\newcommand{\co}{{\mathsf{Cross}}}
\newcommand{\circuit}{{\mathsf{Circ}}}

\renewcommand{\liminf}{\varliminf}
\newrefformat{prop}{Proposition~\ref{#1}}
\newrefformat{fig}{Figure~\ref{#1}}
\newrefformat{cor}{Corollary~\ref{#1}}
\newrefformat{conj}{Conjecture~\ref{#1}}

\begin{document}

\title[Geometry of the 2D Ising interface in critical pre-wetting]
{Local and global geometry of the 2D Ising interface in critical pre-wetting}

\author{Shirshendu Ganguly}
\author{Reza Gheissari}

\address[Shirshendu Ganguly]{Department of Statistics, UC Berkeley}\email{sganguly@berkeley.edu}
\address[Reza Gheissari]{Departments of Statistics and EECS, UC Berkeley}\email{gheissari@berkeley.edu}

\maketitle
\vspace{-.5cm}
\begin{abstract}
Consider the Ising model at low-temperatures and positive external field $\lambda$ on an $N\times N$ box with \emph{Dobrushin boundary conditions} that are plus on the north, east, and west boundaries and minus on the south boundary. If $\lambda = 0$, the interface separating the plus and minus phases is diffusive, having $O(\sqrt N)$ height fluctuations, and the model is \emph{fully wetted}. Under an order one field, the interface fluctuations are $O(1)$ and the interface is only partially wetted, being pinned to its southern boundary. We study the \emph{critical pre-wetting regime} of $\lambda_N \downarrow 0$, where the height fluctuations are expected to scale as $\lambda^{ -1/3}$ and the rescaled interface is predicted to converge to the Ferrari--Spohn diffusion. Velenik (2004)~\cite{Velenik} identified the order of the area under the interface up to logarithmic corrections. Since then, more refined features of such interfaces have only been identified in simpler models of random walks under area tilts. 

In this paper, we resolve several conjectures of Velenik regarding the refined features of the Ising interface in the critical pre-wetting regime. Our main result is a sharp bound on the one-point height fluctuation, proving $e^{ - \Theta(x^{3/2})}$ upper tails reminiscent of the Tracy--Widom distribution, capturing a tradeoff between the locally Brownian oscillations and the global field effect. We further prove a concentration estimate for the number of points above which the interface attains a large height.  
These are used to deduce various geometric properties of the interface, including the order and tails of the area it confines, and the poly-logarithmic pre-factor governing its maximum height fluctuation. Our arguments combine classical inputs from the random-line representation of the Ising interface, with novel local resampling and coupling schemes.  
\end{abstract}

{\hypersetup{linkcolor=black}
\tableofcontents}
\vspace{-.5cm}

\vspace{-.4cm}
\section{Introduction, main results, and key ideas}
Many models of stochastic interfaces and random growth in two dimensional statistical mechanics exhibit a competition between  local roughening forces and global smoothening constraints. These are expected to exhibit characteristic spatial and temporal fluctuation exponents as predicted in the seminal work \cite{KPZ}. The last few decades have witnessed a mathematical revolution with rigorous verification of such predictions in a handful of cases, mostly relying on remarkable connections to algebraic objects such as random matrices and determinantal point processes: see e.g.,~\cite{corwin1,corwin2}.  Perhaps the most well-known such example is Dyson Brownian motion (DBM), which considers $N$ independent Brownian curves starting from zero and conditioned to not intersect; DBM describes the evolution of eigenvalues of a random Hermitian matrix under a Brownian flow on its entries. The non-intersection constraint causes the top curve
to follow a parabolic trajectory with locally Brownian fluctuations~{\cite{prahofer1}; notably, its renormalized one-point distribution converges to the GUE Tracy-Widom distribution 
with the characteristic non-Gaussian $e^{-\Theta(x^{3/2})}$ upper tail ($x>0$).

An interesting model for the top curve in a DBM, which was analyzed by Ferrari-Spohn~\cite{FerrariSpohn} relying on exact expressions, is a Brownian motion 
constrained to stay above a circular or parabolic barrier; the resulting process was shown to converge to an Airy type process which has subsequently become known as the Ferrari--Spohn (FS) diffusion. 
The hard floor effect in a Brownian excursion 
constrained to confine a semi-circle or parabola   
makes it a paradigm of a wide array of random curves exhibiting entropic repulsion arising from global constraints on the area they confine. 

Such phenomena are believed to arise naturally in various contexts in statistical physics: these include two-dimensional FK percolation (the relevant quantity being a sub-critical cluster conditioned to confine large area~\cite{alexander1, hammond1, hammond2}), and perhaps most canonically the low-temperature two and three-dimensional Ising models:  namely in the phase separation lines in two dimensions in the critical pre-wetting regime, and in the ensemble of level lines of separation surfaces in three dimensions in the presence of a hard floor, and the associated wetting transition~\cite{FisherWalksWetting,IoVe16,VelenikSurvey}. Prototypes for these, including a class of random walks (e.g., the (1+1)D SOS model), and families of such random walks conditioned on non-intersection, under linear area tilts have been found to indeed have FS and DBM limits, respectively~\cite{Abraham86,BEF86,HrVe04,ISV15,IVW18,CLMST2}.  
In this article, we focus on the fundamental example of the phase separation line (interface) of the Ising model in its critical pre-wetting regime, addressing several long standing questions and conjectures from \cite{Velenik,IoVe16} about its global, and more delicate local, properties.  

\begin{figure}
\begin{tikzpicture}[scale = .75]

\tikzstyle{dual-site}=[fill={rgb,255: red,255; green,155; blue,155}, draw={rgb,255: red,191; green,0; blue,64}, shape=circle]
\tikzstyle{plus-site}=[fill=red, opacity =.3, shape=circle, font = \tiny]
\tikzstyle{minus-site}=[fill=blue, opacity =.3, shape=circle, font = \tiny]

\foreach \i in {-1,...,13}{\node[style = minus-site] (\i,-1) at (\i,-1) {$-$};};
\foreach \i in {-1,13}{\foreach \j in {0,...,6}{\node[style = plus-site] (\i,\j) at (\i,\j) {$+$};}};
\foreach \i in {0,...,12}{\foreach \j in {6}{\node[style = plus-site] (\i,\j) at (\i,\j) {$+$};}};


%
		\node [style=plus-site] (1,5) at (1, 5) {$+$};
		\node [style=plus-site] (1,4) at (1, 4) {$+$};
		\node [style=minus-site] (1,3) at (1, 3) {$-$};
		\node [style=plus-site] (1,2) at (1, 2) {$+$};
		\node [style=minus-site] (1,1) at (1, 1) {$-$};
		\node [style=minus-site] (2, 2) at (2, 2) {$-$};
		\node [style=plus-site] (2, 3) at (2, 3) {$+$};
		\node [style=minus-site] (2, 4) at (2, 4) {$-$};
		\node [style=minus-site] (2, 5) at (2, 5) {$-$};
		\node [style=plus-site] (3,5) at (3, 5) {$+$};
		\node [style=plus-site] (3,4) at (3, 4) {$+$};
		\node [style=minus-site] (3,3) at (3, 3) {$-$};
		\node [style=plus-site] (3,2) at (3, 2) {$+$};
		\node [style=plus-site] (2,0) at (2, 0) {$+$};
		\node [style=plus-site] (2,1) at (2, 1) {$+$};
		\node [style=minus-site] (3,1) at (3, 1) {$-$};
		\node [style=plus-site] (3,0) at (3, 0) {$+$};
		\node [style=plus-site] (4,1) at (4, 1) {$+$};
		\node [style=plus-site] (4,2) at (4, 2) {$+$};
		\node [style=minus-site] (4,3) at (4, 3) {$-$};
		\node [style=plus-site] (4,4) at (4, 4) {$+$};
		\node [style=minus-site] (4,5) at (4, 5) {$-$};
		\node [style=plus-site] (5,5) at (5, 5) {$+$};
		\node [style=minus-site] (5,4) at (5, 4) {$-$};
		\node [style=minus-site] (5,3) at (5, 3) {$-$};
		\node [style=minus-site] (5,2) at (5, 2) {$-$};
		\node [style=minus-site] (5,1) at (5, 1) {$-$};
		\node [style=minus-site] (6,2) at (6, 2) {$-$};
		\node [style=plus-site] (6,3) at (6, 3) {$+$};
		\node [style=minus-site] (6,4) at (6, 4) {$-$};
		\node [style=plus-site] (6,5) at (6, 5) {$+$};
		\node [style=plus-site] (7,5) at (7, 5) {$+$};
		\node [style=plus-site] (7,4) at (7, 4) {$+$};
		\node [style=plus-site] (7,3) at (7, 3) {$+$};
		\node [style=minus-site] (66) at (7, 2) {$-$};
		\node [style=plus-site] (6,1) at (6, 1) {$+$};
		\node [style=plus-site] (7,1) at (7, 1) {$+$};
		\node [style=minus-site] (8,5) at (8, 5) {$-$};
		\node [style=plus-site] (8,4) at (8, 4) {$+$};
		\node [style=plus-site] (8,3) at (8, 3) {$+$};
		\node [style=minus-site] (8,2) at (8, 2) {$-$};
		\node [style=minus-site] (8,1) at (8, 1) {$-$};
		\node [style=minus-site] (9,5) at (9, 5) {$-$};
		\node [style=minus-site] (9,4) at (9, 4) {$-$};
		\node [style=plus-site] (9,3) at (9, 3) {$+$};
		\node [style=plus-site] (9,2) at (9, 2) {$+$};
		\node [style=plus-site] (10,5) at (10, 5) {$+$};
		\node [style=plus-site] (10,4) at (10, 4) {$+$};
		\node [style=plus-site] (10,3) at (10, 3) {$+$};
		\node [style=minus-site] (10,2) at (10, 2) {$-$};
		\node [style=plus-site] (10,1) at (10, 1) {$+$};
		\node [style=minus-site] (9,1) at (9, 1) {$-$};
		\node [style=plus-site] (11,1) at (11, 1) {$+$};
		\node [style=plus-site] (11,5) at (11, 5) {$+$};
		\node [style=plus-site] (11,4) at (11, 4) {$+$};
		\node [style=plus-site] (11,3) at (11, 3) {$+$};
		\node [style=plus-site] (11,2) at (11, 2) {$+$};
			\node [style=plus-site] (11,0) at (11, 0) {$+$};

		\node [style=minus-site] (12,0) at (12, 0) {$-$};
		\node [style=minus-site] (12,1) at (12, 1) {$-$};
		\node [style=plus-site] (12,5) at (12, 5) {$+$};
		\node [style=minus-site] (12,4) at (12, 4) {$-$};
		\node [style=plus-site] (12,3) at (12, 3) {$+$};
		\node [style=plus-site] (12,2) at (12, 2) {$+$};
		\node [style=plus-site] (0,5) at (0, 5) {$+$};
		\node [style=plus-site] (0,4) at (0, 4) {$+$};
		\node [style=plus-site] (0,3) at (0, 3) {$+$};
		\node [style=plus-site] (0,2) at (0, 2) {$+$};
		\node [style=plus-site] (0,1) at (0, 1) {$+$};

		\node [style=minus-site] (0,0) at (0, 0) {$-$};
		\node [style=minus-site] (1,0) at (1, 0) {$-$};
		\node [style=minus-site] (4,0) at (4, 0) {$-$};
		\node [style=minus-site] (5,0) at (5, 0) {$-$};
		\node [style=minus-site] (6,0) at (6, 0) {$-$};
		\node [style=minus-site] (7,0) at (7, 0) {$-$};
		\node [style=minus-site] (8,0) at (8, 0) {$-$};
		\node [style=minus-site] (9,0) at (9, 0) {$-$};
		\node [style=minus-site] (10,0) at (10, 0) {$-$};

\draw[rounded corners, line width = 2.2pt, black] (-1.5,-.5)--(-.5,-.5)--(-.5,.5)--(.5,.5)--(.5,1.5)--(1.5,1.5)--(1.5,2.5)--(2.5,2.5)--(2.5,3.5)--(3.5,3.5)--(4.5,3.5)--(4.5,4.5)--(5.5,4.5)--(6.5,4.5)--(6.5,3.5)--(5.5,3.5)--(5.5,2.5)--(6.5,2.5)--(7.5,2.5)--(8.5,2.5)--(8.5,1.5)--(9.5,1.5)--(9.5,2.5)--(10.5,2.5)--(10.5,1.5)--(9.5,1.5)--(9.5,.5)--(10.5,.5)--(10.5,-.5)--(11.5,-.5)--(11.5,1.5)--(12.5,1.5)--(12.5,-.5)--(13.5,-.5); 

\draw[rounded corners, line width = 1.3pt, blue] (1,2.5)--(1.5,2.5)--(1.5,3.5)--(2.5,3.5)--(2.5,4.5)--(2.5,5.5)--(1.5,5.5)--(1.5,4.5)--(1.5,3.5)--(.5,3.5)--(.5,2.5)--(1,2.5);

\draw[rounded corners,  line width = 1.3pt, blue] (4,4.5)--(4.5,4.5)--(4.5,5.5)--(3.5,5.5)--(3.5,4.5)--(4,4.5);

\draw[rounded corners,  line width = 1.3pt, blue] (2.5,1)--(2.5,1.5)--(3.5,1.5)--(3.5,.5)--(2.5,.5)--(2.5,1);

\draw[rounded corners,  line width = 1.3pt, blue] (11.5,4)--(11.5,4.5)--(12.5,4.5)--(12.5,3.5)--(11.5,3.5)--(11.5,4);

\draw[rounded corners,  line width = 1.3pt, red] (2,1.5)--(2.5,1.5)--(2.5,2.5)--(3.5,2.5)--(4.5,2.5)--(4.5,1.5)--(4.5,.5)--(3.5,.5)--(3.5, - .5) -- (2.5,-.5)--(1.5,-.5)--(1.5,.5)--(1.5,1.5)--(2,1.5);

\draw[rounded corners,  line width = 1.3pt, red] (6.5,.5)--(7.5,.5)--(7.5,1.5)--(6.5,1.5)--(5.5,1.5)--(5.5,.5)--(6.5,.5);

\draw[rounded corners,  line width = 1.3pt, blue] (7.5,5)--(7.5,5.5)--(8.5,5.5)--(9.5,5.5)--(9.5,4.5)--(9.5,3.5)--(8.5,3.5)--(8.5,4.5)--(7.5,4.5)--(7.5,5);

\end{tikzpicture}
\vspace{-.1cm}
\caption{An Ising configuration with $\pm$-boundary conditions. The dual-edges separating differing spins are depicted; in black is the \emph{interface}, $\mathscr I$.}
\label{fig.interface}\vspace{-.2cm}
\end{figure}
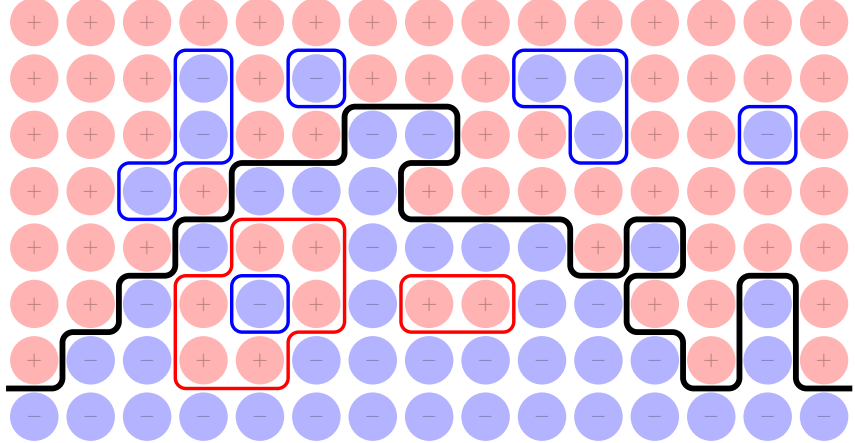

Before proceeding let us define the two-dimensional Ising model and formalize the setting we are interested in. Let $\mathbb Z^2$ denote the integer lattice graph with vertices at $\mathbb Z^2$ and with nearest neighbor edges $E_{\mathbb Z^2}$. Consider the box $\Lambda_N$ of side-length $N$ with vertices 
\begin{align*}
\Lambda_N : = \llb 0 ,N\rrb \times \llb 0 ,N \rrb  = \{0,...,N\} \times \{0,...N\}\,,
\end{align*}
and with nearest-neighbor edges $E_N$. A boundary condition $\eta$ on $\Lambda_N$ is an assignment of $\{\pm 1\}$ \emph{spins} to the vertices of $\mathbb Z^2 \setminus \Lambda_N$.  
The Ising model on $\Lambda_N$ with boundary conditions $\eta$, at inverse-temperature $\beta$ with external field $\lambda$, is the probability measure on configurations $\sigma \in \{\pm 1\}^\Lambda$, 
\begin{align*}
    \mu_{\beta,\lambda,N}^\eta(\sigma)= \frac 1{Z_{\eta,\beta,\lambda,N}} \exp\Bigg(- 2\beta\Big( \sum_{(v,w)\in E_N} \mathbf 1\{\sigma_v\neq \sigma_w\} + \sum_{\substack{(v,w)\in E_{\mathbb Z^2} \\  v\in \Lambda_N , w\notin \Lambda_N}} \mathbf 1\{ \sigma_v\neq \eta_w\} \Big) + \lambda \sum_{v\in \Lambda_N} \sigma_v\Big)\Bigg)\,,
\end{align*}
where $Z_{\eta,\beta,\lambda,\Lambda}$ is a normalizing constant called the \emph{partition function}. The 2D Ising model at $\lambda = 0$ has a sharp phase transition at a critical point $\beta_c$ roughly described as follows. At high temperatures, when $\beta<\beta_c$, under $\eta = +$ boundary conditions, the effect of the boundary decays exponentially in the distance and e.g., the center site $(N/2,N/2)$ has probability $\frac{1}{2} \pm e^{ - \Omega(N)}$ of being plus or minus. At low temperatures,  $\beta>\beta_c$ under $\eta = +$ (resp., $\eta = -$), the effect of the boundary does not decay in $N$, and uniformly in $N$, the center site $(N/2,N/2)$ is more likely to be plus (resp., minus) than not; the plus and minus \emph{phases} are both stable. 

In this paper, we are interested in the low-temperature regime under \emph{Dobrushin boundary conditions}, $\eta = \pm$, which are $+$ on the upper half plane $(\mathbb Z\times\mathbb Z_{\geq 0}) \setminus \Lambda$ and $-$ elsewhere on $\mathbb Z^2 \setminus \Lambda$. 
The boundary condition induces an \emph{interface} separating the plus and minus phases (see Figure~\ref{fig.interface}), i.e., for a configuration $\sigma$, define the \emph{interface}, denoted by ${\mathscr I}={\mathscr I}(\sigma)$, as the simple path ({using the usual south-east splitting rule---see Section~\ref{subsec:Ising-interface} for a formal definition}), amongst the set of dual-edges separating differing spins in $\sigma$, that splits the plus boundary from the minus boundary.

When $\beta>\beta_c$ and $\lambda = 0$ the behavior of $\mathscr I$ is diffusive and scales to a Brownian bridge~\cite{Hryniv98,DH97,GrIo05,InvarianceFloor2020}: in particular, on $\Lambda_N$, it has $O(\sqrt N)$ height fluctuations, and its maximum height scales like $O(\sqrt N)$ as well. In fact, $\sI$ can be coupled to an approximate random walk, as encoded by the powerful \emph{Ornstein--Zernike theory} of~\cite{CIV03}. Here, the interface is \emph{fully wetted}, as the predominant plus phase is separated from the minus phase by a mesoscopic layer of minuses induced by the entropic repulsion of the minus floor.   In the presence of an order one field ($\lambda>0$ uniformly in $N$), the interface fluctuations become $O(1)$ and the interface is only \emph{partially wetted}. The focus of this paper is the transition between these two regimes as $\lambda = \lambda_N$ decays to zero with $N$. In this regime, there is a delicate tradeoff between the entropic repulsion from the minus floor, and the positive external field; balancing these two effects, it has been predicted that the interface height diverges continuously as $\Theta(\lambda^{-1/3})$: this is called the \emph{critical pre-wetting regime} and is believed to have the Ferrari--Spohn diffusion as its scaling limit in local $\lambda^{-2/3} \times \lambda^{-1/3}$ windows~\cite[Conjecture 3.2]{IoVe16}.

In a breakthrough paper \cite{Velenik}, Velenik made significant progress in the understanding of Ising interface in this critical pre-wetting regime. Namely, he showed that the interface has $\lambda^{-1/3}$ height oscillations in the following averaged sense: the area confined by $\mathscr I$ has size at least $N \lambda^{-1/3} (\text{polylog}(1/\lambda))^{-1}$ and the number of minus sites in the component of the southern boundary (presumably a proxy for the area confined by $\mathscr I$), has size at most $N\lambda^{-1/3} \text{polylog}(1/\lambda)$.
The work also introduced a number of new techniques that were the starting point of our paper and have proven to be useful in several follow up works. Since~\cite{Velenik}, very refined results on random walk approximations to the critical pre-wetting Ising interface have been obtained; these include the local behavior and sharp Tracy--Widom type tails, and subsequently the  convergence to Ferrari--Spohn diffusion for a universal family of random-walks under area tilts~\cite{HrVe04,ISV15}.
These random walk models admit various important structural properties that do not hold for the Ising interface: 
\begin{enumerate}
 \item Their interfaces are graphs of functions whereas $\mathscr I$ is a simple curve (and may backtrack).
  \item They are distributions only over the interface whereas the Ising measure also depends on the configuration above and below the interface.
  \item Most importantly their interfaces are Markovian, whereas due to (1)--(2), the law of $\mathscr I$ does not satisfy any domain Markov property.
\end{enumerate}
In this paper, following the program initiated in~\cite{Velenik}, we develop a technology to overcome these obstacles and extend the universality results of~\cite{HrVe04,ISV15} to the Ising interface. Our paper 
resolves a number of open questions including establishing sharp Tracy--Widom type tails for the normalized one-point height oscillation---strong evidence for the conjectured convergence to the FS diffusion in local windows---and the global order of the area under and maximum height oscillation of $\mathscr I$.
Our proofs merge several approaches, combining inputs from the, by now traditional, \emph{random-line representation} of 2D Ising interfaces via duality with high temperature two-point functions, with local resampling and coupling methods; in particular, we do not use the full strength of the Ornstein--Zernike machinery of~\cite{CIV03}, instead using percolation-theoretic arguments to decouple long-range dependencies in the interface. We elaborate on our approach in Section~\ref{sec:ideas-of-proofs}.

\subsection{Statements of main results}  In this section,  we describe our main results on the 2D Ising interface in its critical pre-wetting regime on $\Lambda_N$. For concreteness and ease of presentation, we take $\lambda$ going to zero as $\Theta(N^{-1})$, which is also the critical scale for the bulk magnetization in the presence of an external field~\cite{Martirosyan,ScShJSP,ScShCMP}. Namely, we take $\lambda = \lambda_N$ such that $N\lambda$ is constant and for distinctness, reserve the notation $c_\lambda>0$ for that constant. In the regime where $\lambda \downarrow 0$ more generally, the critical pre-wetting behavior should depend continuously on $\lambda$ and many of our arguments should generalize to a wider range of $\lambda \downarrow 0$ straightforwardly, though we expect them to  break down at $\lambda$ decreasing sufficiently slowly.

\medskip
\noindent
\textbf{Height of the interface.}  Recall that $\mathscr I$ is the (edge-disjoint) path, amongst the dual-edges separating differing spins, connecting $(-\frac 12,- \frac 12)$ to $(N+\frac 12, - \frac 12)$ (see Definition~\ref{def:Ising-interface}). We wish to assign a \emph{height} to the interface above each index point $x \in \llb 0,N\rrb$. Since $\mathscr I$ may backtrack, such a height is a priori ill-defined. As such, we consider the lowest and highest intersection points of $\mathscr I$ with the column $\{x\} \times \mathbb R$ and denote them by $\hgt_x^- = \hgt_x^-(\mathscr I)$ and $\hgt_x^+ = \hgt_x^+(\mathscr I)$ respectively. 

\subsection{Sharp upper tail behavior}
Our first result is the sharp upper tail for the height of $\mathscr I$ above any given point rescaled by $N^{-1/3}$, analogous to the well known Tracy--Widom distribution. Many of the new ideas and work in this paper go into establishing these one-point tail estimates.   

\begin{theorem}\label{thm:one-point-tail-bounds}
Let $\beta>\beta_c$ and $\lambda = \frac{c_\lambda}{N}$ for $c_\lambda>0$ and fix any $\varepsilon>0$. There exist $c,C>0$ such that for every sequence $x= x_N \in \llb \varepsilon N, (1-\varepsilon)N \rrb$ and every $R= R_N = o(N^{2/9})$,  
 \begin{align*}
     ce^{-CR^{3/2}} & \leq \mu_{\beta,\lambda,N}^{\pm} (\hgt_x^- >RN^{1/3})  \leq  \mu_{\beta,\lambda,N}^\pm (\hgt_x^+> RN^{1/3}) \leq Ce^{-cR^{3/2}} 
\end{align*}
\end{theorem}
While the above is written for all $x$ in the bulk, we in fact prove stronger versions, holding for all $x$ in the upper bound and closer to the boundary in the lower bound: see Propositions~\ref{prop:right-tail-ub}--\ref{prop:right-tail-lb}. 

This tail behavior follows from a picture of the interface wherein it has heights that are of order $N^{1/3}$ and its correlations decay on the $O(N^{2/3})$ scale. The typical behavior of $\mathscr I$, when it attains a height of $RN^{1/3}$ above $x$ is to be lifted on a horizontal window of $\sqrt R N^{2/3}$, whereby the locally Brownian oscillations of the interface dictate the above tail behavior. This behavior is characteristic of the upper tail of locally Brownian processes above either a hard or soft floor including the top curve of the DBM, FS diffusions, and the random walk and SOS models under linear area tilts.

\subsubsection{Area confined under ${\mathscr I}$:} 
Our next result harnesses the above one-point upper tail to deduce sharp global behavior of the interface $\sI$. Viewing $\mathscr I$ as a simple curve in $\mathbb R^2$ connecting the corners $(-\frac 12,- \frac 12)$ to $(N+\frac 12, - \frac 12)$, we denote by $\Lambda^+ = \Lambda^+({\mathscr I})$ the subset of $\Lambda_N$ ``above" the interface ${\mathscr I}$, and by $\Lambda^- = \Lambda^-({\mathscr I})$ the subset ``below" the interface (see Definition~\ref{def:Lambda-minus}). For a configuration~$\sigma$, let  $\mathcal C^-= \mathcal C^-(\sigma)$ be the connected component of $-$ spins containing the bottom boundary of~$\Lambda_N$.

As mentioned, it was shown in~\cite{Velenik} that the area under $\mathscr I$, i.e., $|\Lambda^-(\mathscr I)|$, is at least order $N^{4/3} (\log N)^{-C}$, and the size of the minus component of the bottom boundary, i.e., $|\mathcal C^-(\sigma)|$ is at most order $N^{4/3} (\log N)^C$. Using Theorem~\ref{thm:one-point-tail-bounds} one already deduces tightness of the rescaled area without the extraneous polylogarithmic prefactors. We further prove a multi-point version of the upper bound of Theorem~\ref{thm:one-point-tail-bounds}, and show (almost) sharp upper and lower tail bounds on $|\Lambda^-|$ at the $N^{4/3}$ scale.   
The following proves Conjecture 3.1 of the survey~\cite{IoVe16}.

\begin{theorem}\label{thm:area-under-interface}
Let $\beta>\beta_c$ and $\lambda = \frac{c_\lambda}{N}$ for $c_\lambda>0$. The sequence of areas confined under $\mathscr I$, i.e.,
\begin{align*}
\big(N^{-4/3} |\Lambda^-(\mathscr I)|\big)_{N\geq 1}\qquad \mbox{is uniformly tight under $\mu_{\beta,\lambda,N}^{\pm}$}\,.
\end{align*}
Moreover, there exists $K(\beta, c_\lambda)>0$ such that 
\begin{align}\label{eq:area-tail-bounds}
\mu_{\beta, \lambda, N}^{\pm} \big(  N^{-4/3} | \Lambda^-(\sI)|\notin [K^{-1},K]\big) \leq \exp ( - N^{\frac 13 -o(1)})\,.
\end{align}
\end{theorem}
The $o(1)$ in the exponent is an artifact of the proof of the upper tail, and one expects that the above holds without it. As mentioned, an important ingredient of the proof of the upper bound is a concentration bound for the number of high points of the interface (see Thorem~\ref{thm:multipoint-height-bounds}). It is worth emphasizing that this is obtained by developing a machinery to prove exponential correlation decay in the interface heights at the critical order $N^{2/3}$ widths. 
We discuss this further in Section \ref{sec:ideas-of-proofs}.

\subsubsection{Minus component of the boundary} Using Theorem~\ref{thm:area-under-interface}, we resolve the distinction drawn in~\cite{Velenik} between the area confined under $\mathscr I$, denoted $|\Lambda^-|$, and the size of the connected minus component of $\partial \Lambda^-$, i.e.,~$|\mathcal C^-|$. That $|\mathcal C^-|\leq |\Lambda^-|$ is immediate from set inclusion, and hence the following lower bound  shows that with high probability $|\mathcal C^-|$ and $|\Lambda^-|$ are comparable,, thereby resolving~\cite[Conjecture 1]{Velenik}.  Our proof of comparability goes by  proving separately, that $|\mathcal C^-|$ and $|\Lambda^-|$ are individually of order $N^{4/3}$. However, as emphasized in~\cite{Velenik}, the potentially irregular geometry of $\Lambda^-$ complicates the understanding of the competition between the positive external field and the stability of the minus phase at $\beta>\beta_c$ inside $\Lambda^-$ causing it to be delicate to prove a comparability result directly. Nonetheless, such an approach indeed can be made to work, even conditionally on the interface (this was implemented in version 1 of the paper on arXiv).

\begin{theorem}\label{thm:comparability-whp}
Let $\beta>\beta_c$ and $\lambda = \frac{c_\lambda}{N}$ for $c_\lambda>0$. There exist $C(\beta, c_\lambda), K(\beta)>0$  such that, 
\begin{align}\label{componentlowerbound}
\mu_{\beta,\lambda, N}^\pm \big(|\cC^-|<K^{-1} N^{4/3}\big)\leq C \exp\big( -  N^{1/3}/ C\big).\end{align}
Hence, combined with~\eqref{eq:area-tail-bounds}, there exists $\eta (\beta,c_\lambda)>0$ such that  
\begin{align}\label{comparability1}
\mu_{\beta,\lambda, N}^\pm \big(|\cC^-| \notin [\eta |\Lambda^-|, |\Lambda^-|] \big)\leq \exp\big( -  N^{\frac 13-o(1)}\big).\end{align}
\end{theorem}
Again, the tail in the second display is not quite optimal and inherits the $o(1)$ factor from~\eqref{eq:area-tail-bounds}. 

\subsubsection{Maximum height fluctuations of the interface} 
Finally, we turn to analyzing the asymptotics of the maximum height fluctuations of the interface.  In the fully wetted regime where $\lambda =0$ and the interface is diffusive with order $\sqrt N$ fluctuations, the maximum fluctuation is typically of order~$\sqrt N$, (e.g., combining~\cite[Theorem 5.3]{LMST13} with a simple coupling as explained in the discussion preceding Corollary~\ref{cor:max-height-fluctuation-floor}). 
Since the correlations in the interface decay on scales that are $o(N)$, in the critical pre-wetting regime there is an additional poly-logarithmic pre-factor in the maximum height of the interface relative to the typical height of $N^{1/3}$. Using the sharp tail bounds of Theorem~\ref{thm:one-point-tail-bounds}, we identify the asymptotics of this maximum height to be of order $N^{1/3} (\log N)^{2/3}$.   Sharp poly-logarithmic pre-factors for maximum fluctuations in the context of large subcritical clusters in FK percolation were obtained in \cite{hammond1, hammond2}. 

\begin{theorem}\label{thm:max-tightness}
Let $\beta>\beta_c$ and $\lambda = \frac{c_\lambda}{N}$ for $c_\lambda>0$. There exists $K(\beta, c_\lambda)>0$ such that 
\begin{align*}
    \mu_{\beta,\lambda,N}^\pm \Bigg(\frac{1}{N^{1/3}(\log N)^{2/3}} \max_{x\in \partial_\south \Lambda_N} \hgt_x^+\notin [K^{-1},K]\Bigg) \to 0\qquad \mbox{as $N\to\infty$}\,.
\end{align*}
In particular, $(N^{-1/3}(\log N)^{-2/3} \max_{x\in \partial_\south \Lambda_N} \hgt_x^+)_{N\geq 1}$ is uniformly tight.
\end{theorem}

\subsection{Main ideas in the proofs}\label{sec:ideas-of-proofs}
 In this section we outline the key ideas going into our arguments emphasizing the interplay between inputs from the random-line representation, and more geometric and percolation-theoretic arguments relying on delicate local coupling and resampling arguments.  

\medskip
\noindent
\textbf{Measuring the effect of the tilt at global and local scales.}
Many of the sharp results on the low-temperature 2D Ising model up to $\beta_c$ rely on the random-line representation, which leverages the duality of the model in 2D to relate correlation functions at high temperatures to interfaces at low-temperatures. 
This powerful representation is only valid at $\lambda = 0$ and therefore, following~\cite{Velenik}, the first step to understand the measure $\mu_{\beta,\lambda,N}^\pm$ is to compare it to the no-field measure $\mu_{\beta,0,N}^\pm$. 

Although our focus is ultimately to understand the tilt induced by the external field on the law of the interface at local scales, we begin by proving the following useful improvement of \cite[Lemmas 3--4]{Velenik}, removing polylogarithmic factors in the exponent. 
We give an a priori bound on the distortion any set $\Gamma$ of interfaces face between the with and without-field measures. Whereas the distortion over the measure on Ising configurations may be of order $e^{ \lambda N^2} = e^{ c_\lambda N}$, the distortion over the marginal on the interface is much smaller. 

\medskip
\noindent
\textit{Global estimate (Theorem \ref{thm:effect-of-global-tilt}) :} 
Fix $\beta>\beta_{c}$ and let $\lambda= \frac{c_\lambda}{N}$ for $c_\lambda >0$. There exist
$C_{1}(\beta,c_\lambda)>0$ and $C_{2}(\beta,c_\lambda)>0$ such that for any set of interfaces~$\Gamma$,
\begin{align}\label{eq:sets-interface-tilt-intro}
\mu_{\beta,\lambda,N}^{\pm}(\Gamma)\leq & e^{C_{1}N^{1/3}}\mu_{\beta, 0,N}^{\pm}(\Gamma)+e^{-C_{2}N}\,.
\end{align}
This is proved by the following comparison of the finite-volume surface tensions (formally defined later in Section~\ref{subsec:surface-tension}) which is a refinement of~\cite[Lemma 3]{Velenik}:  
\begin{align}\label{eq:global-tilt-estimate-intro}
\frac{Z_{\pm,\beta,\lambda,N}}{Z_{+,\beta,\lambda,N}}\geq & \frac{Z_{\pm,\beta,0,N}}{Z_{+,\beta, 0,N}}e^{-CN^{1/3}}\,,
\end{align}
Eq.~\eqref{eq:global-tilt-estimate-intro} is proved by considering the contribution to the partition functions from interfaces that concatenate $N^{1/3}$ random walk bridges of length $N^{2/3}$, at their typical $N^{1/3}$ height and away from the bottom boundary, $\partial_\south \Lambda_N$. The re-weighting from the external field, of such interfaces can be shown to be at worst exponential in $\lambda$ times the area they confine, and therefore greater than  $\exp ( - C\lambda  N^{4/3})$}; moreover, the weight of such interfaces under the no-field measure is effectively the probability that a random walk excursion maintains height $O(N^{1/3})$, which is $\exp( - CN^{1/3})$. The fact that these two balance out when $\lambda = \Theta(1/N)$ suggests that the tilt in~\eqref{eq:sets-interface-tilt-intro} is optimal, and $N^{1/3}$ is the typical height of the interface under $\mu_{\beta,\lambda,N}^{\pm}$ as we explore next. 

\medskip
\noindent
\textit{Local estimate:} While the above controls coarse global quantities like the area under the interface, local behavior like the \emph{typical} height, and its upper tail behavior requires a more delicate \emph{local} version of~\eqref{eq:global-tilt-estimate-intro}. Consider a strip of width $L$ in $\Lambda_N$ for $L$ ranging between $N^{2/3}$ and the full $N$; unlike the global situation, the interface enters and exits this strip at some random height, say $H$. For technical reasons related to the lack of the Markov property for $\mathscr I$, and the existence of subcritical droplets all along the boundary of this strip, we move from this strip to a T-shaped enlargement $\tshape$ with its $(\pm,H)$ boundary conditions denoting plus above $H$ and minus below.      

\begin{figure}
\centering
\begin{tikzpicture}
\node at (0,0){
\includegraphics[width = .85\textwidth]{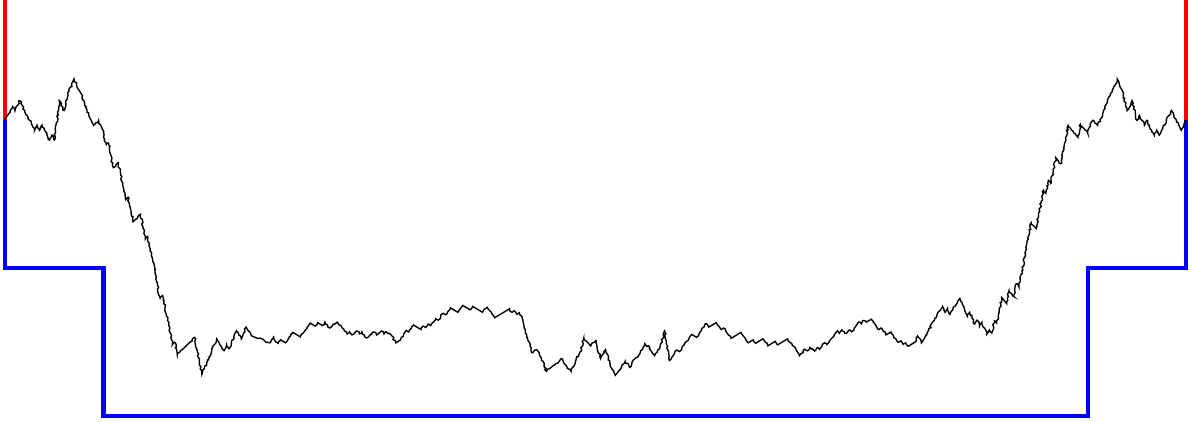}
};
\draw[<->, thick] (-7.1, -2.4)--(-7.1,1);
\draw [<->, thick] (-5.8,-2.55)--(5.8,-2.55); 
\node at (-7.4, -.7) {$H$};

\node at (0, -2.78) {$L$};

\draw[<->, thick] (6, -2.4)--(6,-1.4);

\node at (6.8, -1.9) {$O(N^{1/3})$};

\end{tikzpicture}
\caption{The T-shaped enlargement with its $(\pm,H)$ boundary conditions is the building block of our local estimates. When $H\gg N^{1/3}$, the local tilt on the surface tension from the external field is dominated by interfaces that quickly drop down to $O(N^{1/3})$ heights and do excursions at that height of width $\Theta(N^{2/3})$.}\label{fig:local-tilt-intro}
\end{figure}
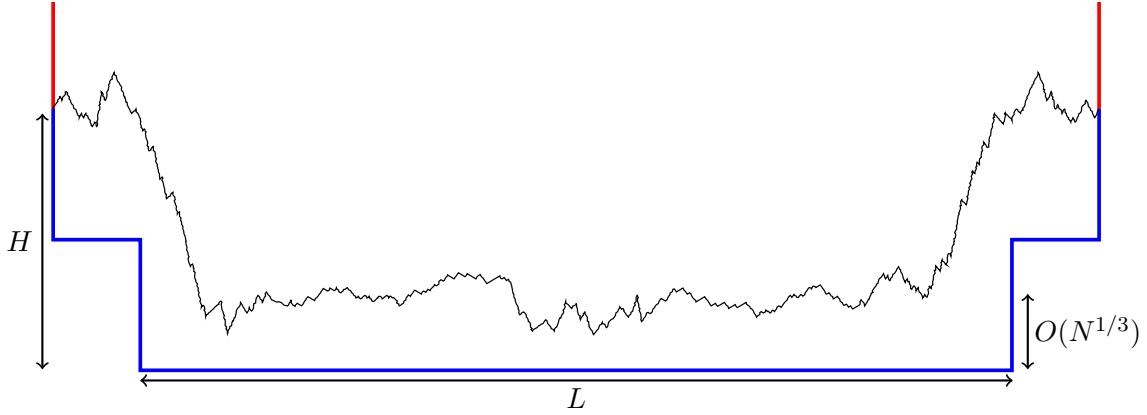

We prove the local analogue of~\eqref{eq:global-tilt-estimate-intro} by considering the contribution to the partition functions from the class of interfaces depicted in Figure~\ref{fig:local-tilt-intro}. This indicates that the predominant contribution to the surface tension under a field comes from interfaces that quickly drop from height $H$ to height $O(N^{1/3})$ then do a series of independent excursions at the critical length scale of $N^{2/3}$. In particular, we deduce the following local analogue of the global tilt estimate, when the boundary conditions enter and exit $\tshape$ at height $H = \Omega(N^{1/3})$:  
 \begin{align}\label{eq:local-tilt-estimate-intro}
\frac{Z_{(\pm,H),\beta,\lambda,\tshape}}{Z_{+,\beta,\lambda,\tshape}} \geq \frac {Z_{(\pm,H),\beta,0,\tshape}}{Z_{+,\beta,0,\tshape}}\exp\Big(-C_{\beta,\lambda}\Big(LN^{-2/3} + \Big({H N^{-1/3}}\Big)^{3/2}\Big)\,.
\end{align}
A similar contribution from such interfaces (without the need for a T-shaped enlargement) was also found in the local behavior of simpler random walks and 1+1 SOS models under area tilts~\cite{ISV15,HrVe04}. 

\medskip
\noindent
\textit{From tilt estimates to area upper bounds.} While~\eqref{eq:sets-interface-tilt-intro} (and its local analogue) give a priori bounds on the tilt induced by $\lambda$ on \emph{any} set of interfaces, they do not capture the fact that the external field further penalizes interfaces that force a large number of minuses. Unlike the approximations to the critical pre-wetting interface, this cost is not exactly an area tilt: instead the tilt is exponential in $\lambda$ times the difference in expected magnetizations under $\mu_{\beta,\lambda,N}^{\pm}(\cdot \mid \Gamma)$ and $\mu_{\beta,0,N}^{\pm}$ (see Eq.~\eqref{eq:part-function-ratio-magnetization-integral}). 

Controlling the difference in expected magnetizations conditionally on a set $\Gamma$ of interfaces that confine area at least $A$ can be delicate (c.f.~\cite[Conjecture 1]{Velenik} and our Theorem~\ref{thm:comparability-whp}). One could in principle handle this using the reasoning developed in~\cite[Section 2.3]{IoffeSchonmann98} for concentration estimates on the magnetization;  for self-containedness we instead appeal to cruder arguments involving  coupling the magnetization to its contribution from  independent tiles on each of which the field effect is negligible. That difficulty notwithstanding, suppose we are in a situation where we can show that if the interface confines area at least $A$, it also has an excess of order $A$ many minus sites.  Then for large $K$, if $A$ is greater than $K N^{4/3}$ in the global case, or $ K (LN^{1/3} + (HN^{-1/3})^{3/2} N)$ in the local case, this cost beats the tilts of~\eqref{eq:global-tilt-estimate-intro}--\eqref{eq:local-tilt-estimate-intro}, and confining area $A$ is exponentially unlikely in $C \lambda A$ as one might expect.

\medskip
\noindent
\textbf{Local fluctuation bounds and consequences.} Bounding the one-point oscillations using the local estimate above is the heart of our paper; we start with a short heuristic for the $e^{ - \Theta(R^{3/2})}$ upper tail. Consider $x\in \llb 0,N\rrb$ and the event $\{\hgt^+_x\ge RN^{1/3}\}$. Let $\Lambda_I = I \times \llb 0,N\rrb$ be the stopping domain, defined by the inner-most interval $I \subset \partial_\south \Lambda_N$ containing $x$ for which the interface enters and exits $\Lambda_I$ below $H= R N^{1/3}/2.$ Let $|I|= K N^{2/3}$ for some $K.$  If we could somehow condition on the location of this interval and resample the configuration inside, the Brownian cost of reaching height $2H  = RN^{1/3}$ above $x$ decays as $e^{-\Theta({R^2}/{K})}$. The fact that $I$ was such a stopping domain, and therefore $\min_{y\in I} \hgt_y^- \geq H$ forces the interface in $\Lambda_I$ to confine area $RKN/2$. As explained in the above discussion, this induces a cost of $e^{ -c_\lambda RK/2}$ times the tilt factor from~\eqref{eq:local-tilt-estimate-intro} of $e^{ C_{\beta,\lambda} R^{3/2}}$. Putting these together, we would find that the probability of $\{\hgt_x^+\geq RN^{1/3}\}$ is at most
\begin{align*}
e^{-\Theta({R^2}/{K})} + e^{ - \Theta(RK)} e^{ \Theta(R^{3/2})}\,;
\end{align*}
an optimal choice of $K= T  \sqrt{R}$ for a sufficiently large $T$ yields the tail behavior of $e^{ - \Theta(R^{3/2})}$.  
 
Making this intuition precise entails handling several difficulties that are inherent to Ising interfaces. We refer to Section~\ref{sec:right-tail-ideas-of-proof} for a sketch of the proof of Theorem~\ref{thm:one-point-tail-bounds}, but mention some of these complications. The key step in formalizing the above heuristic is bounding the probability that $\Lambda_I$ has width larger than $T\sqrt R N^{2/3}$ for a large $T$. There are two difficulties in showing this: 
\begin{enumerate}
\item We would like to condition on the configuration on $\Lambda_N \setminus \Lambda_I$, revealing the entry and exit points of the interface of $\Lambda_I$, and subsequently apply the local tilt estimate of~\eqref{eq:local-tilt-estimate-intro}, and bound the cost of locally confining large area. Conditionally on the configuration of $\Lambda_N \setminus \Lambda_I$, however, the boundary conditions induced on $\Lambda_I$ are not of $\pm$-type, and are random with  many subcritical plus bubbles below and minus bubbles above the interface's entry and exit. Under such complicated boundary conditions, we do not have tools to understand the typical behavior of the interface. 
\item The exponential cost of a fixed interval $I$ of length $L$ being the stopping domain is essentially constant across length-scales of $N^{2/3}$. Therefore, a union bound over all possible stopping domains would induce a polynomial pre-factor in the upper tail. To circumvent this issue, we round the stopping domain to a nearby interval of length that is an integer multiple of $N^{2/3}$, using a priori control of the local regularity of the interface. 
\end{enumerate}

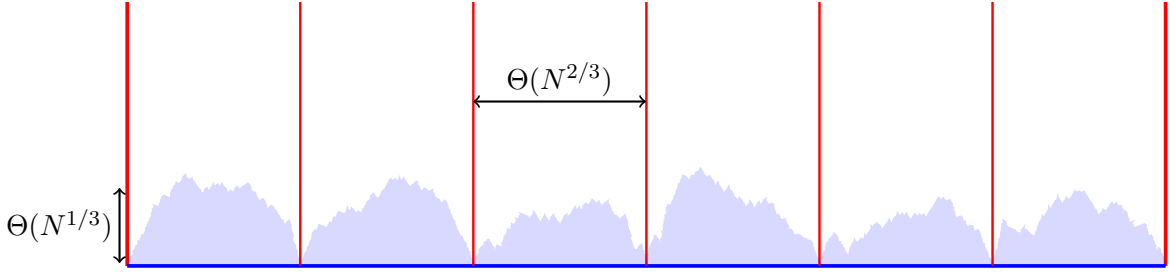
\begin{figure}
\centering
\begin{tikzpicture}
\node at (0,0){\includegraphics[width=.85\textwidth]{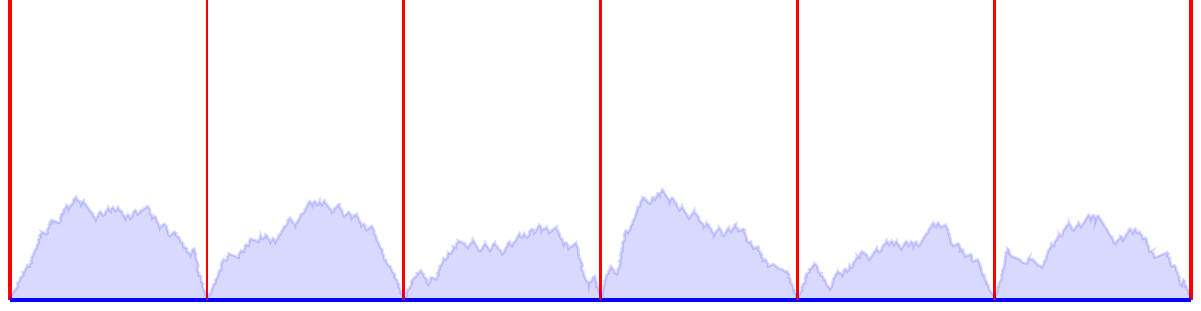}};
\draw[<->, thick]  (-2.3,.5)--(0,.5);
\draw[<->, thick]  (-7,-1.65)--(-7,-.65);

\node at (-1.15,.8) {$\Theta(N^{2/3})$};
\node at (-7.8,-1.15) {$\Theta(N^{1/3})$};
\end{tikzpicture}
\caption{Our lower bounds follow by pushing down the interface via addition of columns of pluses at $N^{2/3}$ spacing; this lower bounds $\mathscr I$ by $N^{1/3}$ many independent copies of interfaces at some smaller scale.}\label{fig:lower-bound-fig}
\end{figure}

\medskip
\noindent
\textbf{Domain enlargements and smoothness estimates.} As is clear from the above discussion, it is crucial to understand the behavior of the interface in a domain $\Lambda_I = I\times \llb 0,N\rrb$, conditionally on its entry and exit data. This conditioning induces a \emph{random} distribution on the boundary conditions of $\Lambda_I$, along with a conditioning on the configuration inside that it does not alter the entry/exit data (by connecting to the sub-critical bubbles above and below the entry/exit). To handle this complicated induced distribution, we construct monotone couplings in both increasing and decreasing directions between the configuration on $\Lambda_I$ conditionally on entry/exit below or above a height $H$, and the configuration induced on $\Lambda_I$ by a small enlargement with $(\pm,H)$ boundary conditions.

For domains with $(\pm,H)$ boundary conditions, we then have access to known bounds on the interface using the random-line representation, using which we can, e.g., prove a priori smoothness estimates controlling the local regularity of $\mathscr I$.  
These couplings and the smoothness estimates are delicate, and the key technical tool for translating arguments that have had success in the context of (Markovian) random walk/height function models to the context of Ising interfaces. We therefore expect that they may be useful in various other applications.

\medskip
\noindent
\textbf{Sharp upper bound on the area under the interface.}
The starting point is the one-point estimate of 
Theorem \ref{thm:one-point-tail-bounds}. 
As elaborated in the discussion on local fluctuation bounds, the height at a point $x$ exceeding $RN^{1/3}$ typically causes its effect to be felt in an interval of width 
$\sqrt{R}N^{2/3}$ around that point. Thus for every $R$, to estimate the size of the set of points  $x\in \llb 0,N\rrb$ such that $\{\hgt^+_x\ge RN^{1/3}\}$,  we decompose $\llb 0,N\rrb$ into $\sqrt{R} N^{2/3}$ translates of a mesh of $R^{-1/2} N^{1/3}$ points $(x_i)$ separated by $\sqrt{R}N^{2/3}$. The idea is that, given such a translate, $(x_i)_i$, the correlation between $\{\hgt^+_{x_i}\ge RN^{1/3}\}$ and $\{\hgt^+_{x_j}\ge RN^{1/3}\}$ decays exponentially in $|j-i|R^{3/2}$; this yields a large deviations type upper bound on the number of points among $(x_i)$ that attain height $RN^{1/3}$; call the set of such points $\mathscr H_R$. As far as we are aware, this is the first such decorrelation estimate reliance on a strong coupling of the interface to an effective random walk.

This concentration bound for $|\mathscr H_R|$ is proved by developing a multi-stopping domain extension of the approach used to prove the one point fluctuation bound.
Partition the mesh points in $\mathscr{H}_R$ according to their stopping domains $I$, on which the interface remains above height $RN^{1/3}/2$. Roughly, if mesh points $x_i,x_j$ belong to the same stopping domain, their heights are correlated, but such a stopping domain is exponentially unlikely in the number of such points it contains; on the other hand, if two points are in different stopping domains, their heights are effectively decorrelated and we can apply multi-height analogues of the domain enlargement machinery (see  Propositions~\ref{prop:multi-strip-minus-enlargement}--\ref{prop:multi-strip-plus-enlargement}). This dichotomy shows that large deviations for the quantity $|\mathscr H_R|$ are exponentially decaying in $R^{3/2} |\mathscr{H}_R|$: i.e., there exists $C(\beta,c_\lambda)>0$ such that for large $R$, 
\begin{align}\label{eq:multi-height-intro}
 \mu_{\beta,\lambda,N}^{\pm} \big( |\mathscr H_R| >  M \big)  \leq C\exp ( - R^{3/2} M /C)
\end{align}
(see Theorem~\ref{thm:multipoint-height-bounds} for a precise statement and Section~\ref{sec:multi-point-oscillations} for an extended sketch of proof.) The upper tail on $|\Lambda^-|$ follows from summing~\eqref{eq:multi-height-intro} over the translates of the mesh, and then summing over $R$.

\medskip
\noindent
\textbf{Lower bounds:}
Lower bounds in this article (i.e., those in Theorems  \ref{thm:one-point-tail-bounds}, \ref{thm:area-under-interface}, and ~\ref{thm:max-tightness}), all follow from a variant of the following observation. By monotonicity the interface in $\Lambda_N$ lies above a collection of $N^{1/3}$ many independent interfaces pinned at the boundaries of strips of width $N^{2/3}$ as shown in Figure~\ref{fig:lower-bound-fig}. Since the effect of the tilt on these critical windows is $O(1)$ (as in~\eqref{eq:local-tilt-estimate-intro} with $L = N^{2/3}$) and in the absence of an external field these independent interfaces behave as independent Brownian excursions, we can translate positive probability events of the Brownian excursion, e.g., lower bounds on their area or height, to the interface in the presence of a field.

\subsection{Organization of the article} We collect all the notations, preliminaries about the Ising model and in particular, a review of the inputs from the random-line representation that we will rely on in Section \ref{sec:preliminaries}. Section \ref{sec:enlargements} is devoted to developing the crucial enlargement couplings; as applications, we obtain local smoothness estimates on $\mathscr I$. In Section~\ref{sec:global-tilt}, we pin down the effect of the tilt at both global and local scales. The proof of our main result Theorem \ref{thm:one-point-tail-bounds} appears in Section \ref{sec:upper-tail}. In Section~\ref{sec:multi-point-oscillations}, the upper tail on multi-height large deviations is proven by generalizing the argument used to prove the one-point upper tail. These are then used in Section~\ref{sec:averaged-statistics} to deduce several global features of the interface, namely, Theorems \ref{thm:area-under-interface}--\ref{thm:max-tightness}. We conclude in Section \ref{sec:proofs-domain-enlargements} by providing deferred proofs from Section \ref{sec:enlargements}.

\subsection{Acknowledgements} The authors thank Yvan Velenik for useful discussions.
S.G.\ is partially supported by NSF grant DMS-1855688, NSF CAREER Award DMS-1945172, and a Sloan Research Fellowship. R.G.\ thanks the Miller Institute for Basic Research in Science for its support. 

\section{Ising interfaces, the random line representation, and other preliminaries}\label{sec:preliminaries}
Before proceeding to the proofs we develop all the necessary notation as well as include all the inputs we will be relying on, to assist the reader not familiar with the many properties of the two-dimensional Ising model that repeatedly get used throughout the rest of the paper.
\subsection{The Ising model}\label{subsec:Ising-model}
Fix an underlying graph $G= (V,E)$. For a finite subgraph $\Lambda = (\Lambda, E_\Lambda)$ with $\Lambda \subset V$, the Ising model on $\Lambda$ at \emph{inverse temperature} $\beta$, with external field $\lambda$, and \emph{boundary conditions} $\eta \in \{\pm 1\}^{V\setminus \Lambda}$ is the probability measure on configurations $\sigma \in \{\pm 1\}^\Lambda$ defined as 
\begin{align}\label{eq:Ising-measure} 
    \mu_{\beta,\lambda,\Lambda}^\eta(\sigma)= \frac 1{Z_{\eta,\beta,\lambda,\Lambda}} \exp\Bigg(\!\!-\!2\beta\Big( \sum_{\substack{v,w\in \Lambda: \\ (v,w)\in E_\Lambda}} \mathbf 1\{\sigma_v\neq \sigma_w\} + \sum_{\substack{v\in \Lambda, w\in V\setminus \Lambda: \\ (v,w)\in E}} \mathbf 1\{\sigma_v\neq \eta_w\}\Big)+ \lambda \sum_{v\in \Lambda} \sigma_v\Bigg)
\end{align}
where $Z_{\eta,\beta,\lambda,\Lambda}$ is the \emph{partition function}, i.e., the normalizing constant such that $\mu$ is a probability distribution. For an event $A$, define the restricted partition function $Z_{\eta,\beta,\lambda,\Lambda} (A) =  Z_{\eta,\beta,\lambda,\Lambda}. \mu_{\beta,\lambda,\Lambda}^\eta(A)$. We have hidden the implicit dependence on the underlying geometry $G$,  and in the context of this paper, we think fix $G$ to be the integer lattice $\mathbb Z^2$, as defined below.

\subsubsection{Geometry of $\mathbb Z^2$}
Throughout, we restrict attention to the Ising model in two-dimensions, i.e., on subgraphs $\Lambda$ of the integer lattice  $\mathbb Z^2 = (\mathbb Z^2, E_{\mathbb Z^2})$ whose vertex set is $\mathbb Z^2$, and edge set $E_{\mathbb Z^2}$ is the set of nearest-neighbor pairs of vertices. In $\mathbb Z^2$, we thus say that two vertices are adjacent if they are the end-points of an edge, and we say that two edges are adjacent if they share a vertex. 

 For two distinct vertices $v\neq w\in \mathbb Z^2$, a \emph{path} between $v,w$ is for some $k$, there exists a sequence of distinct edges $e_0,...,e_k$ of $E_{\mathbb Z^2}$ such that $v\in e_0$, $w\in e_k$, and $e_{i}$ is adjacent to $e_{i+1}$ for every $i=0,...,k-1$. A \emph{loop} is defined as in a path, but from $v$ to itself so that $v\in e_0$ and $v\in e_k$. For a set of edges $F\subset E_{\mathbb Z^2}$, a \emph{path in $F$} (resp., \emph{loop in $F$}) is a path (loop) all of whose constituent edges are in $F$. We say a set of edges $F\subset E_{\mathbb Z^2}$ is connected if every for every two vertices incident an edge in $F$, there exists path in $F$ between them.  
 
A set of vertices $V\subset \mathbb Z^2$ is said to be \emph{connected} if its induced edge-set $E_V \subset E_{\mathbb Z^2}$ (edges having both endpoints in $V$) is connected. For a set of vertices $V$, the (maximal) connected components of $V$ are the maximal (under the inclusion relation) subsets of $V$ that are connected. For a subgraph $\Lambda \subset \mathbb Z^2$ and a $\sigma\in \{\pm 1\}^\Lambda$, for each vertex $v\in \Lambda$, we define its \emph{plus cluster in $\sigma$}, denoted $\mathcal C_v^+$, to be the (possibly empty) connected component of $\{v\in \Lambda: \sigma_v = +1\}$ that contains $v$. We define the \emph{minus cluster} of $v$, denoted $\mathcal C_v^-$, analogously. For a subset $A\subset \mathbb Z^2$, define its \emph{plus (resp.\ minus) cluster in $\sigma$}, denoted $\mathcal C_A^+$ (resp.\ $\mathcal C_A^-$) to be the union of plus (resp.\ minus) clusters of $v\in A$.

\subsubsection{Subgraphs of $\mathbb Z^2$}
We will consider the Ising model on subgraphs $\Lambda$ of $\mathbb Z^2$ with edge-set $E_\Lambda$ consisting of all edges of $E_{\mathbb Z^2}$ both of whose endpoints are in $\Lambda$. We are primarily interested in rectangular subgraphs of $\mathbb Z^2$, for which we use the unified notation 
\begin{align*}
\Lambda_{n,m} = \llb 0,n\rrb \times \llb 0,m\rrb = \{0,...,n\}\times \{0,...,m\}\,.
\end{align*} 
We use the shorthand $\Lambda_n = \Lambda_{n,n}$. 
For a subgraph $\Lambda\subset \mathbb Z^2$, its (outer) boundary, denoted $\partial \Lambda$ is the set of all vertices in $\mathbb Z^2 \setminus \Lambda$ that are adjacent to some site in $\Lambda$. For a rectangular subset $R= \llb x_1, x_2 \rrb \times \llb y_1,y_2\rrb$, we partition its (outer) boundary into its west boundary $\partial_\west R = \{x_1-1\}\times \llb y_1,y_2\rrb$, and its north, east, and south boundaries defined analogously and denoted $\partial_\north R, \partial_\east R, \partial_\south R$. We occasionally use multiple subscripts, e.g., $\partial_{\north,\east,\west} R$ to denote the union of $\partial_\north R,\partial_\east R, \partial_\west R$. Likewise, for a rectangular subset $R$, define its inner boundary $\partial_\interior R$ as the set of vertices in $R$ adjacent to vertices of $\mathbb Z^2 \setminus R$, and partition that into $\partial_{\interior,\west} R, \partial_{\interior,\north} R, \partial_{\interior,\east} R, \partial_{\interior,\south} R$. 

\subsubsection{Boundary conditions on $\Lambda$} The simplest boundary conditions on subgraphs $\Lambda$ are those that are all-plus or all-minus on $\mathbb Z^2 \setminus \Lambda$, which we denote by $\eta = +$ or $\eta = -$. In this paper, we are primarily interested in the behavior of the low-temperature Ising model under \emph{Dobrushin} boundary conditions, which are $+$ on all sites with non-negative $y$-coordinate (i.e., on the rectangle $\Lambda_{n,m}$, on $\partial_{\north,\east,\west} \Lambda_{n,m}$) and $-$ on all vertices with negative $y$-coordinate (i.e., on $\partial_\south \Lambda_{n,m}$). We denote these boundary conditions by $\eta = \pm$. More generally, we consider boundary conditions that we say are of ``$\pm$-type" if for some line in $\mathbb R^2$, the boundary conditions are $+$ above that line, and $-$ below; if that line is the horizontal line at height $h$, we denote it by $\eta = (\pm,h)$. Note from~\eqref{eq:Ising-measure} that it suffices to prescribe the restriction of $\eta$ to  $\partial \Lambda$, as the rest of $\eta$ does not affect the measure. 

\subsubsection{Domain Markov property} 
A fundamental property of the Ising model is a Markovian property known as the \emph{domain Markov property}. The domain Markov property says that for every two graphs $\Lambda \subset \Lambda'$,   for every configuration $\eta_{\Lambda'\setminus \Lambda} \in\{\pm 1\}^{\Lambda'\setminus \Lambda}$, every boundary condition 
\begin{align*}
\mu_{\beta,\lambda,\Lambda'} (\sigma_{\Lambda}\in \cdot \mid \sigma_{\Lambda'\setminus \Lambda} = \eta_{\Lambda'\setminus \Lambda}) = \mu_{\beta,\lambda,\Lambda'} (\sigma_{\Lambda}\in \cdot \mid \sigma_{\partial \Lambda} (\Lambda'\setminus \Lambda) = \eta_{\partial \Lambda})  = \mu_{\beta,\lambda,\Lambda}^{\eta_{\partial \Lambda}} (\sigma \in \cdot)\,,
\end{align*}
where for a set $A$, $\sigma_A$ denotes the restriction of the configuration $\sigma$ to $A$. 
This allows us to express Ising probabilities conditionally on some spin configuration as just a different Ising model with boundary conditions prescribed by the configuration we conditioned on; moreover, any boundary conditions on $\Lambda$ only influence the measure through their restriction to~$\partial \Lambda$.

\subsubsection{Infinite volume Gibbs measures} One can extend the Ising measure~\eqref{eq:Ising-measure} from finite subgraphs of $\mathbb Z^2$ to infinite subgraphs of $\mathbb Z^2$. There, the partition functions above are of course infinite, but infinite-volume measures can be formalized by a consistency criterion known as the DLR condition: for an infinite subgraph $\Lambda_\infty \subset \mathbb Z^2$, a measure $\mu_{\Lambda_\infty}$ on $\{\pm 1\}^{\Lambda_\infty}$ is called a DLR measure, if for every finite subset $\Lambda \subset \Lambda_{\infty}$, we have 
\begin{align*}
\mu_{\beta, \lambda, \Lambda_\infty} ( \sigma_\Lambda \in \cdot )  = \mathbb E_{\mu_{\beta, \lambda, \Lambda_\infty}(\eta_{\Lambda_\infty\setminus \Lambda} \in \cdot)} \big[\mu_{\beta, \lambda, \Lambda}^{\eta_{\Lambda_\infty \setminus \Lambda}} (\sigma_\Lambda \in \cdot)\big] = \mathbb E_{\mu_{\beta, \lambda, \Lambda_\infty}(\eta_{\partial \Lambda} \in \cdot)} \big[\mu_{\beta, \lambda, \Lambda}^{\eta_{\partial \Lambda}} (\sigma_\Lambda \in \cdot)\big]
\end{align*}
where the expectation is with respect to the measure $\mu_{\beta,\lambda,\Lambda_{\infty}}$ restricted to $\Lambda_\infty\setminus \Lambda$.  

\subsubsection{Phase transition on $\mathbb Z^2$:} The famous phase transition of the 2D Ising model can be expressed both in terms of the multiplicity of infinite-volume Gibbs measures on $\mathbb Z^2$ for a given $(\beta,\lambda)$ pair. 

More precisely, it is a famous result of Onsager that when $\lambda = 0$, the Ising model undergoes the following sharp phase transition at inverse temperature $\beta_c := \log (1+\sqrt 2)$, between the regimes (a) $\beta \leq \beta_c$ (high and critical temperatures): there exists a \emph{unique} infinite-volume Gibbs measure $\mu_{\beta, \mathbb Z^2}$ on $\mathbb Z^2$ and (b) $\beta>\beta_c$ (low temperature): there exist multiple \emph{distinct} infinite-volume Gibbs measures---the extremal ones $\mu_{\beta,\mathbb Z^2}^+$ and $\mu_{\beta,\mathbb Z^2}^-$ are obtained as the limits of Ising measures on finite boxes with all-plus, resp., all-minus boundary conditions (see e.g.,~\cite{McCoyWu,GrimmettRC}).  

This infinite-volume phase transition also exhibits itself in finite connectivities/correlations. This is best captured by the presence/absence of correlation decay. At high-temperatures (when $\beta<\beta_c$, and $\lambda = 0$) the Ising model undergoes \emph{exponential decay of correlations}: there exists $C(\beta)>0$ such that for every $n$, every two $v,w\in \mathbb Z^2$, we have   
\begin{align}\label{eq:exp-decay-correlations}
\langle \sigma_v \sigma_w \rangle_{\beta,0,\mathbb Z^2} := \mathbb E_{\mu_{\beta,0,\mathbb Z^2}} [\sigma_v \sigma_w ] \leq C \exp (  -  |v-w|/C)
\end{align}
We pause here to mention that by the GKS inequality (see e.g.,~\cite{McCoyWu}), for every $\beta$ and every pair of edge-sets $F \supset E$ and every pair of vertices $v,w$, we have 
\begin{align}\label{eq:two-point-function-monotonicity}
\langle\sigma_v \sigma_w\rangle_{\beta, 0, F}\geq \langle \sigma_v\sigma_w\rangle_{\beta,0,E}
\end{align} 
from which it follows that~\eqref{eq:exp-decay-correlations} also holds on finite subsets of $\mathbb Z^2$ containing $v,w$. 

At low-temperatures $\beta>\beta_c$, on the other hand, the two-point functions above are uniformly bounded away from zero under, i.e., $\langle \sigma_v, \sigma_w\rangle>0$ under each of $\mu_{\beta,\mathbb Z^2}^+$ and $\mu_{\beta,\mathbb Z^2}^-$. 
All the same, a high-temperature structure emerges by duality at low-temperatures, inducing an exponential decay of \emph{truncated correlations}: if we denote the ball $B_r(v) = \{w: d(v,w) \leq r\}$, then for every $\beta>\beta_c$,  there exists $C(\beta)>0$ such that for every $v\in \mathbb Z^2$,   
\begin{align}\label{eq:exp-decay-truncated-correlations}
\mu_{\beta, 0,\mathbb Z^2}^+(  \mathcal C_v^- \cap \partial B_r(v) \neq \emptyset ) \leq Ce^{ - r/C} \quad \mbox{and}\quad \mu_{\beta, 0, \mathbb Z^2}^-(  \mathcal C_v^+\cap \partial B_r(v)\neq \emptyset ) \leq Ce^{ - r/C}\,.
\end{align}

By~\eqref{eq:exp-decay-correlations}, when $\beta<\beta_c$, the block magnetization on the square $\Lambda_n$ (i.e., $\frac{1}{|\Lambda_{n}|} \sum_{v\in \Lambda_{n}} \sigma_v$) concentrates around $0$ with fluctuations that are or order $n^{-1}=|\Lambda_{n}|^{-1/2}$ and Gaussian tail bounds. In the low-temperature regime, from~\eqref{eq:exp-decay-truncated-correlations}, the block magnetization under, say, $\mu_{\beta,\mathbb Z^2}^-$ concentrates about $-m_{\star,\beta}$ for some $m_{\star,\beta} >0$. While the concentration of the left-tail is Gaussian, the probability that the block magnetization is greater than some $-m > -m_\star$, is only exponentially concentrated. These \emph{surface order large deviations} were proven in the work of~\cite{Schonmann87, CCS}: for every $\varepsilon>0$, there exists $C(\beta,\varepsilon)>0$ such that 
\begin{align}\label{eq:surface-order-ld-magnetization}
\mu_{\beta,0,\mathbb Z^2}^- \Big(\frac{1}{|\Lambda_{n}|} \sum_{v\in \Lambda_{n}} \sigma_v  > - m_\star + \varepsilon \Big) \leq C  \exp ( - n/C)\,.
\end{align}
We will also use an alternate form of these surface-order large deviations, expressed in terms of the number of sites in $\Lambda$ which are connected to $\infty$ by minuses. This follows closely from~\cite{Schonmann87,CCS} or immediately by applying the related surface order large deviations of~\cite{Pisztora1996,Couronne}. Let $\mathcal C_\infty^-$ be the minus cluster of $\infty$ (i.e., the set of all infinite connected components of minus sites).  There exists $\hat m_{\star,\beta}>0$ such that for every $\varepsilon>0$, there exists $C(\varepsilon, \beta)>0$ such that 
\begin{align}\label{eq:surface-order-ld-component-of-infinity}
\mu_{\beta,0,\mathbb Z^2}^- \Big( \frac{1}{|\Lambda_{n}|} \sum_{v\in \Lambda_{n,n}} \mathbf 1\{v\in \mathcal C_\infty^-\} <   \hat m_\star -  \varepsilon \Big) \leq C\exp ( - n/C)\,.
\end{align}

\subsubsection{FKG inequality}
 One of the most important properties of the Ising model is that the model is positively correlated. Namely, if $\sigma \geq \sigma'$ denotes the natural partial order on configurations, then for every graph $G= (V,E)$, and every $\beta, \lambda$, we have that for every two increasing functions (in the p.o.\ $\sigma \geq \sigma'$) $f,g:\{\pm1\}^V \to \mathbb R$, 
\begin{align*}
\mathbb E_{\mu_{\beta, \lambda,G}} [f(\sigma) g(\sigma)] \geq  \mathbb E_{\mu_{\beta,\lambda, G}} [ f(\sigma)] \mathbb E_{\mu_{\beta,\lambda, G}} [ g(\sigma)]\,.
\end{align*}  
Of course taking $f,g$ to be indicators of increasing events, we see that for every two increasing events $A,B$, we have $\mu_{\beta,\lambda,G}(A\cap B)\geq\mu_{\beta,\lambda,G}(A)\mu_{\beta,\lambda,G}(B)$. 

From the FKG inequality,  we deduce the following monotonicity in boundary conditions: consider two boundary conditions $\eta$ and $\eta'$ on $G$ such that $\eta' \geq \eta$: then, for every increasing function $f$,  
\begin{align}\label{eq:monotonicity-in-bc}
\mathbb E_{\mu_{\beta,\lambda,G}^{\eta'}} [ f(\sigma)] \geq \mathbb E_{\mu_{\beta,\lambda,G}^{\eta}}[f(\sigma)]\,.
\end{align}

\subsubsection{Crossing and circuit events}
Some examples of monotone events that we consider in this paper are connectivity events by pluses and minuses. Recall the definitions of the plus cluster and minus cluster of a set $A$ through a configuration $\sigma$, denoted $\mathcal C_A^+$ and $\mathcal C_A^-$.  Observe, first of all, that the map $\sigma \mapsto \mathcal C_A^+$ is increasing (as sets) in $\sigma$, and $\sigma \mapsto \mathcal C_A^-$ is decreasing in $\sigma$. Then, for a general rectangle $R = \llb x_1, x_2 \rrb \times \llb y_1, y_2 \rrb$, we define the crossing events on $\{\pm 1\}^R$,
\begin{equation}\label{eq:crossing-events}
\begin{aligned}
\co_v^+(R) & :=  \{ \sigma: \mathcal C^+_{\partial_{\interior,\south} R} \cap \partial_{\interior,\north} R \neq \emptyset\}\quad \qquad & \co_v^-(R) :=  \{\sigma: \mathcal C^-_{\partial_{\interior,\south} R} \cap \partial_{\interior,\north} R \neq \emptyset\} \\
\co_h^+(R) & :=  \{ \sigma: \mathcal C^+_{\partial_{\interior,\west} R} \cap \partial_{\interior,\east} R \neq \emptyset\}\quad \qquad & \co_h^-(R) :=  \{\sigma: \mathcal C^-_{\partial_{\interior,\west} R} \cap \partial_{\interior,\east} R \neq \emptyset\} 
\end{aligned}
\end{equation}
i.e., that there is a vertical or horizontal, $+$ or $-$ crossing of $R$. These are easily seen to be increasing events in the case where of $+$ superscript, and decreasing events in the case of $-$ superscript. It is an easy consequence of~\eqref{eq:exp-decay-truncated-correlations} and a union bound that for there exists $C(\beta)>0$ such that 
\begin{equation}\label{eq:crossing-probability-bound}
\begin{aligned}
 \max_{\eta \in \{+,-\}} \mu_{\beta,0,\mathbb Z^2}^\eta \big((\co_v^\eta(R))^c\big)& \leq C |x_2 - x_1| \exp(- |y_2 - y_1|/C)\qquad \mbox{and similarly}  \\
 \max_{\eta \in \{+,-\}} \mu_{\beta,0,\mathbb Z^2}^\eta \big((\co_h^\eta(R))^c\big) & \leq C |y_2 - y_1| \exp(- |x_2 - x_1|/C).
 \end{aligned}
\end{equation}
For two concentric rectangles $R \subset \Lambda$, we can define a circuit event, $\circuit^+(\Lambda \setminus R)$, as the set of $\sigma$ such that there is a plus loop of non-trivial homology in the annulus $\Lambda \setminus R$. Viewing the existence of such a loop as a subset of the intersection of four crossing events, and applying the FKG inequality, 
\begin{align}\label{eq:circuit-probability-bound}
\max_{\eta\in \{+,-\}} \mu_{\beta,0, \mathbb Z^2}^\eta \big((\circuit^\eta(\Lambda\setminus R))^c \big) \leq C |\partial R| \exp ( - d(R,\Lambda)/C)\,.
\end{align}

\subsubsection{Role of the external field}
A large portion of the present paper deals with controlling the tilt induced on the Ising measure by the presence of a global external field.  
When $\lambda>0$, the phase transition in $\mathbb Z^2$ is destroyed and for every $\beta>0$, one has $\mu_{\beta,\lambda,\mathbb Z^2}(\sigma_{(0,0)} = +1) >\frac 12$. Therefore, in order to get interesting intermediate competition between a low-temperature $-$ phase and a $+$ external field, we study the Ising model on $\Lambda_N$ with $\lambda$ going to zero as a function of $N$, and in particular at the critical speed of $\lambda = \frac{c_\lambda}{N}$, where interesting transitions occur both in the bulk magnetization, and along an interface~\cite{Martirosyan,ScShCMP,ScShJSP,Velenik}: see also the surveys~\cite{VelenikSurvey,IoVe16}.  

Bounding the effect of the field in the regime where $\lambda = \frac{c_\lambda}{N}$ is one of the key steps in our analysis. It relies on carefully bounding the tilt induced on the partition function from the external field, using the random-line representation defined in Section~\ref{subsec:random-line-representation}. Before getting to that, however, we mention two elementary bounds on the effect of the field that will prove useful. 
A simple consequence of the monotonicity of the Ising model is that increasing the external field, in turn increases the probability of increasing events: for $\lambda >\lambda'$, for every increasing function $f$, 
\begin{align}\label{eq:monotonicity-in-external-field}
\mathbb E_{\mu_{\beta,\lambda,\Lambda}^{\eta}} [f(\sigma)] \geq \mathbb E_{\mu_{\beta,\lambda',\Lambda}^{\eta}}[f(\sigma)]\,.
\end{align}
At other times, when the direction of the bounds we desire is opposite the monotonicity above, it will be useful to have a priori control on how much the presence of the external fields can change probabilities: by computing the Radon--Nikodym derivative, it is evident that for every event $A$, 
\begin{align}\label{eq:radon-nikodym-bound}
e^{-2 \lambda |\Lambda|} \mu_{\beta ,0,\Lambda}^\eta(A) \leq \mu_{\beta,\lambda,\Lambda}^\eta (A) \leq e^{2 \lambda |\Lambda|} \mu_{\beta,0,\Lambda}^\eta(A)\,.
\end{align}

\subsection{The surface tension}\label{subsec:surface-tension}
A characteristic feature of the low-temperature Ising model is the existence of a \emph{surface tension} between the two phases (predominantly plus, and predominantly minus). 
One way to investigate the coexistence of phases is to force it via boundary conditions of $\pm$-type, in terms of which we can define the surface tension. 
Roughly speaking, the surface tension captures the influence of the presence of an interface between the plus and minus phases, on the free energy, and is an important tool for understanding the behavior of the interface. 

Let $\Lambda_{N,\infty}$ be the infinite strip $\llb 0,N\rrb \times \llb -\infty, \infty \rrb$. For $\theta \in [0,\frac{\pi}2)$, define $\tau_{\beta}(\theta),$ the \emph{surface tension at angle $\theta$} as follows: let $y_{\theta,N} = \lfloor N\tan \theta\rfloor$ and let the $(\pm,\theta)$ boundary conditions be those that are plus above the line segment connecting $(0,0)$ to $(N,y_{\theta,N})$ and are minus on and below. Then 
\begin{align}\label{eq:surface-tension-def}
\tau_\beta(\theta): = - \lim_{N\to\infty} \frac 1{N\sqrt{1 +(\tan \theta)^2}} \log \frac{Z_{(\pm,\theta),\beta,\Lambda_{N,\infty}}}{Z_{+,\beta, \Lambda_{N,\infty}}}\,.
\end{align}
(While the partition functions on the right-hand side are infinite sums, viewing the ratio as the $M\to\infty$ limit of ratios of $Z_{\eta,\beta, \Lambda_{N,M}}$, one straightforwardly sees that the ratio is well-defined for all fixed~$N$: see e.g., \cite{DKS,PV99}.) For every $\beta>\beta_c$ and every $\theta \in [0,\frac{\pi}2)$, the above defined surface tension is strictly positive. By homogeneity, we can extend the definition of the surface tension to a (not-normalized) vector $(y-x)$ at angle $\varphi$,  by letting $\tau_\beta (y-x)=|y-x|\tau_\beta (\varphi)= |y-x| \tau_\beta (\tfrac{y-x}{|y-x|})$. 

\subsubsection{Duality, surface tension, and two-point functions}
In $\mathbb Z^2$, the Ising model has an important duality property.  Let $(\mathbb Z^2)^*$ be the planar dual of $\mathbb Z^2$; observe that $(\mathbb Z^2)^* = \mathbb Z^2 + (\frac 12, \frac 12)$, showing that $\mathbb Z^2$ is isomorphic to its dual. An edge $e^*$ of $E_{(\mathbb Z^2)^*}$ intersects exactly one edge $e$ of $E_{\mathbb Z^2}$: if the primal edge $e^*$ intersects is $e= (v,w)$, we say $e^*$ \emph{separates} $v,w$.

Similarly, for a subgraph $\Lambda = (\Lambda, E_\Lambda)$ of $\mathbb Z^2$, we let $\Lambda^* = (\Lambda^*, E_{\Lambda^*})$ be its dual graph, i.e., the subgraph of $(\mathbb Z^2)^*$ whose vertex-set $\Lambda^*$ consists of sites in $(\mathbb Z^2)^*$ at $\ell^1$-distance one from $\Lambda$, and whose edge-set is induced by $E_{(\mathbb Z^2)^*}$ (i.e., all edges in $E_{(\mathbb Z^2)^*}$ both of whose endpoints are in $\Lambda^*$). 

For every $\beta$, define its dual inverse temperature $\beta^*$ by 
    \begin{align*}
    \tanh (\beta^*) = e^{ - 2\beta}\,.
    \end{align*}
    Call the fix-point of this duality relation $\beta_{\textsc{sd}}$ and observe that the critical point is the self-dual point $\beta_c = \beta_{\textsc {sd}}$. In particular, for every $\beta>\beta_c$, we have $\beta^*<\beta_c$. 
    
    There is a well-known relationship between the surface tension at low temperature $\beta>\beta_c$ and the exponential decay rate of the two-point function at the dual temperature $\beta^*$. This is captured by the identity (see e.g.,~\cite[Proposition 2.2]{PV99})
    \begin{align}\label{eq:duality-surface-tension}
     \tau_\beta(\theta) = \hat \tau_{\beta^*}(\theta)\,, \qquad \mbox{where}\qquad \hat \tau_{\beta^*}(\theta) : = - \lim_{N\to\infty}  \frac{1}{N} \langle \sigma_{(-\frac 12 ,-\frac 12 )} \sigma_{(N + \frac 12 ,\lfloor N\tan \theta \rfloor - \frac 12)}\rangle_{\beta^*,\Lambda_N^*}\,.
    \end{align} 
The nature of this duality relation is perhaps best elucidated through a combinatorial framework connecting the \emph{contour representation} of the low-temperature Ising model to the \emph{random-line representation} of the high-temperature model; this contour representation is also needed to formally define the interface induced by $\pm$ boundary conditions.

\subsection{The Ising interface}\label{subsec:Ising-interface}In this section, we describe the contour representation of Ising configurations, and then formally define the interface $\mathscr I$ under $\pm$-type boundary conditions on~$\Lambda$.

\subsubsection{Separating edges of an Ising configuration} 
We describe a polymer representation of the Ising model, which allows us to view the \emph{low-temperature} Ising model as a ``high-temperature" model over \emph{contours}. For an Ising configuration $\sigma$ on $\mathbb Z^2$ (or a subset $\Lambda$ of $\mathbb Z^2$), we can define a set of \emph{separating edges} $E(\sigma)$ to be the set of all edges $e\in E_{(\mathbb Z^2)^*}$ separating vertices of $\mathbb Z^2$ assigned different spin values under $\sigma$. Observe that for every $\sigma\in \{\pm 1\}^{\mathbb Z^2}$, the set $E(\sigma)$ is an edge subset of $E_{(\mathbb Z^2)^*}$ with the constraint that every vertex of $(\mathbb Z^2)^*$ is incident an even number of edges in $E(\sigma)$. 

For a subgraph $\Lambda$ of $\mathbb Z^2$ with some prescribed boundary conditions $\eta\in \{\pm 1\}^{\mathbb Z^2 \setminus \Lambda}$, given a configuration $\sigma \in \{\pm 1\}^{\Lambda}$, construct the set $E^{\eta}_\Lambda(\sigma)$ to be the restriction of $E(\sigma \amalg \eta)$ to $E_{\Lambda^*}$ (where $\sigma \amalg \eta$ is the configuration on $\mathbb Z^2$ given by $\sigma$ on $\Lambda$ and $\eta$ elsewhere). For every boundary condition $\eta$, for every $\sigma \in \{\pm 1\}^{\Lambda}$, the separating set $E^{\eta}_{\Lambda}(\sigma)$ is such that every vertex of $\Lambda^*$ is incident to an even number of edges in $E^\eta_{\Lambda}(\sigma)$ except those sites in $\partial_\interior \Lambda^*$ incident to an edge of $E(\eta)$---these are incident to an odd number of edges. It is clear that this set of boundary vertices are determined by $\eta$ (do not depend on $\sigma$) and we call them \emph{source points of $\eta$}, denoted $\partial \eta$.  In the case of $\eta \in \{+,-\}$ it has no source points ($\partial \eta = \emptyset$); if $\eta$ is of $\pm$-type, $\partial \eta$ is exactly the two sites in $\Lambda^*$ incident to the dual edges separating the plus boundary spins from the minus boundary spins. 

\subsubsection{Contour decomposition}
For a subset $F\subset E_{(\mathbb Z^2)^*}$, we wish to decompose it into a sequence of paths between the vertices incident to an odd number of edges of $F$, and a collection of loops. A priori, such a decomposition is ill-defined due to vertices incident to three or four edges of $F$. To remedy this ambiguity, we introduce a \emph{splitting convention} according to which we partition the edges of $F$. For concreteness we use the south-east ($\south\east$) splitting rule and partition $F$ into loops and paths as follows: for every vertex $v\in (\mathbb Z^2)^*$ incident to three or four edges, if two of those edges are the edges directly south and east of $v$, \emph{split} that pair apart from the remaining edge(s), and analogously if two of the edges are those directly north and west of $v$, split that pair apart from the remaining edge(s). Partition $F$ into $\{\gamma_1 ,\gamma_2 ,...,\gamma_k\}$ by taking the (maximal) connected components of $F$, where edges are connected to one another if they share a vertex and they have not been split at that vertex. The resulting partition $\underline \gamma := \{\gamma_1,\gamma_2,...,\gamma_k\}$ consists only of edge-disjoint (simple) paths and loops, and is called the \emph{contour decomposition}. This procedure is depicted in Figure~\ref{fig:contour-decomposition}. 

\begin{figure}
\begin{tikzpicture}[scale = .75]

\tikzstyle{plus-site}=[fill=red, opacity =.3, shape=circle, font = \tiny]
\tikzstyle{minus-site}=[fill=blue, opacity =.3, shape=circle, font = \tiny]

\node[style=plus-site] (0,0) at (0,0) {$+$};
\node[style=plus-site] (0,1) at (0,1){$+$};
\node[style=minus-site] (0,2) at (0,2) {$-$}; 
\node[style = minus-site] (0,3) at (0,3) {$-$};
\node[style=minus-site] (1,0) at (1,0) {$-$};
\node[style=plus-site] (1,1) at (1,1){$+$};
\node[style=minus-site] (1,2) at (1,2) {$-$}; 
\node[style = plus-site] (1,3) at (1,3) {$+$}; 

\node[style=minus-site] (2,0) at (2,0) {$-$};
\node[style=minus-site] (2,1) at (2,1){$-$};
\node[style=plus-site] (2,2) at (2,2) {$+$}; 
\node[style = plus-site] (2,3) at (2,3) {$+$};
\node[style=minus-site] (3,0) at (3,0) {$-$};
\node[style=plus-site] (3,1) at (3,1){$+$};
\node[style=plus-site] (3,2) at (3,2) {$+$}; 
\node[style = plus-site] (3,3) at (3,3) {$+$};

\node[style=plus-site] (5,0) at (5,0) {$+$};
\node[style=plus-site] (5,1) at (5,1){$+$};
\node[style=minus-site] (5,2) at (5,2) {$-$}; 
\node[style = minus-site] (5,3) at (5,3) {$-$};
\node[style=minus-site] (6,0) at (6,0) {$-$};
\node[style=plus-site] (6,1) at (6,1){$+$};
\node[style=minus-site] (6,2) at (6,2) {$-$}; 
\node[style = plus-site] (6,3) at (6,3) {$+$}; 

\node[style=minus-site] (7,0) at (7,0) {$-$};
\node[style=minus-site] (7,1) at (7,1){$-$};
\node[style=plus-site] (7,2) at (7,2) {$+$}; 
\node[style = plus-site] (7,3) at (7,3) {$+$};
\node[style=minus-site] (8,0) at (8,0) {$-$};
\node[style=plus-site] (8,1) at (8,1){$+$};
\node[style=plus-site] (8,2) at (8,2) {$+$}; 
\node[style = plus-site] (8,3) at (8,3) {$+$}; 

\draw[line width = 1pt, black] (5+-.5,1.5)--(5+1.5,1.5)--(5+1.5,2.5)--(5+.5,2.5)--(5+.5,3.5);
\draw[line width = 1pt, black] (5+.5,-.5)--(5+.5,.5)--(5+1.5,.5)--(5+1.5,1.5)--(5+2.5,1.5)--(5+2.5,.5)--(5+3.5,.5);

\node[style=plus-site] (10,0) at (10,0) {$+$};
\node[style=plus-site] (10,1) at (10,1){$+$};
\node[style=minus-site] (5+5,2) at (5+5,2) {$-$}; 
\node[style = minus-site] (5+5,3) at (5+5,3) {$-$};
\node[style=minus-site] (5+6,0) at (5+6,0) {$-$};
\node[style=plus-site] (5+6,1) at (5+6,1){$+$};
\node[style=minus-site] (5+6,2) at (5+6,2) {$-$}; 
\node[style = plus-site] (5+6,3) at (5+6,3) {$+$}; 

\node[style=minus-site] (5+7,0) at (5+7,0) {$-$};
\node[style=minus-site] (5+7,1) at (5+7,1){$-$};
\node[style=plus-site] (5+7,2) at (5+7,2) {$+$}; 
\node[style = plus-site] (5+7,3) at (5+7,3) {$+$};
\node[style=minus-site] (5+8,0) at (5+8,0) {$-$};
\node[style=plus-site] (5+8,1) at (5+8,1){$+$};
\node[style=plus-site] (5+8,2) at (5+8,2) {$+$}; 
\node[style = plus-site] (5+8,3) at (5+8,3) {$+$}; 

\draw[line width = 1.2pt,  rounded corners, green!50!black] (10+-.5,1.5)--(10+1.5,1.5)--(10+1.5,2.5)--(10+.5,2.5)--(10+.5,3.5);
\draw[line width = 1.2pt, rounded corners,  magenta] (10+.5,-.5)--(10+.5,.5)--(10+1.5,.5)--(10+1.5,1.5)--(10+2.5,1.5)--(10+2.5,.5)--(10+3.5,.5);
\draw[color = black, dashed] (10.5,.5)--(12.5,2.5);

\end{tikzpicture}
\caption{An Ising configuration (left), its separating edge set (middle), and the contour decomposition of that edge set according to the south-east splitting convention (right); the green and purple separating edge-sets are distinct contours.}\label{fig:contour-decomposition}\vspace{-.2cm}
\end{figure}
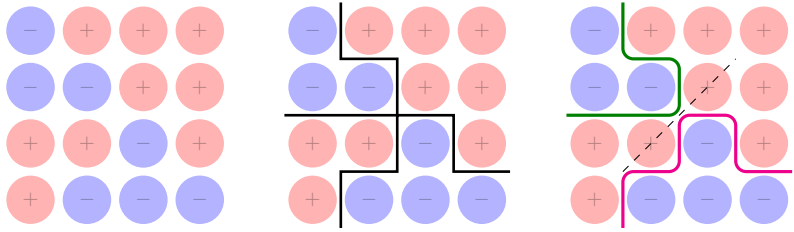

For a path or loop $\gamma$, let $|\gamma|$ denote its \emph{length}, i.e., the number of edges in $\gamma$, and let $|\underline \gamma| = \sum_k |\gamma_i|$. We say a set of loops and paths $\underline \gamma$ is \emph{$F$-admissible} if every edge of $\underline \gamma$ is contained in $F$, and $\underline \gamma$ is the contour decomposition of its edge-set $\bigcup_i \gamma_i$. If $F$ is the edge-set of $\Lambda^*$, we also say $\Lambda^*$-admissible. The boundary of a contour decomposition, $\partial \underline \gamma$, is the set of vertices that are an incident an odd number of vertices of $\bigcup_i \gamma_i$; notice that the paths of $\underline \gamma$ must each begin/end at an element of $\partial \underline \gamma$. 

    Given a subgraph $\Lambda$ of $\mathbb Z^2$, with boundary conditions $\eta$, every configuration $\sigma$ on $\Lambda$ is in bijection with a contour collection $\underline \gamma (\sigma) := \{\gamma_1,...,\gamma_k\}$ having $\partial \underline \gamma = \partial \eta$. In particular, by~\eqref{eq:Ising-measure}, for $\lambda = 0$, a measure on $\Lambda^*$-admissible contour collections $\underline \gamma$ with $\partial \underline \gamma = \partial \eta$, with weights proportional to $e^{ - 2\beta |\underline \gamma|}$ is in direct correspondence to the Ising measure $\mu_{\beta,0,\Lambda}^\eta$.

\subsubsection{The Ising interface}When $\eta$ is of $\pm$-type, the corresponding contour decomposition has exactly two boundary vertices, at the points $\partial \eta$ in  $\partial_\interior \Lambda^*$ incident the dual edges where the boundary conditions flip sign. We then use the contour representation to define the resulting interface.

    \begin{definition}\label{def:Ising-interface}
    For a subgraph $\Lambda$ of $\mathbb Z^2$, and a boundary condition $\eta$ of $\pm$-type  ($|\partial \eta | =2$), for every configuration $\sigma\in \{\pm 1\}^\Lambda$, we define its \emph{interface} $\mathscr I = \mathscr I(\sigma)$ as follows: if $\underline \gamma$ is the contour decomposition of the separating edge-set $E^{\eta}_{\Lambda}(\sigma)$; then $\mathscr I$ is the unique path in  $\underline \gamma$, i.e., the unique contour $\gamma_i$ having $\partial \gamma_i = \partial \eta$.    
      \end{definition}
      
      The interface is not a height function, in the sense that for $x\in \mathbb Z$ a column $\{x\}\times \mathbb R$ can be intersected by $\mathscr I$ at multiple heights; however, it is standard that these overhangs have an exponential tail when $\beta>\beta_c$ and thus the minimal and maximal heights above $x$ are both good proxies for ascribing a height above $x$ to the interface (see Proposition~\ref{prop:overhang-height-tail}).  Towards this, define $\hgt_x^+(\mathscr I)$ (resp., $\hgt^-_x(\mathscr I)$) to be the highest (resp., lowest) intersection point of $\mathscr I$ with $\{x\}\times \mathbb R$. 

\subsubsection{Monotonicity and interfaces} Finally, we relate the monotonicity of the Ising model to its interface. Let $\eta$ be a boundary condition of $\pm$-type on a subgraph $\Lambda\subset \mathbb Z^2$. Then for every $x$, the maps $\sigma \to \hgt^+_x(\mathscr I)$ and $\sigma \to \hgt^-_x(\mathscr I)$ are non-decreasing in $\sigma$. We will use this property extensively. Observe further that if $\Lambda_1 \subset \Lambda_2$ are two subgraphs of $\mathbb Z^2$ with boundary conditions $\eta$ which agree on $\Lambda_2^c$, then a configuration $\sigma_1$ on $\Lambda_1$ can be extended to a configuration $\sigma_1 \amalg \eta_{\Lambda_2 \setminus \Lambda_1}$ on $\Lambda_2$; in this manner, a configuration on $\Lambda_1$ gives rise to an interface on all of $\Lambda_2$ for which we can use the above monotonicity properties with respect to only the configuration on $\Lambda_1$.       As a consequence, monotonicity relations of the configuration in $\lambda$~\eqref{eq:monotonicity-in-external-field} and in the boundary conditions~\eqref{eq:monotonicity-in-bc} translate to monotonicity relations on $(\hgt_x^+)_x$ and $(\hgt_x^-)_x$.

    \subsubsection{Weights on interfaces} 
    Our goal is to understand the law of the interface $\mathscr I$ as defined above. Towards this, we study the effect of having a fixed $\mathscr I$ on the law of the Ising configuration. 
    \begin{definition}\label{def:Lambda-minus}
    Every interface induces sets $\Lambda^+(\mathscr I)$ ``above" $\mathscr I$ and $\Lambda^-(\mathscr I)$ ``below", which are, informally speaking, in the plus-phase and minus-phase respectively, given $\mathscr I$. To define these sets precisely, consider the contour collection $\{\mathscr I\}$ (clearly $\Lambda^*$-admissible), and define the spin configuration $\sigma(\mathscr I)$ to be the configuration whose separating edge-set $E_{\Lambda}^\eta(\sigma(\mathscr I))$ equals $\mathscr I$. Set $\Lambda^+(\mathscr I)$ to be the set of plus sites of $\sigma(\mathscr I)$ and set $\Lambda^-(\mathscr I)$ to be the set of minus sites $\sigma(\mathscr I)$. 
        \end{definition}
        
           Let $\partial_\interior \Lambda^+(\mathscr I)$ be the set of sites (necessarily plus) that are \emph{frozen} by the interface $\mathscr I$, i.e., every configuration with interface $\mathscr I$ has those spin assignments (these will be a subset of the sites in $\Lambda^+(\mathscr I)$ $*$-adjacent to (at $\ell^1$-distance at most $1$ away from) an edge of $\mathscr I$, but may differ at some corners of $\mathscr I$ due to the south-east splitting rule). Define $\partial_\interior \Lambda^-(\mathscr I)$ analogously. Even though $\partial_\interior \Lambda^+\cup \partial_\interior \Lambda^-$ is only a subset of sites adjacent to $\mathscr I$, it is enough to uniquely identify $\mathscr I$, and to split the measures above and below $\mathscr I$, making the configurations on $\Lambda^+\setminus \partial_\interior \Lambda^+$ and $\Lambda^- \setminus \partial_\interior \Lambda^-$ conditionally independent. Abusing notation slightly, we view the Ising measures on $\Lambda^+$ with $+$ boundary conditions as that on $\Lambda^+ \setminus \partial_\interior \Lambda^+$ with $+$ boundary conditions, and similarly for $\Lambda^-$, so that e.g., $Z_{+,\beta,\lambda,\Lambda^+(\mathscr I)}$ is the partition function on $\Lambda^+\setminus \partial_\interior \Lambda^+$ with all-plus boundary conditions. 
           
           In particular, if $\eta$ is a boundary condition on $\Lambda$ of $\pm$-type so that $|\partial \eta|= 2$, then for any external field $\lambda$, and $\mathscr I$, we can rewrite the partition function as
\begin{align}\label{eq:low-temp-interface-probability}
\mu_{\beta,\lambda,\Lambda}^{\eta}(\mathscr I)= & \omega(\mathscr I)\frac{Z_{+,\beta,\lambda,\Lambda^+(\mathscr I)} Z_{-,\beta,\lambda,\Lambda^{-}(\mathscr I)}}{Z_{\eta,\beta, \lambda,\Lambda}} \qquad \mbox{where} \qquad \omega(\gamma) = e^{ - 2\beta |\gamma|}\,.
\end{align}
In this manner, we have prescribed a set of weights on interfaces that are a tilt of the weight $e^{ - 2\beta |\mathscr I|}$ from forcing disagreements along $\mathscr I$. Using the contour representation of the Ising model, we can prescribe a dual set of weights on interfaces, which we develop in what follows. 

\subsection{The random-line representation}\label{subsec:random-line-representation} A tool that has had tremendous success in analyzing high-temperature correlation functions (and by duality on $\mathbb Z^2$, e.g.,~\eqref{eq:duality-surface-tension}) low-temperature interfaces, is the \emph{random-line representation} of the model. Since our interest is in the low-temperature ($\beta>\beta_c$) regime on (primal) subgraphs of $\mathbb Z^2$, we introduce the representation on the dual graph at a dual temperature $\beta^*<\beta_c$. The discussion in this subsection is in the absence of an external field and therefore we drop the dependencies on $\lambda$ from the notation. This random-line representation was extensively developed in~\cite{PV97,PV99}.

First of all, recalling the contour representation of the Ising model,  we can express the Ising measure on $\Lambda$ with all-$+$ boundary conditions by the probability distribution over $\Lambda^*$-admissible contour collections $\underline \gamma$ with $\partial \underline \gamma = \emptyset$,  each getting weight proportional to $e^{ - 2\beta |\underline \gamma|}$. In this manner, we have the equality in partition functions 
\begin{align}\label{eq:partition-function-random-line}
Z_{+,\beta,\Lambda} = Z_{\beta^*,\Lambda^*}^{\textsc{rl}} \qquad \mbox{where}\qquad Z^{\textsc{rl}}_{\beta^*,\Lambda^*} : = \sum_{\underline \gamma \text{ $\Lambda^*$-admissible} \,;\, \partial \mathbf \gamma = \emptyset } e^{-2\beta |\underline \gamma|}\,.
\end{align}
where the parametrization of the latter by the dual $\beta^*$ will become clear going forward.

For a (finite) subgraph $\Lambda^*$ of $(\mathbb Z^2)^*$, and a $\Lambda^*$-admissible collection of contours, $\underline \zeta$, we can define the partition function restricted to configurations having the collection of contours $\underline \zeta$ by 
\begin{align*}
Z_{\beta^*,\Lambda^*}^{\textsc{rl}} (\underline \zeta) := \sum_{\underline \gamma : \underline \zeta \cup \underline \gamma \text{ $\Lambda^*$-admissible}\,;\, \partial \underline \gamma = \emptyset } e^{ - 2\beta |\underline \zeta \cup \underline \gamma|}\,.
\end{align*}
We can then assign a \emph{weight} to collections of contours $\underline \zeta$, denoted by $q_{\beta^*, \Lambda^*}( \underline \zeta)$ as 
\begin{align}\label{eq:q-def}
q_{\beta^*,  \Lambda^*}(\underline \zeta) : = \begin{cases} \frac{Z_{\beta^*,\Lambda^*}^{\textsc{rl}}(\underline \zeta)}{Z_{\beta^*,\Lambda^*}^{\textsc{rl}}} & \mbox{ if $\underline \zeta$ is $\Lambda^*$-admissible} \\ 0 & \mbox{ else} \end{cases}\,.
\end{align}
As in~\eqref{eq:partition-function-random-line}, for a boundary condition $\eta$ of $\pm$-type on $\Lambda$, we have by~\eqref{eq:partition-function-random-line},
\begin{align}\label{eq:random-line-representation-partition-functions}
Z_{\eta,\beta,\Lambda} = \sum_{\substack{ \underline \gamma: \partial \underline\gamma = \partial \eta \\  \underline \gamma \text{ is $\Lambda^*$-admissible}}} e^{-2\beta |\underline \gamma|}  \qquad \mbox{so that}\qquad\frac{Z_{\eta,\beta,\Lambda}}{Z_{+,\beta,\Lambda}} = \sum_{\substack{ \zeta: \partial \zeta = \partial \eta \\  \zeta \text{ is $\Lambda^*$-admissible}}} q_{\beta^*,\Lambda^*} (\zeta)
\end{align}
where we observe that the latter ratio of partition functions is precisely such that its normalized logarithm is a finite-volume surface tension-like quantity. 

At the same time, the central relation of the random-line representation is the following identity expressing multi-point correlations of the Ising model in terms of sums of the weights~\eqref{eq:q-def}: for any even number of vertices $A^*$ in $\Lambda^*$, we have the identity
\begin{align}\label{eq:random-line-representation-identity}
\sum_{ \underline \zeta \text{ $\Lambda^*$-admissible}\, :\, \partial \underline \zeta = A^* } q_{\beta^*,\Lambda^*} (\underline \zeta) =  \langle  \prod_{v^*\in A^*} \sigma_{v^*} \rangle_{\beta^*,\Lambda^*}\,.
\end{align}
Together,~\eqref{eq:random-line-representation-partition-functions}--\eqref{eq:random-line-representation-identity} give us a finite volume analogue of~\eqref{eq:duality-surface-tension}. This allows us to combine analytic tools and correlation inequalities with combinatorial arguments in analyzing $\mathscr I$. 

As a result of~\eqref{eq:random-line-representation-partition-functions}--\eqref{eq:random-line-representation-identity}, we see that if $\eta$ is of $\pm$-type, with $\partial \eta = \{\vw, \ve\}$ for $\vw, \ve \in \partial_\interior \Lambda^*$, for every set of \emph{admissible interfaces} $\Gamma$, ($\Lambda^*$-admissible contours with $\partial \gamma = \{\vw,\ve\}$)
\begin{align}\label{eq:random-line-two-point-correlations}
\mu_{\beta,\Lambda}^{\eta} (\Gamma) = \frac{1}{Z_{\eta, \beta,\Lambda}} \sum_{\gamma \in \Gamma} Z_{\eta, \beta, \Lambda}(\gamma)  = \frac{1}{\langle \sigma_{\vw} \sigma_{\ve}\rangle_{\beta^*,\Lambda^*}} \sum_{\gamma \in \Gamma} q_{\beta^*,\Lambda^*} (\gamma)\,.
\end{align}

Finally, we can extend the contour weights of~\eqref{eq:q-def} to infinite subsets of $(\mathbb Z^2)^*$ by taking limits: given any $\Lambda^*\subset (\mathbb Z^2)^*$, and family of $\Lambda^*$-admissible contours $\underline \gamma$, 
\begin{align*}
q_{\beta^*,\Lambda^*} (\underline \gamma) =  \lim_{\Lambda_n^* \uparrow \Lambda^*} q_{\beta^*,\Lambda_n^*} (\underline \gamma)\,.
\end{align*}
(The existence of the above is guaranteed by the monotonicity~\eqref{eq:q-monotonicity-domain}.) Two such choices of infinite domain we will work with are $(\mathbb Z^2)^*$ and the half-infinite plane $(\mathbb Z\times \mathbb Z_{\geq 0})^*$ which we sometimes use the subscript $s.i.$ to denote. See~\cite[Section 6]{PV99} for more details on all the constructions above.

\subsubsection{Properties of the random-line representation}
We now discuss some important  properties of the weights~\eqref{eq:q-def} and this random-line representation. We refer to~\cite{PV97,PV99} as well as~\cite[Appendix~A]{Velenik} for these properties. As before, we think of $\beta>\beta_c$ (so that $\beta^* <\beta_c$) fixed, and our contour collections as living on the dual edges of $\Lambda$, in this section and in everything that follows in this paper.  Thus, for $v^*,w^*\in (\mathbb Z^2)^*$, we henceforth abbreviate 
 $$\langle \sigma_{v^*} \sigma_{w^*} \rangle^*_{\Lambda} = \langle \sigma_{v^*} \sigma_{w^*} \rangle_{\beta^*,\Lambda^*} \qquad \mbox{and}\qquad q_{\Lambda}(\underline \gamma) = q_{\beta^*,\Lambda^*} (\underline \gamma)\,.$$
 
The most basic inequalities are the following exponential lower and upper bounds on the weight functions. On the one hand, we have the following finite-energy type property: there exists $C(\beta)>0$ such that for every graph $\Lambda$,  and every $\Lambda^*$-admissible contour collection $\underline \gamma$, 
\begin{align}\label{eq:weight-function-lb}
q_\Lambda(\underline\gamma) \geq e^{- C|\underline \gamma|}\,.
\end{align}  
On the other hand, we have the following uniform upper bound on the weight function: there exists $c(\beta)>0$ and $l_0<\infty$ such that for every $\Lambda^*$, and every two points $v^*,w^*\in \Lambda^*$,
\begin{align}\label{eq:weight-function-ub}
\sum_{\substack{\gamma: \partial \gamma = \{v^*,w^*\}\,;\, |\gamma|\geq l \\ \gamma \text{ $\Lambda^*$-admissible}}} q_E(\gamma)\leq e^{- c(l-l_0)}\,.
\end{align} 
Recall that by the GKS inequality~\eqref{eq:two-point-function-monotonicity}, the two-point functions of the Ising model only increase when we increase the domain. 
At the same time, contour weights have the opposite monotonicity, i.e., if $\Lambda_1^*\supset \Lambda_2^*$ and $\underline \gamma$ is any $\Lambda_2^*$-admissible contour collection, we have 
\begin{align}\label{eq:q-monotonicity-domain}
q_{\Lambda_1}(\underline \gamma) \leq q_{\Lambda_2}(\underline \gamma)\,.
\end{align}
Denote by $q_{s.i.}$ the weight function on the upper-half plane $(\mathbb Z \times \mathbb Z_{\geq 0})^*$. For every $\Lambda\subset (\mathbb Z\times \mathbb Z_{\geq 0})$, every $\Lambda^*$-admissible $\underline \gamma$, we have
\begin{align}\label{eq:q-monotonicity-infinite-domain}
q_\Lambda(\underline \gamma) \geq q_{s.i.}(\underline \gamma) \geq q_{\mathbb Z^2} (\underline \gamma)\,.
\end{align}

Finally, we have the following inequalities describing how $q$ behaves for $\gamma$ that breaks up into disjoint pieces. For two admissible collections of contours $\underline \zeta, \underline \gamma$, we define their \emph{concatenation} $\underline \zeta \amalg \underline \gamma$ as the contour representation of the edge set $(\bigcup_i \zeta_i) \cup (\bigcup_i\gamma_i)$. The concatenation is admissible if their edge-sets are disjoint. 
On the one hand, if $\underline \gamma$ and $\underline \zeta$ are two $\Lambda^*$-compatible contour collections with disjoint edge-sets, we have  (see~\cite[Lemma 6.4]{PV99})
\begin{align}\label{eq:concatenation}
q_\Lambda(\underline \gamma \amalg \underline \zeta) \geq q_\Lambda(\underline \gamma)q_\Lambda(\underline \zeta)\,.
\end{align}
In the opposite direction, for a sequence of sites $v_1^*, v_2^* ,..., v_r^* \in \Lambda^*$, we have the BK-type inequality,
\begin{align}\label{eq:gamma-decomposition}
\sum_{\substack{ \gamma:v_1^*\to v_2^* \to \cdots \to v_r^* \\ \gamma \text{ $\Lambda^*$-admissible}}} q_\Lambda(\gamma) \leq \prod_{i} \sum_{\substack{\gamma_i : v_i^* \to v_{i+1}^* \\ \gamma_i \text{ $\Lambda^*$-admissible}}} q_\Lambda(\gamma_i)\,.
\end{align}
where here and throughout the paper, we will use the notation $\gamma: v^*\to w^*$ to denote $\partial \gamma = \{v^*,w^*\}$ and $\gamma: v^*_1 \to v^*_2 \to \cdots \to v^*_r$ to denote $\partial \gamma  = \{v^*_1, v^*_r\}$ and $v^*_2,...,v^*_{r-1} \in \gamma$.

\subsection{Sharp estimates on two-point functions}
Due to the combination of~\eqref{eq:random-line-representation-partition-functions}--\eqref{eq:random-line-representation-identity}, bounds on the low-temperature surface tension and the hight-temperature two-point function will be important tools for understanding the Ising interface. 

Long range connections at high-temperature, and 2D interfaces at low-temperature are known to have random walk like behavior. 
Using the cluster expansion framework, this was first known for $\beta$ sufficiently large~\cite{DKS}, and subsequently developed in a non-perturbative manner, holding for every $\beta>\beta_c$~\cite{CIV03,CIV08}. Though a lot of work has gone in these directions, we only quote the specific estimates which we will need to use in this paper as quoted in~\cite[Appendix A]{Velenik}. 
 We begin with the following sharp asymptotics of the two-point function. A version of these was found in~\cite{McCoyWu} for large $\beta$, and extended to all $\beta>\beta_c$ in~\cite{CIV03,CIV08}. Let $\beta>\beta_c$ so that $\beta^*<\beta_c$: there exist constants $K_1(\beta), K_2(\beta)>0$ such that for every $v^*,w^*\in (\mathbb Z^2)^*$, 
\begin{align}\label{eq:point-to-point-bounds}
\frac{K_1}{|v^*-w^*|^{\frac 12}} \exp\big( -\tau_\beta (\tfrac{v^*-w^*}{|v^*-w^*|}) |v^*-w^*|\big) \leq \langle \sigma_{v^*} \sigma_{w^*} \rangle_{\mathbb Z^2}^* \leq \frac{K_2}{|v^*-w^*|^{\frac 12}} \exp \big(- \tau_\beta (\tfrac{v^*-w^*}{|v^*-w^*|}) |v^*-w^*|\big)\,.
\end{align}
We also recall from e.g.,~\cite[Proposition 2.1]{PV97} that for every $\beta>\beta_c$, the surface tension $\tau_\beta(x)$ is strictly positive for all $x\ne 0$, and is uniformly Lipschitz as a function on $\mathbb R^2$, say with Lipschitz constant $\kappa_\beta$. 
We observe here that by a simple calculation, if we consider the two point correlation from a vertex $v^*\in (\mathbb Z^2)^*$ to $w^* = v^*+ (\ell,r)$ for $|r|\leq \ell$, we see that $|(\ell,r) -(\ell,0)| \leq r^2/(2\ell)$ and 
\begin{align}\label{eq:point-to-point-l-r}
\langle \sigma_{v^*} \sigma_{w^*}\rangle_{\mathbb Z^2}^* \geq \frac{K_1}{\sqrt{2\ell}} e^{ - \kappa_\beta r^2/(2\ell)}  e^{ -  \tau_\beta(0) \ell}\,.
\end{align}
In the presence of a floor, we feel an additional ballot theorem effect, yielding the following asymptotics of the half-plane correlation function: see e.g.,~\cite[A.23]{Velenik} and \cite[Corollary 5.2]{LMST13}. For every $\beta^*<\beta_c$, there exist constants $K_1(\beta), K_2(\beta)>0$ such that 
\begin{align}\label{eq:floor-point-to-point-bounds}
\frac{K_1}{N^{3/2}} e^{ - \tau_\beta(0) N} \leq \langle \sigma_{(-\frac 12,-\frac 12)} \sigma_{(N+\frac 12,-\frac 12)} \rangle_{s.i.}^* \leq \frac{K_2}{N^{3/2}} e^{-\tau_\beta(0) N}\,.
\end{align} 

With these, the following \emph{sharp triangle inequality} on the surface tension bounds the transversal fluctuations of the interface between $v^*,w^*\in \Lambda^*$: for every $\beta>\beta_c$ there exists $\xi_\beta>0$ such that 
\begin{align}\label{eq:sharp-triangle-inequality}
\tau_\beta(v^*)+\tau_\beta(w^*)\geq & \tau_\beta(v^*+w^*)+\xi_{\beta}(|v^*|+|w^*|-|v^*+w^*|)\,.
\end{align}

\subsection{Preliminary estimates on the Ising interface}
Given the inputs from the random-line representation and bounds on the high-temperature two-point function, we can use~\eqref{eq:random-line-two-point-correlations} to deduce properties of the Ising interface at low temperatures. We recap some of these consequences, including diffusive bounds on the maximal transversal fluctuation of the interface. These can be expressed in various forms, either in terms of the random-line weights, interface probabilities at low-temperature: we present them in the form(s) they will be applied in our paper. Recall that $\Lambda_N = \llb 0,N\rrb \times \llb 0,N\rrb$ and $\pm$-.b.c.\ are  $-$ below the $x$-axis, and $+$ on and above the $x$-axis.

\subsubsection{Overhangs and backtracking}
We begin with recalling simple bounds on the \emph{overhangs} of the interface, wherein $\mathscr I$ intersects a column $\{x\}\times \mathbb R$ at multiple heights. Using e.g., the tools in the previous section, it is classical to see the following bound on the maximum overhang size. These are standard and we e.g., refer to Eq.~(6) of the survey~\cite{IoVe16} . 

\begin{proposition}\label{prop:overhang-height-tail}
Let $\beta>\beta_c$ and consider the Ising interface on $\Lambda_N$ with $\pm$ boundary conditions. There exists $C(\beta)>0$  such that for every $r$, 
\begin{align*}
    \max_{x\in \partial_\south \Lambda_{N}} \mu_{\beta,\Lambda_N}^\pm (\hgt_x^+ -\hgt_x^-> r) \leq Ce^{ - r/C}\,.
\end{align*}
The same bound holds in other geometries, including the infinite strip $\Lambda_{N,\infty} = \llb 0,N\rrb \times \llb -\infty,\infty\rrb$. 
\end{proposition}
This demonstrates that the probability of overhangs of a large height have exponential tails. However, to leverage the BK type inequalities of ~\eqref{eq:gamma-decomposition}, it will be important for us to also bound the probability that the interface has no overhang through a particular column above $x$, i.e., that $\hgt_x^+ \neq \hgt_x^-$. Specifically, we will be interested in the probability that a contour $\gamma: \vw \to \ve$ has no overhangs in its two boundary columns, i.e., on the vertical columns to which $\vw$ and $\ve$ belong respectively: we call this property a no-backtracking property, denoted $\gamma \in \nbt$.

\begin{proposition}\label{prop:no-backtracking-domain}
Fix any $\beta>\beta_c$. There exists a constant $\varepsilon_\beta>0$ (going to $1$ as $\beta\to\infty$) such that, for every $\vw \in (\mathbb Z^2)^*$, $\ve = \vw+(\ell,r)$ for $\ell\geq 1$ and  $|r|\leq \ell$, 
\begin{align}\label{eq:prob-no-leaving-strip}
    \frac{1}{\langle \sigma_{\vw} \sigma_{\ve}\rangle_{\mathbb Z^2}^*} \sum_{\gamma: \vw \to \ve\,;\,\gamma \notin \nbt} q_{\mathbb Z^2} (\gamma) <1-\varepsilon  \,.
\end{align}
\end{proposition}

We do not include a proof as it is standard (see e.g.,~\cite[Proof of Lemma 4.4.7]{VelenikThesis},~\cite[Proof of Theorem 5.3]{LMST13}): by~\eqref{eq:weight-function-lb}, with a probability bounded away from zero, we can force the interface to traverse horizontally away from the columns of $\vw,\ve$ to a large constant distance $K_\beta$ so that the probability, from there, of backtracking to the columns of $\vw,\ve$ sums to $1-\varepsilon$ for $\varepsilon>0$.

\subsubsection{Transversal fluctuations}

An important aspect of the random walk behavior of Ising interfaces is that their transversal fluctuations satisfy Gaussian tail bounds. 
Consider the random-line probability of going through a point $(0,r)$  on the transit from $(-\ell,0)$ to $(\ell,0)$. A simple geometric calculation that for the vectors $v=(\ell,r), w= (\ell,-r)$ for $\ell, r$ such that $r\leq \ell$, shows the following ``quadratic" remainder in the sharp triangle inequality~\eqref{eq:sharp-triangle-inequality}: 
\begin{align}
|v|+|w| - |v+w|  = 2 \sqrt{\ell^2 + r^2}  - 2\ell   \geq  r^2/\ell\,.
\end{align}
Thus, using the fact that $\tau_\beta (\theta) = \tau_\beta(-\theta)$, we deduce from~\eqref{eq:sharp-triangle-inequality} that $\tau_{\beta} ( (\ell,r)) \geq \tau_{\beta}((\ell,0)) + \xi_\beta \frac{r^2}{2\ell}$; by~\eqref{eq:point-to-point-bounds} and~\eqref{eq:gamma-decomposition}, this bounds the weight of random-lines $\gamma: (-\ell,0) \to (0,r)\to (\ell,0)$. For any $\rho\leq \ell$, summing this over heights $\rho \vee K\sqrt \ell \leq r \leq \ell$ gives the bound 
\begin{align}\label{eq:midpoint-height-fluctuation-bound}
\mu_{\beta,\Lambda_{\ell,\infty}}^{\pm}(\hgt_{(\ell/2,0)}^+\geq\rho) & \leq Ce^{-\xi_\beta \rho^{2}/(2\ell)}\,.
\end{align}

This is in fact a simple case of the following Gaussian upper tail on the \emph{maximum} height deviation of the random-line. While a union bound using the sharp triangle inequality as in~\eqref{eq:midpoint-height-fluctuation-bound} induces a polynomial prefactor, removing this prefactor requires a refined multi-scale analysis~\cite{LMST13}. 

\begin{proposition}[{\cite[Theorem 5.3]{LMST13}}]\label{prop:max-height-fluctuation-no-field}
Let $\beta>\beta_c$ and consider the Ising model on the strip $\Lambda_{\ell,\infty}$ with $\pm$ boundary conditions.  There exists $C(\beta)>0$ such that for every $\rho \leq \ell$, 
\begin{align}\label{eq:max-height-fluctuation-no-field}
    \mu_{\beta,\Lambda_{\ell,\infty}}^\pm \big(\max_{x \in \partial_\south \Lambda_{\ell,\infty}} \hgt_x^+ >\rho \big) \leq Ce^{ - \rho^2 /(C\ell) }\,,
\end{align}
and the same bound holds for $\min_{x\in \partial_\south \Lambda_{\ell,\infty}} \hgt_x^-<-\rho$. 
\end{proposition}

By monotonicity in boundary conditions, we can translate this to the following bound, for the maximal height fluctuations of interfaces under $(\pm,\langle \vw,\ve\rangle)$ boundary conditions which are $+$ above the line between $\vw = (-\frac 12,-\frac 12)$ and $\ve = (\ell+ \frac 12,r-\frac 12)$ and minus below. We present this for random-line representation, using ~\eqref{eq:two-point-function-monotonicity} with~\eqref{eq:q-monotonicity-domain} to lower bound random-lines weights in $\Lambda_{\ell,\infty}^*$ by those in $(\mathbb Z^2)^*$, and upper bound the two-point function in $\Lambda_{\ell,\infty}^*$ by that in $(\mathbb Z^2)^*$. 

\begin{corollary}\label{cor:max-height-fluctuation-not-straight}
Let $\beta>\beta_{c}$. For every $r$ such that $0\leq r \leq \ell$ and every  $\rho \leq \ell$, 
\begin{align}\label{eq:max-height-fluctuation-not-straight}
\mu_{\beta, \Lambda_{\ell,\infty}}^{\pm,\langle \vw,\ve\rangle } \big(\max_{x\in \partial_\south \Lambda_{\ell,\infty}} \hgt_x^+ \geq r+\rho\big)\leq & Ce^{-\rho^{2}/(C\ell)}\,.
\end{align}
and the same bound holds for $\min_{x\in \partial_\south \Lambda_{\ell,\infty}} \hgt_x^- <-\rho$. 
In particular, we have 
\begin{align}\label{eq:max-height-fluctuation-not-straight-random-line}
\frac{1}{\langle \sigma_{\vw} \sigma_{\ve} \rangle_{\mathbb Z^2}^*} \sum_{\substack{\gamma: \vw\to \ve \\ \gamma \subset \Lambda_{\ell,\infty}^*\,,\,\gamma \not \subset \mathbb R \times [- \rho,r+ \rho]}} q_{\mathbb Z^2}(\gamma) \leq Ce^{ -  \rho^2/(C\ell)}\,.
\end{align}
\end{corollary}

Another useful corollary of Proposition~\ref{prop:max-height-fluctuation-no-field} will be the analogous inequality on the maximal height oscillation, in the presence of a floor, up to constant multiple. The proof is by a standard coupling argument (coupling the interface on $\Lambda_{\ell,\infty}^{s.i.}$ to be below the interface of $\Lambda_{\ell,\infty}$ endowed with $(\pm,\rho/2)$ boundary conditions): we refer to e.g.,~\cite[Proposition 3.6]{GL16c} for this kind of coupling argument. 

\begin{corollary}\label{cor:max-height-fluctuation-floor}
Let $\beta>\beta_c$ and consider the Ising model on the semi-infinite strip $\Lambda_{\ell, \infty}^{s.i.} = \llb 0,\ell\rrb \times \llb 0,\infty\rrb$ with $\pm$ boundary conditions. There exists $C(\beta)>0$ such that for every $\rho \leq \ell$, 
\begin{align}\label{eq:max-height-fluctuation-floor}
\mu_{\beta,\Lambda_{\ell,\infty}^{s.i.}}^{\pm} \big(\max_{x\in \partial_\south \Lambda_{\ell,\infty}^{s.i.}} \hgt_x^+ >\rho \big) \leq Ce^{ - \rho^2 /(C\ell)}\,. 
\end{align}
\end{corollary}

\subsection{Notational disclaimers}
Let us mention some notational choices we take for readability, which will apply to the remainder of the paper. Throughout the paper, we will take $\beta>\beta_c$ fixed, and therefore drop $\beta$ from the notation (e.g., in partition functions and Ising distributions). Moreover, since the principal domain we consider is $\Lambda = \Lambda_N$, we use an $N$ subscript to denote that choice of graph, e.g., $\mu_{\lambda,N}^{\eta} = \mu_{\beta, \lambda, \Lambda_N}^{\eta}$. 

Our results concern the asymptotic regime of large $N$: thus, we will not quantify over $N$ and statements should be understood to hold uniformly over all $N$ sufficiently large. Also, for ease of presentation, when we divide regions into subregions of some fixed size and we do not differentiate between the remainder sub-regions which may have different size if there are issues of divisibility. All such rounding issues and integer effects can be handled with the obvious modifications.    

Finally, since $\beta, c_\lambda$ will be fixed throughout all proofs, throughout this paper, we let $C$ denote some constant $C(\beta,c_\lambda)>0$, that may differ between lines in a proof, unless specified otherwise.

\section{Domain enlargements and a priori regularity estimates}\label{sec:enlargements}
Our arguments at local scales rely on conditioning on the entry and exit points of the interface through a smaller-scale domain, then examining the behavior inside that smaller domain conditionally on this entry and exit data. As emphasized in Section~\ref{sec:ideas-of-proofs}, a key difficulty posed by the overhangs of the interface and by the sub-critical bubbles of the Ising model, is that unlike the spin configuration, the law of the Ising interface is not Markovian. 

In Section~\ref{subsec:enlargements}, we develop an enlargement framework to couple the law of the interface, conditionally on the entry and exit data, to the law of an interface of a slightly enlarged strip with $(\pm,h)$ boundary conditions. In Section~\ref{sec:multi-strip-domains}, we generalize this to a collection of $k$ well-separated strips with prescribed entry/exit data. The proofs of these are important technical ingredients involving subtle couplings under conditional measures, but to make the flow of the arguments towards our main theorems transparent, we defer their proofs to Section~\ref{sec:proofs-domain-enlargements}. 
Still, as a crucial application, in Section~\ref{subsec:spikiness}, we use these couplings to get a priori regularity bounds on the local oscillations of $\mathscr I$.

\subsection{Coupling to domain enlargements}\label{subsec:enlargements}
Consider some rectangular strip 
\begin{align}\label{eq:R-def}
R= \llb x_\west,x_\east \rrb \times \llb 0,N\rrb \subset \Lambda_N\,.
\end{align}
 Suppose we know that the entry and exit data of the interface into $R$ satisfies 
 \begin{equation}\label{eq:entry-exit-conditioning-event}
\begin{aligned}
\mathcal E_{\leq h} &: = \{{\mathscr I}: \hgt_{x_\west}^+\leq h , \hgt_{x_\east}^+\leq h\}\,,\qquad \mbox{or alternatively} \\
\mathcal E_{\geq h} & := \{ {\mathscr I}: \hgt_{x_\west}^- \geq h, \hgt_{x_\east}^- \geq h\}\,.
\end{aligned}
\end{equation}
 We would like to say, e.g., that the law of the Ising measure induced on $R$ given $\mathcal E_{\leq h}$ stochastically dominates the one with $(\pm,h)$ boundary conditions on $R$  (plus above $h$ and minus below $h$). While this is not true, one can, with high probability, couple the configuration induced on $R$ under $\mathcal E_{\leq h}$ to stochastically dominate the one sampled from an \emph{enlargement} of $R$ with its own induced $(\pm,h)$ boundary conditions.  
The strategy of coupling induced \emph{random} boundary conditions to an enlargement was used in simpler settings without any conditioning on the external behavior of the interface in~\cite{MaTo10} then in similar veins,~\cite{LMST13,GL16a}. The conditioning on the entry and exit data of the interface, make such a coupling quite delicate. (Indeed one would not have a similar coupling if the conditioning were on non-monotone events in the entry/exit data.)  

\begin{figure}
\centering
\begin{tikzpicture}
\node at (0,0) {
\includegraphics[width=.9\textwidth]{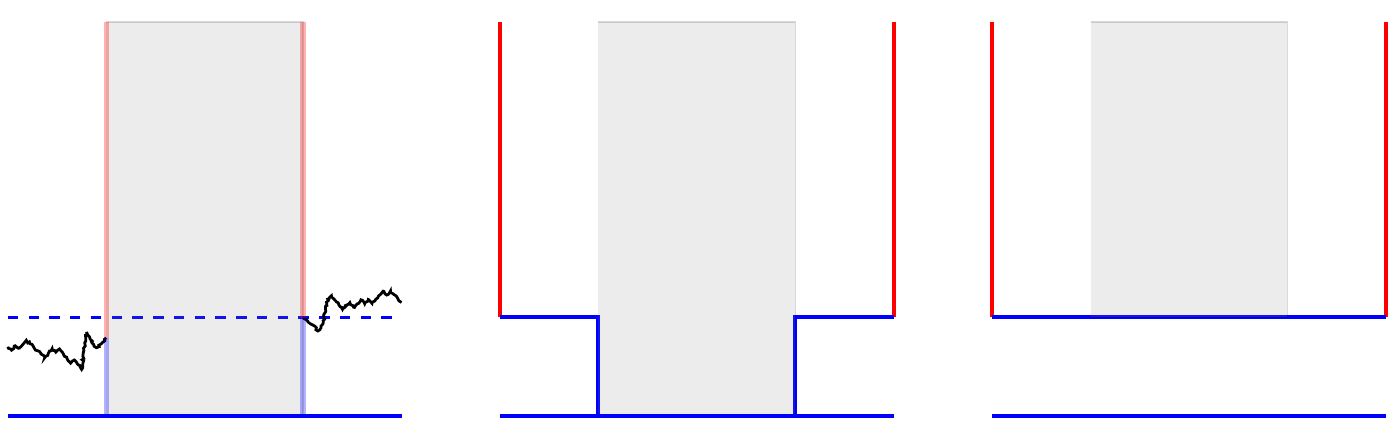}
};
\node[font = \large] at (-5.25,1.5) {$R$}; 
\node[font = \large] at (0,1.5) {$\tshape$}; 
\node[font = \large] at (5.25,1.5) {$\enlarge^\uparrow$};

\draw[<->, thick]  (1.15,-2.05)-- (1.15,-1.05);
\draw[<->, thick]  (1.02,-.9)-- (1.98,-.9);
\node at (1.4, -1.5) {$h$};
\node at (1.5, -.75) {$r$};

\end{tikzpicture}
\vspace{-.2cm}
\caption{Conditionally on its entry and exit through $R$, the induced Ising measure on $\partial_{\east,\west} R$ is a random b.c.\ that is mostly plus above $\sI$, and mostly minus below (left). This can be coupled to two enlargements $\tshape$ (middle) and $\enlarge^\uparrow$ (right) in such a way that on the gray regions, it is stochastically below the law induced by $\tshape$ and $\enlarge^\uparrow$ with $(\pm,h)$ boundary conditions, except with exponentially small probability.}\label{fig:enlargement-sets}\vspace{-.2cm}
\end{figure}
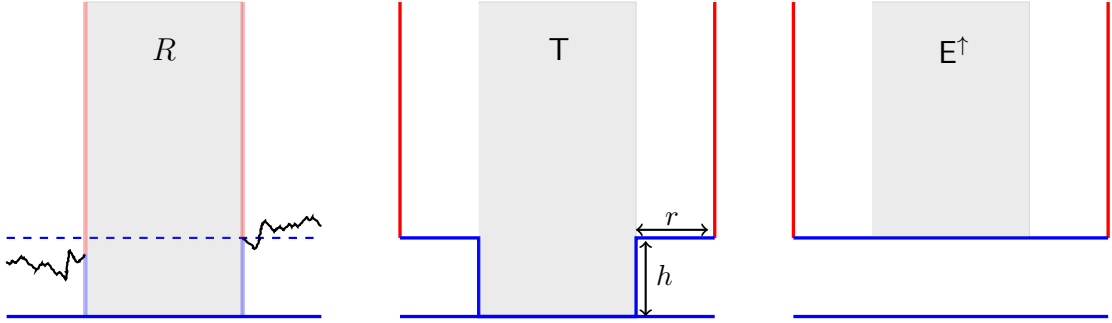

\subsubsection{The decreasing enlargement coupling}
We begin with the case where we have conditioned on the interface entering and exiting $R$ below some height $h$---here we couple it to a ``more minus" measure coming from an enlargement with truly plus-minus boundary conditions. For every rectangular subset 
$R$ as in~\eqref{eq:R-def}, define its \emph{enlargement} by a distance $r$ as 
\begin{align}\label{eq:enlargement}
\enlarge = \enlarge_{r}(R) := \llb x_\west -r, x_\east +r\rrb \times \llb 0, N\rrb\,.
\end{align} 

There are two kinds of enlarged domains of the strip $R$ which we will consider (see Figure~\ref{fig:enlargement-sets}): the most basic is the decreasing strip-enlargement, similar to $\enlarge$ but beginning at height $h$: 
\begin{align}\label{eq:strip-enlargement}
\enlarge^\uparrow:= \enlarge^\uparrow_{r,h}(R) := \enlarge_{r}(R) \cap (\mathbb R \times \llb h,\infty\rrb) \,.
\end{align}
A somewhat different domain that we also must consider is the \emph{T-shape-enlargement} described by 
\begin{align}\label{eq:tshape-enlargement}
\tshape:=\tshape_{r,h}(R) := R \cup \enlarge_{r,h}^\uparrow(R)\,.
\end{align}
Both of the above will be viewed as decreasing enlargements, i.e., when endowed with $(\pm,h)$ boundary conditions that are minus  on and below height $h$, and plus above, the distributions these enlargements induce on $R\cap \enlarge^\uparrow$ and $R$ are respectively (approximately) stochastically lower than that induced on $R\cap \enlarge^\uparrow$ and $R$ under $\mu_{\lambda,N}^{\pm}(\cdot \mid \mathcal E_{\leq h})$. Note that by monotonicity of the boundary conditions, $\mu_{\lambda, \enlarge_{r,h}^\uparrow(R)}^{(\pm,h)}$ is stochastically below $\mu_{\lambda, \tshape_{r,h}(R)}^{(\pm,h)}$ on $\enlarge^\uparrow$, and in particular $R\cap \enlarge^\uparrow$.

\begin{proposition}\label{prop:minus-enlargement}
Fix $\beta>\beta_c$ and any $\lambda \geq 0$. Let $R$ be as in~\eqref{eq:R-def} for any $x_\west\leq x_\east$. There exists $C(\beta)>0$ such that for every $h$, every $r$ and every decreasing $A\subset \{\pm 1\}^{R}$, 
\begin{align*}
\mu_{\lambda, N}^{\pm}(\sigma_R \in A \mid \mathcal E_{\leq h}) \leq \mu_{\lambda, \tshape}^{(\pm,h)} (\sigma_R \in A) + CNe^{ - r/C}\,,
\end{align*}
and for every decreasing $A\subset \{\pm1 \}^{R\cap \enlarge^\uparrow}$\,,
\begin{align*}
\mu_{\lambda, N}^{\pm}(\sigma_{R\cap \enlarge^\uparrow} \in A \mid \mathcal E_{\leq h}) \leq \mu_{\lambda, \enlarge^\uparrow}^{(\pm,h)} (\sigma_{R\cap \enlarge^\uparrow} \in A) + C N e^{ - r/C}
\end{align*}
where the $(\pm, h)$ boundary conditions are plus above height $h$ and minus at and below height $h$. 
\end{proposition}

\subsubsection{The increasing enlargement coupling}\label{subsec:increasing-enlargement}
As indicated above we also need to consider enlargements in the other direction, when we know that the interface \emph{exceeds} some height $h$. 
This enlargement is complicated by the external field, making it difficult to couple the induced measure on the box $R$ to a nearby all-minus boundary. Here, various monotone events  go in opposite directions and we need to control them in different ways. Recall $\mathcal E_{\geq h}$ from~\eqref{eq:entry-exit-conditioning-event}, 
and recall $\enlarge_{r}(R)$ from~\eqref{eq:enlargement}. 
Define the 
 (increasing) strip-enlargement of $R$ up to $h$ as  
\begin{align*}
 \enlarge^\downarrow:= \enlarge^\downarrow_{r,h}(R) := \enlarge_r(R) \cap (\mathbb R \times \llb 0,h\rrb)\,.
\end{align*}

\begin{proposition}\label{prop:plus-enlargement}
Fix $\beta>\beta_c$ and $\lambda  = \frac{c_\lambda}{N}$. There exists $C(\beta)>0$ such that we have the following. For every $h,\mathcal M$ and $r$, and every increasing event $A\in \{\pm 1\}^{R\cap \enlarge^\downarrow}$, 
\begin{align*} 
\mu_{\lambda, N}^{\pm} \big(\sigma_{R\cap \enlarge^\downarrow} \in A\mid \mathcal E_{\geq h} \big) & \leq \mu_{\lambda, \enlarge^\downarrow}^{(\pm,h)} (\sigma_{R\cap \enlarge^\downarrow} \in A) + C N e^{ 8c_\lambda (\mathcal M+\frac{r|x_\west - x_\east|+ r^2}N)} e^{- r/C}  \\ 
& \quad + \frac{\mu_{\lambda,N}^{\pm}(\max_{x}\hgt_x^+>\mathcal M)}{\mu_{\lambda,N}^{\pm}(\mathcal E_{\geq h})}\,.
\end{align*}
\end{proposition} 

\subsection{Multi-strip domain enlargements}\label{sec:multi-strip-domains}
Propositions~\ref{prop:minus-enlargement} and~\ref{prop:plus-enlargement} were special cases of the following more general enlargement machinery which can be used to obtain the exponential decay of correlations for monotone events; in particular, we apply this to deduce our large deviations estimate for multi-height oscillations, as in Theorem~\ref{thm:multipoint-height-bounds}.

Consider a collection of disjoint rectangular strips $\mathfrak R = (R_1 ,..., R_k)$ where $$ R_{i}  = \llb x_\west^{(i)}, x_\east^{(i)}\rrb \times \llb 0 ,N\rrb \subset \Lambda_N$$ 
and the endpoints are such that $x_\west^{(i)} < x_\east^{(i)} < x_{\west}^{(i+1)} < x_\east^{(i+1)}$ for all $i\leq k-1$. Let $\mathfrak h = (h_1 ,...,h_k)$ be the entry/exit heights for these interfaces, so that 
\begin{align}\label{eq:entry-exit-conditioning-event-multi-height}
    \mathcal E_{\leq \mathfrak h}:= \bigcap_{i\leq k} \big\{\sI : \hgt_{x_\west^{(i)}}^+ \leq h_i , \hgt_{x_\east^{(i)}}^+ \leq h_i \big\}\,, \quad \mbox{and} \quad \mathcal E_{\geq \mathfrak h} = \bigcap_{i\leq k}  \big\{ \sI : \hgt_{x_\west^{(i)}}^- \geq h_i, \hgt_{x_\east^{(i)}}^- \geq h_i \big \}\,.
\end{align}
We give a version of the decreasing and increasing enlargement couplings that couple the measures on the strips $(R_i)$ above to their decreasing/increasing enlargements up to an error that is exponentially decaying in the distances between the strips. For every collection of $k$ disjoint rectangular subsets $\mathfrak R$ and every collection of heights $\mathfrak h$, for every $r$, recalling the notation \eqref{eq:strip-enlargement}, define   
\begin{align}\label{eq:multistrip-enlarge-minus}
    \mathfrak E^{\uparrow} : = \mathfrak E^{\uparrow}_{r,\mathfrak h} (\mathfrak R) = \big(\enlarge^\uparrow_{r,h_i}(R_i)\big)_{i\leq k}\,.
\end{align}
Similarly, recalling \eqref{eq:tshape-enlargement}, we can define the ensemble of T-shape enlargements, 
\begin{align}\label{eq:multistrip-Tshape}
    \mathfrak T := \mathfrak T_{r,\mathfrak h}(\mathfrak R) : = \big(\tshape_{r,h_i}(R_i)\big)_{i\leq k} = \mathfrak R \cup \mathfrak E^{\uparrow}_{r,\mathfrak h}(\mathfrak R)\,,
\end{align}
where this intersection is naturally taken element by element. We will view $\mathfrak R, \mathfrak E^{\uparrow}, \mathfrak T$ both as collections of subsets of $\Lambda_N$, and as a single subset i.e., as the union of their elements. Notice that if the distances between $x_\east^{(i)}$ and $x_\west^{(i+1)},$ and vice versa are at least $2r+2$, there exists a configuration on $\Lambda_N \setminus \mathfrak E^\uparrow$ (resp., $\Lambda_N\setminus \mathfrak T$) under which the measure is a product measure over $\enlarge_{r,h_i}^\uparrow(R_i)$ with induced $(\pm,h_i)$ boundary conditions (resp., $\tshape_{r,h_i}(R_i)$ with $(\pm,h_i)$ boundary conditions): we denote these boundary conditions on $\Lambda_N \setminus \mathfrak E^\uparrow$ or $\Lambda_N \setminus \mathfrak T$ as $(\pm, \mathfrak h)$ boundary conditions.  

\begin{proposition}\label{prop:multi-strip-minus-enlargement}
Fix $\beta>\beta_c$ and any $\lambda \geq 0$; there exists $C(\beta)>0$ such that the following holds. For every $k$, every $\mathfrak h = (h_i)_{i\leq k}$, every $r$ and every  $\mathfrak R$ as above such that $ x_\west^{(i+1)}- x_\east^{(i)}> 2r+2$ for all $i\leq k-1$, we have for every decreasing $A\subset \{\pm 1\}^{\mathfrak R}$, 
\begin{align*}
    \mu_{\lambda,N}^{\pm}(\sigma_{\mathfrak R} \in A \mid \mathcal E_{\leq \mathfrak h}) \leq \mu_{\lambda,\mathfrak T}^{(\pm,\mathfrak h)} (\sigma_\mathfrak R \in A) + CN k e^{-r/C}\,,
\end{align*}
and for every decreasing $A\subset \{\pm 1\}^{\mathfrak R \cap \mathfrak E^\uparrow}$, 
\begin{align*}
    \mu^{\pm}_{\lambda,N} (\sigma_{\mathfrak R \cap \mathfrak E^\uparrow} \in A \mid \mathcal E_{\leq \mathfrak h}) \leq \mu_{\lambda,\mathfrak E^{\uparrow}}^{(\pm,\mathfrak h)} (\sigma_{\mathfrak R \cap \mathfrak E^\uparrow} \in A ) + CN ke^{ - r/C}\,.
\end{align*}
\end{proposition}

The analogous increasing multi-strip enlargement coupling entails coupling to the collection 
\begin{align}\label{eq:multistrip-enlarge-plus}
    \mathfrak E^{\downarrow}  = \mathfrak E_{r,\mathfrak h}^{\downarrow} (\mathfrak R) := \big(\enlarge^\downarrow_{r,h_i} (R_i)\big)_{i\leq k}\,.
\end{align}
\begin{proposition}\label{prop:multi-strip-plus-enlargement}
Fix $\beta>\beta_c$ and $\lambda = \frac{c_\lambda}{N}$. There exists $C(\beta)>0$ such that we have the following. For every $k$, $\mathfrak h = (h_i)_{i\leq k}$, $\mathcal M$, and every  $\mathfrak R$ such that $d(x_\east^{(i)},x_\west^{(i+1)})> 2r+2$ for all $i\leq k-1$, for every increasing $A\subset \{\pm 1\}^{\mathfrak R\cap \mathfrak E^\downarrow}$, 
\begin{align*}
    \mu_{\lambda,N}^{\pm}(\sigma_{\mathfrak R \cap \mathfrak E^{\downarrow}} \in A \mid \mathcal E_{\leq \mathfrak h}) & \leq \mu_{\lambda,\mathfrak E^{\downarrow}}^{(\pm,\mathfrak h)} (\sigma_{\mathfrak R \cap \mathfrak E^{\downarrow}} \in A) + CN k e^{ 8 c_\lambda (\mathcal M + \frac{r  \sum_{i\leq k} |x_\east^{(i)} -x_\west^{(i)}| + k r^2}{N})} e^{ - r/C} \\ 
    & \qquad + \frac{\mu_{\lambda,N}^{\pm}(\max_x \hgt_x^+ >\mathcal M)}{\mu_{\lambda,N}^{\pm}(\mathcal E_{\geq \mathfrak h})}.
\end{align*}
\end{proposition}

\subsection{A priori regularity estimates of ${\sI}$}\label{subsec:spikiness}
In this section, we use the enlargements above, together with the locally Brownian nature of no-field interfaces to show that with high probability, the interface is not ``spiky", i.e., does not locally oscillate in ways that would be atypically irregular for a no-field interface. Consider 
$$I  = \llb x_\west ,x_\east\rrb \subset \partial_\south \Lambda_N \qquad \mbox{and let} \qquad R:= I \times \llb 0,N\rrb\,.$$
 We have the following bound on the local oscillations of the interface in the strip $\Lambda_I$.

\begin{proposition}\label{prop:spikiness-bound}
Fix $\beta>\beta_c$ and $\lambda = \frac{c_\lambda}{N}$ for $c_\lambda>0$; there exists $C= C(\beta, c_\lambda)>0$ such that the following holds. For every $I: |I|\geq N^{2/3}$ and every $h$, 
\begin{align*}
\mu_{\lambda,N}^{\pm} \big(\hgt_{x_\west}^+ \leq h, \hgt_{x_\east}^+  \leq h,\max_{x\in I} \hgt_{x}^+ \geq 3h/2\big) &  \leq C \exp ( - h^2 /C |I|) + C e^{ - N^{2/3}/C}\,. 
\end{align*}
For every fixed $A>0$ and every $I: N^{2/3}\leq |I|\leq AN^{2/3}$, for every $h\geq K\sqrt{|I|}$ for $K$ large (depending on $A, \beta, c_\lambda)$, ,
\begin{align*}
\mu_{\lambda, N}^{\pm} \big(\hgt_{x_\west}^- \geq 2h, \hgt_{x_\east}^- \geq 2h, \min_{x\in I} \hgt_{x}^- \leq 3h/2\big)  & \leq  C\exp(- h^2 /C|I|) + C e^{ - N^{1/3}/C}\,.
\end{align*}
\end{proposition} 

For upward spikiness of the interface, upon localizing to the appropriate scale of Brownian oscillations with the decreasing enlargement, the external field only decreases the probability of an upwards oscillation. On the other hand, for downward spikiness there is a competition between the tilt on the measure induced by the external field and Gaussian tail bounds on the interface's fluctuations, in addition to the extra difficulties introduced from using $E^\downarrow$ rather than $E^\uparrow$.       

\begin{proof}[\textbf{\emph{Proof of Proposition~\ref{prop:spikiness-bound}}}]
Let us begin with the simpler first bound.

\smallskip
\noindent
\textbf{Upwards oscillation:} We can obviously bound the probability in question as 
\begin{align*}
\mu_{\lambda,N}^{\pm}(\hgt_{x_\west}^+ \leq h, \hgt_{x_\east}^+  \leq h,\max_{x\in I} \hgt_{x}^+ \geq 3h/2) \leq \mu_{\lambda,N}^{\pm}( \max_{x\in I} \hgt_x^+ \geq 3h/2 \mid \mathcal E_{\leq h} )\,,
\end{align*}
where $\mathcal E_{\leq h}$ is as in~\eqref{eq:entry-exit-conditioning-event}. Note that under $\mathcal E_{\leq h}$, in order for $\max_{x\in I} \hgt_x^+ \geq 3h/2$, there must be a minus vertical crossing of $I \times \llb h,3h/2\rrb$, denoted $\co_v^-(I\times \llb h,3h/2\rrb)$. (It is important that we are working with $\hgt^+_x$ here, as the interface is therefore not allowed to use minuses along $\partial_{\east,\west} R$ above height $h$ to attain climb to $3h/2$.) Observe that $\co_v^-(I\times \llb h,3h/2\rrb)$ is measurable with respect to the configuration on $R \cap \enlarge^\uparrow$ and is a decreasing event, where $\enlarge^\uparrow = \enlarge^\uparrow_{r,h}(R)$ is as in~\eqref{eq:strip-enlargement}. As such, by Proposition~\ref{prop:minus-enlargement}, for every $r$, the right-hand side above is at most  
\begin{align*}
\mu_{\lambda,N}^{\pm}\big(\co_v^-(I\times \llb h,3h/2\rrb)\big) \leq \mu_{\lambda,\enlarge^\uparrow}^{(\pm,h)} \big(\co_v^-(I\times \llb h,3h/2\rrb)\big) + C N e^{ - r/C}\,.\,.
\end{align*}
But if under $\enlarge^\uparrow$ the event $\co_v^-(I\times \llb h,3h/2\rrb)$ holds, its interface will reach height $3h/2$; thus,   
\begin{align*}
\mu_{\lambda,N}^{\pm}\big(\co_v^-(I\times \llb h,3h/2\rrb)\big)\leq \mu_{\lambda,\enlarge^\uparrow}^{(\pm,h)}\Big(\max_{x\in I} \hgt_x^+ \geq 3h/2\Big) + CN e^{ -r/C}\,.
\end{align*}
If we choose $r= N^{2/3}$, we see that the strip-enlargement $\enlarge^\uparrow$ is a rectangle of height $N-h$ and width $L = |I|+2N^{2/3}$, with boundary conditions that are minus on $\partial_\south \enlarge^\uparrow$, and plus on $\partial_{\north,\east,\west}\enlarge^\uparrow$.  Then, by monotonicity in $\lambda$ and in boundary conditions~\eqref{eq:monotonicity-in-bc} and \eqref{eq:monotonicity-in-external-field}, we can remove the external field and push the top boundary to infinity, so that by Corollary~\ref{cor:max-height-fluctuation-floor}, for every $h\leq |I|$, 
\begin{align*}
\mu_{\lambda, \enlarge^\uparrow}^{(\pm,h)} \Big(\max_{x\in I} \hgt_x^+ \geq 3h/2\Big) &  \leq \mu_{0,\Lambda_{L,\infty}}^{\pm} \Big(\max_{x\in \partial_\south \Lambda_{L,\infty}} \hgt_x^+ \geq h/2\Big)\leq C\exp( - h^2/C L)\,,
\end{align*}
for some $C(\beta)>0$. Since $|I|\geq N^{2/3}$, we have $L \leq 3 |I|$, concluding the proof. \\

\noindent
\textbf{Downwards oscillation:}
We now turn to the probability of having an irregular downwards oscillation, for which we use the ``increasing" strip-enlargement.  By monotonicity in external field and in boundary conditions, and Corollary~\ref{cor:max-height-fluctuation-floor}, there exists $C(\beta,c_\lambda)>0$ such that for every $\mathcal M \leq N$, 
\begin{align}\label{eq:max-height-fluctuation-with-field}
\mu_{\lambda,N}^{\pm} \big( \max_{x\in \partial_\south \Lambda_N} \hgt_x^+ > \mathcal M \big) \leq  \mu_{0,\Lambda_{N,\infty}}^{\pm}\big(  \max_{x\in \partial_\south \Lambda_N} \hgt_x^+ >\mathcal M\big) \leq C \exp ( -  \mathcal M^2 /CN)\,.
\end{align}
We apply~\eqref{eq:max-height-fluctuation-with-field} with $\mathcal M = \e N^{2/3}$ for an $\e$ to be taken sufficiently small depending on $\beta, c_\lambda$, to get
\begin{align*}
\mu_{\lambda,N}^{\pm}(\max_{x\in \partial_\south \Lambda_N} \hgt_x^+ >\mathcal M ) \leq C\exp( - \e^2 N^{1/3}/C)\,.
\end{align*}
We now split the probability we wish to bound into two cases: either $I,h$ are such that   
\begin{align*}
\mu_{\lambda,N}^{\pm}( \hgt_{x_\west}^- \geq 2h, \hgt_{x_\east}^- \geq 2h) <  C\exp ( - \e^2 N^{1/3}/2C)\,,
\end{align*}
or the reverse inequality holds. In the former case, we obviously would have 
\begin{align*}
\mu_{\lambda,N}^{\pm}(\hgt_{x_\west}^-\geq 2h, \hgt_{x_\east}^-\geq 2h, \min_{x\in I}\hgt_x^+\leq 3h/2) \leq C\exp(- \e^2 N^{1/3}/2C)\,.
\end{align*}
Now consider the latter case, and recall that  $\mathcal E_{\geq h}$ from~\eqref{eq:entry-exit-conditioning-event}. 
As in the upwards spikiness, conditionally on $\mathcal E_{\geq h}$, the event $\min_{x\in I} \hgt_x^+<h$ implies the plus crossing $\co_v^+(I\times \llb 3h/2,2h\rrb)$, which is measurable with respect to the set $R \cap \enlarge^\downarrow$ where $\enlarge^\downarrow  = \enlarge_{r,h}^\downarrow(R)$  (again it was important that  the event $\mathcal E_{\geq h}$  was defined in terms of $\hgt^-$ rather than $\hgt^+$). Thus, by Proposition~\ref{prop:plus-enlargement}, for every $r$, 
\begin{align*}
\mu_{\lambda,N}^{\pm}\big(\hgt_{x_\west}^-,\hgt_{x_\east}^- & \geq 2h, \min_{x\in I}\hgt_x^+\leq 3h/2\big) &   \\
& \leq \mu_{\lambda, \enlarge^\downarrow}^{(\pm,2h)} (\co_v^+(I \times \llb 3h/2,2h\rrb)) + CN e^{8 c_\lambda (\mathcal M + \frac{r|I|+ r^2}{N})} e^{ - r/C} + Ce^{ - \e^2 N^{1/3}/2C}\,.
\end{align*}
Let us first consider the term coming from the cost of coupling to the strip-enlargement. If we take $r= N^{2/3}$, and use the upper bound $|I|\leq AN^{2/3}$, we see that as long as $\e$ was sufficiently small in $\beta,\lambda$, for large $N$, we have that 
\begin{align*}
CN e^{8 c_\lambda (\mathcal M + \frac{r|I|+r^2}{N})} e^{ - r/C} \leq Ce^{ - N^{2/3}/2C}\,.
\end{align*} 
Now consider the crossing probability in the strip-enlargement $\enlarge^\downarrow$. Observe that in the rectangle $\enlarge^\downarrow$ with $(\pm,2h)$ boundary conditions, if the event $\co_v^+(I\times \llb 3h/2,2h\rrb)$ occurs, then for its interface, necessarily $\min_{x}\hgt_x^+ \leq h$. Using~\eqref{eq:radon-nikodym-bound} to remove the tilt from the external field on $\enlarge^\downarrow$, we obtain 
\begin{align*}
\mu_{\lambda,\enlarge^\downarrow}^{(\pm,2h)} (\co_v^+(I\times \llb 3h/2,2h\rrb))\leq e^{ 8c_\lambda \frac{h |I|}{N}} \mu_{0,\enlarge^\downarrow}^{(\pm,2h)} (\min_x \hgt_x^+ <h)\,.
\end{align*}
Sending the bottom minus boundary conditions of $\enlarge^\downarrow$ to $-\infty$ and using (a vertically reflected) Corollary~\ref{cor:max-height-fluctuation-floor}, we obtain 
\begin{align*}
\mu_{\lambda,\enlarge_{r,2h}^\downarrow(\Lambda_I)}^{(\pm,2h)} (\co_v^+(I\times \llb 3h/2,2h\rrb)) & \leq C\exp \Big(  \frac{8 c_\lambda h |I|}{N} - \frac{h^2}{C|I|}\Big)  \leq C'\exp\big(- h^2 /C'|I| \big)
\end{align*} 
for some other $C'$ depending on $\beta,c_\lambda,A$. Here we used the fact that the above inequality holds as long as $\frac{h^2}{|I|}$ is bigger than a sufficiently large constant (depending on $\beta,c_\lambda, A$) times $\frac{h |I|}{N}$, 
which holds since $h\geq K\sqrt{|I|} \geq KN^{1/3}$, so long as $|I| \leq A N^{2/3}$.  
\end{proof}

\section{Local and global control on the tilt induced by the external field}\label{sec:global-tilt}
Recall from Section~\ref{sec:ideas-of-proofs},
that a key ingredient in our proofs is a bound on the difference in surface tensions in the presence and absence of an external field on both local and global scales.  
 As a warm up, to help the reader get accustomed to the nature of such arguments, and towards refining the understanding of the global behavior of $\mathscr I$, in Sections~\ref{subsec:global-tilt-estimates}--\ref{subsec:global-tilt-proof} we establish the a global distortion estimate removing logarithmic corrections in Lemmas 3--4 of \cite{Velenik}. We then adapt the argument to local scales between the critical width of $N^{2/3}$ and the maximal width of $N$ in Section~\ref{subsec:local-tilt}.
\subsection{Global control of the tilt.}\label{subsec:global-tilt-estimates}
Here prove the global distortion estimate~\eqref{eq:global-tilt-estimate-intro}. Recall that an \emph{admissible interface} is a $\Lambda^*$-admissible contour $\gamma$ with $\partial \gamma = \{\vsw,\vse\} := \{(-\frac12, -\frac 12), (N+\frac 12, -\frac 12)\}$. 
\begin{theorem}\label{thm:effect-of-global-tilt}
Fix $\beta>\beta_{c}$ and let $\lambda= \frac{c_\lambda}{N}$ for $c_\lambda >0$. There exist
$C_{1},C_{2}>0$ such that for every set of admissible interfaces~$\Gamma$,
\begin{align*}
\mu_{\lambda,N}^{\pm}(\Gamma)\leq & e^{C_{1}N^{1/3}}\mu_{0,N}^{\pm}(\Gamma)+e^{-C_{2}N}\,.
\end{align*}
\end{theorem}
In~\cite{Velenik}, a weaker version of Theorem~\ref{thm:effect-of-global-tilt} (with logarithmic corrections in the exponent of $e^{C_1 N^{1/3}}$) was proven. 
This followed by bounding the difference in surface tensions between the with and without-field measures, by constructing a set of interfaces carrying an $e^{ - \tilde O(N^{1/3})}$ fraction of the mass under $\mu_{0,N}^{\pm}$, and whose tilt from the field was at most exponential in $\tilde O(N^{1/3})$.  This set of interfaces in~\cite{Velenik} were pinned along a mesh of boundary points in $\partial_\south \Lambda_N$ separated by $N^{2/3}$; however, the pinning sustains polynomial corrections from the entropic repulsion at every mesh-point along $\partial_\south \Lambda_N$. We use a set of interfaces that stay within a window of their typical $N^{1/3}$ distance away from $\partial_\south \Lambda_N$ along the same mesh of columns to only pay constant costs along each mesh site.   

The following refines the surface tension distortion estimate of~\cite[Lemma 3]{Velenik}.  

\begin{proposition}\label{prop:global-tilt1}
Take $\beta>\beta_{c}$ and let $\lambda=\frac{c_{\lambda}}{N}$ for $c_\lambda>0$. There exists $C_1(\beta,c_{\lambda})>0$ 
such that 
\begin{align}\label{eq:sets-interface-tilt-intro1}
\frac{Z_{\pm,0,N}}{Z_{+,0,N}}\geq\frac{Z_{\pm,\lambda,N}}{Z_{+,\lambda,N}}\geq & \frac{Z_{\pm,0,N}}{Z_{+,0,N}}e^{-C_1 N^{1/3}}\,.
\end{align}
\end{proposition}

Before proving the above proposition we quickly show how to finish the proof of Theorem \ref{thm:effect-of-global-tilt}.

\begin{proof}[\textbf{\emph{Proof of Theorem  \ref{thm:effect-of-global-tilt}}}]
Consider any set of admissible interfaces $\Gamma$. Letting $Z_{\pm,\lambda,N}(\gamma)$ be the restricted partition function to configurations with interface $\gamma$, there exists $C(\beta,c_\lambda)>0$ such that 
\begin{alignat*}{1}
\mu_{\lambda,N}^{\pm}(\Gamma) =  \frac{Z_{+,\lambda,N}}{Z_{\pm,\lambda,N}}\frac{\sum_{\gamma \in\Gamma}Z_{\pm,\lambda,N}(\gamma )}{Z_{+,\lambda,N}}
& \overset{\eqref{eq:low-temp-interface-probability}}{=} \frac{Z_{+,\lambda,N}}{Z_{\pm,\lambda,N}}\sum_{\gamma \in\Gamma}w(\gamma )\frac{Z_{+,\lambda,\Lambda^{+}(\gamma)}Z_{-,\lambda,\Lambda^{-}(\gamma)}}{Z_{+,\lambda,N}}\\
& \overset{\eqref{eq:sets-interface-tilt-intro1}}{\leq}  e^{CN^{1/3}}\frac{Z_{+,0,N}}{Z_{\pm,0,N}}\sum_{\gamma \in\Gamma}w(\gamma)\frac{Z_{+,\lambda,\Lambda^{+}(\gamma)}Z_{-,\lambda, \Lambda^{-}(\gamma)}}{Z_{+,\lambda}}\,,\end{alignat*}
where we recall that $w(\gamma) = e^{ - 2 \beta |\gamma|}$ and where the inequality followed from Proposition ~\ref{prop:global-tilt1}. 
We then use the following inequality from~\cite{Velenik} to move from the summands above to $q_N(\gamma)$ (i.e., in the absence of an external field), from which we will be able to leverage the random-line representation. 
 
\begin{lemma}\cite[Lemma 4]{Velenik}\label{lem:effect-of-global-tilt}
Fix $\beta>\beta_{c}$ and let $\lambda = \frac{c_\lambda}{N}$. There exist $C(\beta,c_\lambda)<\infty,C'(\beta,c_\lambda)>0$ such that for every set of admissible interfaces $\Gamma$, 
\begin{align*}
\frac{Z_{+,0,N}}{Z_{\pm,0,N}} \sum_{\gamma\in\Gamma}w(\gamma)\frac{Z_{+,\lambda,\Lambda^{+}(\gamma)}Z_{-,\lambda,\Lambda^{-}(\gamma)}}{Z_{+,\lambda,N}}\leq & e^{C\lambda N}\frac{Z_{+,0,N}}{Z_{\pm,0,N}} \sum_{\gamma\in\Gamma}q_N(\gamma)+e^{-C' N}\,.
\end{align*}
\end{lemma}
Now, combining ~\eqref{eq:random-line-representation-partition-functions}--\eqref{eq:random-line-two-point-correlations}, we recall that
\begin{align*}
\mu_{0,N}^{\pm}(\Gamma) = \frac{Z_{+,0,N}}{Z_{\pm,0,N}} \sum_{\gamma \in \Gamma} q_N(\gamma)\,;
\end{align*}
combined with Lemma~\ref{lem:effect-of-global-tilt} and the bound on $\mu_{\lambda,N}^{\pm} (\Gamma)$ above, we deduce Theorem \ref{thm:effect-of-global-tilt}. 
\end{proof}

\subsection{Proof of Proposition \ref{prop:global-tilt1}}\label{subsec:global-tilt-proof}
The first part of the proof follows~\cite{Velenik} exactly. Notice that by rearranging, 
 the first inequality is implied by monotonicity as 
\begin{align*}
\frac{Z_{+,\lambda,N}}{Z_{+,0,N}}=\langle e^{\lambda\sum_{x}\sigma_{x}}\rangle_{+,0,N}\geq\langle e^{\lambda\sum_{x}\sigma_{x}}\rangle_{\pm,0,N}= & \frac{Z_{\pm,\lambda,N}}{Z_{\pm,0,N}}\,,
\end{align*}
where we recall from Section \ref{sec:preliminaries}, that $\langle\cdot\rangle_{\eta,\lambda,N}$ denotes expectation with
respect to $\mu_{\lambda,N}^{\eta}$.

We proceed to the second inequality.
Let $\mathcal{H}=3K N^{1/3}$, for some $K$ to be chosen later to be sufficiently large. 
Consider the strip $\senlarge_{\mathcal{H}}=\llb 0,N\rrb \times \llb 0,\mathcal H \rrb$. For notational brevity, since our interfaces and contour collections more generally are always edge-subsets of the dual lattice $E_{(\mathbb Z^2)^*}$, we write $\gamma \subset R$ to indicate that $\gamma \subset E_{R^*}$, and moreover, that it is $R^*$-admissible. We have
\begin{align}\label{eq:part-function-random-line-global}
Z_{\pm,\lambda,N}\overset{\eqref{eq:low-temp-interface-probability}}{=}\sum_{\substack{\gamma:\vsw \to \vse \\ \gamma \subset \Lambda_N}} e^{-2\beta|\gamma|}Z_{+,\lambda,\Lambda^{+}}Z_{-,\lambda,\Lambda^{-}}\geq & \sum_{\substack{\gamma:\vsw\to\vse \\ \gamma\subset\senlarge_{\mathcal{H}}}}e^{-2 \beta|\gamma|}Z_{+,\lambda,\Lambda^{+}}Z_{-,\lambda,\Lambda^{-}}\,,
\end{align}
where here and throughout,  $\Lambda^{\pm}=\Lambda^{\pm}(\gamma)$
when $\gamma$ is an admissible interface and the dependence on $\gamma$ is contextually understood. Then, for every fixed $\gamma\subset\senlarge_{\mathcal{H}}$, write
\begin{align}\label{eq:keystep-global}
Z_{+,\lambda,\Lambda^{+}}Z_{-,\lambda,\Lambda^{-}}\geq  Z_{+,\lambda,\Lambda^{+}}Z_{-,-\lambda,\Lambda^{-}}e^{-2\lambda|\Lambda^{-}|} 
&\geq  \frac{Z_{+,\lambda,\Lambda^{+}}Z_{+,\lambda,\Lambda^{-}}}{Z_{+,0,\Lambda^{+}}Z_{+,0,\Lambda^{-}}}Z_{+,0,\Lambda^{+}}Z_{+,0,\Lambda^{-}}e^{-2\lambda\mathcal{H}N} \nonumber \\
& \geq  \frac{Z_{+,\lambda,N}}{Z_{+,0,N}}Z_{+,0,\Lambda^{+}}Z_{+,0,\Lambda^{-}}e^{-2\lambda\mathcal{H}N}\,,
\end{align}
where the first inequality follows from a direct computation on replacing $\lambda$ by $-\lambda$, the second inequality uses that by symmetry $Z_{-,-\lambda,\Lambda^{-}}=Z_{+,\lambda,\Lambda^{-}}$, and the third inequality uses the fact that $\frac{Z_{+,\lambda,\Lambda^{+}}Z_{+,\lambda,\Lambda^{-}}}{Z_{+,\lambda,N}}$
is exactly the probability of all vertices of $\partial_\interior \Lambda^+\cup \partial_\interior \Lambda^-$ being
plus (where we recall these are the sites frozen by the interface $\gamma$), which is larger when $\lambda>0$ than under
$\lambda= 0$ by monotonicity. Altogether, dividing both sides of~\eqref{eq:part-function-random-line-global} by $Z_{+,\lambda,N}$ and applying~\eqref{eq:keystep-global}
\begin{align}\label{eq:global-part-function-ratio-random-line}
\frac{Z_{\pm,\lambda,N}}{Z_{+,\lambda,N}} & \geq e^{-6c_\lambda K N^{1/3} }\sum_{\gamma:\vsw\to \vse, \gamma \subset \senlarge_{\mathcal{H}}}e^{-2 \beta|\gamma|}\frac{Z_{+,0,\Lambda^{+}}Z_{+,0,\Lambda^{-}}}{Z_{+,0,N}} \nonumber \\
& \geq e^{-6 c_\lambda K N^{1/3}}\sum_{\gamma:\vsw\to \vse, \gamma \subset \senlarge_{\mathcal{H}}}q_{N}(\gamma)\,.
\end{align}
We now control this last sum, showing that the contribution of $\sum_{\gamma:\vsw\to \vse, \gamma\subset\senlarge_{\mathcal{H}}}q_{N}(\gamma)$
is not too far (at most a multiplicative factor of some $e^{-cN^{1/3}}$
for a $c$ that does not depend on $K$ in $\mathcal{H}$) from $\frac{Z_{\pm,0,N}}{Z_{+,0,N}}=\sum_{\gamma:\vsw\to \vse, \gamma\subset \Lambda_N}q_{N}(\gamma) = \langle \sigma_{\vsw} \sigma_{\vse} \rangle_{\Lambda_N}^*$. 
From here, our proof diverges from that of~\cite{Velenik}. 

To show this, let us split $\llb -\frac 12,N+\frac 12\rrb$ into $M=N^{1/3}$ many
segments of length $N^{2/3}$, delineated by the $x$-coordinates
$-\frac 12=x_{0},x_{1},x_{2},...,x_{M}=N+\frac 12$. For each $i=1,...,M-1$, define
the window  
$${\sf W}_i:= \{x_i\} \times [ KN^{1/3},  2KN^{1/3}]\,.$$
Let $\Gamma_{\mathcal{H}}$ be the set of interfaces $\gamma:\vsw\to\vse:\gamma\subset S_{\mathcal{H}}$
having the additional property that 
\begin{itemize}
\item For every $i=1,...,M$, the
set $\gamma\cap(\{x_{i}\}\times\mathbb{R})$ is a singleton (in $(\mathbb Z^2)^*$) and is in ${\sf W}_i$.
\end{itemize}
(See Figure~\ref{fig:global-tilt-corridor}.) The property that the intersection $\gamma \cap (\{x_i\}\times \mathbb R)$ is a singleton is important for us to apply~\eqref{eq:concatenation} to decompose the weight of $\gamma$ into its stretches between ${\sf W}_i$ and ${\sf W}_{i+1}$.

 By positivity of
$q_{N}(\gamma)$, we now bound $\sum_{\gamma\subset S_{\mathcal{H}}}q_{N}(\gamma)$
from below by $\sum_{\gamma\in\Gamma_{\mathcal{H}}}q_{N}(\gamma)$;
enumerating over the dual-vertices $w_i^*$ at which $\gamma$ intersects
${\sf W}_i$, we express
\begin{align*}
\sum_{\gamma \in \Gamma_{\mathcal H}} q_N (\gamma) = \sum_{w_i^* \in {\sf W}_i} \, \sum_{\substack{\gamma: \vsw\to w_1^*\to \cdots  \to w_{M-1}^* \to \vse \\ \gamma\in \Gamma_{\mathcal H}}}  q_N(\gamma)\,.
\end{align*}
Now by~\eqref{eq:concatenation} (using crucially that $|\gamma \cap (\{x_i\}\times \mathbb R)|=1$ for all $i$), the above expression is at least 
\begin{align}\label{eq:global-tilt-decomposition}
\Big(\sum_{w_1^* \in {\sf W}_1} \omega (\Gamma_{\mathcal H}^\west ; w_1^*)\Big)\Big( \sum_{w_i^*\in {\sf W}_i: 2\leq i \leq M-1}  \prod_{i=1}^{M-2} \omega (\Gamma_{\mathcal H}^{i}; w_i^*,w_{i+1}^*)\Big) \Big( \omega(\Gamma_{\mathcal H}^{\east}; w^*_{M-1})\Big)\,,
\end{align}
where the weights of the various stretches of the interface are given by 
\begin{align*}
\omega(\Gamma_{\mathcal H}^{\west}; w_1^*) &  := \sum_{\substack{\gamma_\west: \vsw \to w_1^* \\ \gamma_\west \subset \llb 0,x_1\rrb \times \llb 0, \mathcal H\rrb \\ \gamma_\west\in  \nbt}} q_N(\gamma_\west)\,, \qquad  \qquad \omega(\Gamma_{\mathcal H}^{\east}; w^*_{M-1}):=  \sum_{\substack{\gamma_\east: w^*_{M-1} \to \vse \\ \gamma_\east \subset \llb x_{M-1},N\rrb \times \llb 0, \mathcal H \rrb \\ \gamma_\east \in \nbt}} q_N(\gamma_\east)\,, \\
\omega(\Gamma_{\mathcal H}^i; w_i^*,w^*_{i+1}) &  := \sum_{\substack{\gamma_i: w_i^* \to w_{i+1}^* \\ \gamma_i \subset \llb x_i , x_{i+1}\rrb \times \llb 0, \mathcal H\rrb \\ \gamma_i \in \nbt}} q_N(\gamma_i)\,.
\end{align*}
(Recall that for $\gamma$ with $|\partial \gamma|=\{v^*,w^*\}$, we use $\gamma\in \nbt$ to denote the set of non-backtracking $\gamma$ i.e., whose intersection with the columns through $v^*,w^*$ are exactly the two vertices $\{v^*,w^*\}$.)

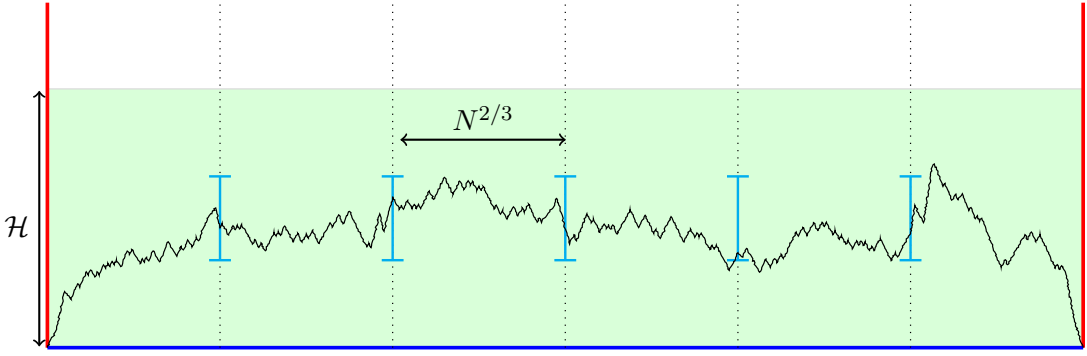
\begin{figure}
\centering
\begin{tikzpicture}
\node at (0,0){
\includegraphics[width=.85\textwidth]{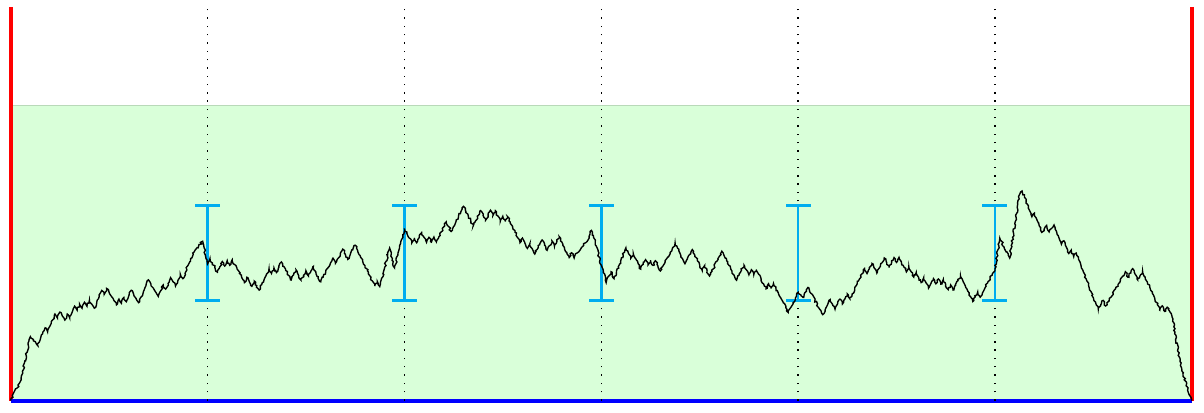}
};
\draw[<->, thick]  (-2.2,.5)--(0,.5);
\draw[<->, thick]  (-7,-2.25)--(-7,1.15);

\node at (-1.1,.8) {$N^{2/3}$};
\node at (-7.3,-.65) {$\mathcal H$};

\end{tikzpicture}
\caption{The interfaces in $\Gamma_{\mathcal H}$ are those that are confined to $\senlarge_{\mathcal H}$ (shaded, green), that intersect each mesh column (black, dashed) at a singleton in ${\sf W}_i$ (cyan).}\label{fig:global-tilt-corridor}
\end{figure}

Let us consider each of these terms individually. By~\eqref{eq:q-monotonicity-infinite-domain}, we can lower bound each of the weights above by their analogues with $q_N$ replaced by $q_{\mathbb Z^2}$. For every $i$, for every $w_i^*\in {\sf W}_i$ and $w^*_{i+1} \in {\sf W}_{i+1}$, (dropping the criterion that $\gamma_i:w_i^*\to w_{i+1}^*$ from the notation), 
\begin{align}\label{eq:global-tilt-interior-weights-1}
\omega(\Gamma_{\mathcal H}^i; w^*_i, w^*_{i+1}) &  \geq \sum_{\gamma_i \subset \llb x_i , x_{i+1}\rrb \times \llb 0, \mathcal H\rrb, \gamma_i \in \nbt} q_{\mathbb Z^2}(\gamma_i) \nonumber \\
& = \langle \sigma_{w_i^*} \sigma_{w_{i+1}^*}\rangle_{\mathbb Z^2}^* -  \sum_{\substack{ \gamma_i \subset [ x_i,x_{i+1} ] \times \mathbb R, \gamma_i \in \nbt \\ \gamma_i \not \subset \mathbb R \times [ 0 , \mathcal H ]}} q_{\mathbb Z^2}(\gamma_i)- \sum_{\substack{\gamma_i \subset [ x_i,x_{i+1} ] \times \mathbb R \\ \gamma_i \notin \nbt}} q_{\mathbb Z^2}(\gamma_i)\,.
\end{align}
By Proposition~\ref{prop:no-backtracking-domain} the ratio of the third term to the two-point function is at most $1-\e_\beta$ for some $\e_\beta>0$; by Corollary~\ref{cor:max-height-fluctuation-not-straight}, the same ratio for the second term is at most $C e^{- \mathcal H^2 / C N^{2/3}}$, which is smaller than $\e_\beta/2$ for $K$ large. Putting these together,  we obtain for some other $\e_\beta>0$, 
\begin{align}\label{eq:global-tilt-interior-weights-2}
\omega(\Gamma_{\mathcal H}^i; w^*_i, w^*_{i+1}) &  \geq \e_\beta \langle \sigma_{w_i^*} \sigma_{w^*_{i+1}} \rangle_{\mathbb Z^2}^*\,.
\end{align}
We next consider the edge terms, $\omega(\Gamma_{\mathcal H}^{\west}; w^*_1)$ and $\omega(\Gamma_{\mathcal H}^\east ; w^*_{M-1})$. By symmetry it suffices to bound the former. 
We can lower bound the weight, by the same quantity, but only summing over interfaces $\gamma_\west: \vsw \to (KN^{1/3},KN^{1/3}) \to w_1^*$ (where $KN^{1/3}$ is understood as the closest half-integer to that number) with $\gamma_\west \subset \llb 0,x_1\rrb \times \llb 0,\mathcal H \rrb, \gamma_\west \in \nbt$; call this set of interfaces $\tilde \Gamma_{\mathcal H}^\west$. In that case, we can again use~\eqref{eq:concatenation}, together with the finite energy property~\eqref{eq:weight-function-lb} to see that this is at least
\begin{align*}
\sum_{\substack{\gamma_\west^{1}: \vsw \to (KN^{1/3},KN^{1/3}) \\ \gamma_\west^{1} \in \tilde \Gamma_{\mathcal H}^\west}}  q_{\mathbb Z^2}(\gamma_\west^{1}) & \!\!\! \sum_{\substack{\gamma_0^{2}:(KN^{1/3},KN^{1/3})\to w^*_1 \\ \gamma_0^{2} \in \tilde \Gamma_{\mathcal H}^\west}} q_{\mathbb Z^2}(\gamma_0^{2}) \geq e^{-  C K N^{1/3}} \!\!\! \sum_{\substack{\gamma_0^{2}:(KN^{1/3},KN^{1/3})\to w^*_1 \\ \gamma_0^{2} \in \tilde \Gamma_{\mathcal H}^\west}} q_{\mathbb Z^2}(\gamma_0^{2})\,,
\end{align*}
for some $C(\beta,c_\lambda)>0$. 
In order to bound the latter sum, we reason as in~\eqref{eq:global-tilt-interior-weights-1}--\eqref{eq:global-tilt-interior-weights-2}, to obtain 
\begin{align*}
\omega(\Gamma_{\mathcal H}^\west; w^*_1) \geq \e_\beta  e^{ - CKN^{1/3}} \langle \sigma_{(KN^{1/3},KN^{1/3})}\sigma_{w_1^*}\rangle_{\mathbb Z^2}^*\,.
\end{align*}
Minimizing the above, and~\eqref{eq:global-tilt-interior-weights-2} over choices of $w^*_i \in {\sf W}_i$, and using~\eqref{eq:point-to-point-l-r}, we find for some $c_\beta>0$, 
\begin{align*}
\min_{w^*_i\in {\sf W}_i,w^*_{i+1}\in {\sf W}_{i+1}} \omega(\Gamma_{\mathcal H}^{i} ; w^*_i, w^*_{i+1}) & \geq \frac{c_{\beta} e^{ - K^2 \kappa_\beta /2}}{N^{1/3}} \exp ( - \tau_{\beta}(0) (x_{i+1}- x_i))\,, \\ 
\min_{w^*_1\in {\sf W}_1} \omega(\Gamma_{\mathcal H}^{\west} ; w^*_1)  & \geq \frac{c_\beta e^{ - CKN^{1/3}}}{N^{1/3}} \exp (- \tau_{\beta}(0) x_1)\,, \\
\min_{w^*_{M-1}\in {\sf W}_{M-1}} \omega(\Gamma_{\mathcal H}^{\east} ; w^*_{M-1}) & \geq \frac{c_\beta e^{ - CKN^{1/3}}}{N^{1/3}} \exp (- \tau_{\beta}(0) (N- x_{M-1}))\,.
\end{align*}
Plugging in these bounds into~\eqref{eq:global-tilt-decomposition}, we obtain  
\begin{align*}
\sum_{\gamma\in \Gamma_{\mathcal{H}}}q_{N}(\gamma) & \geq  (KN^{1/3})^{M-1}\Big(\frac{c_{\beta}e^{-K^{2}\tau_{\beta}(0)/2}}{N^{1/3}}\Big)^{M-2}\Big(\frac{c_{\beta}e^{- CKN^{1/3}}}{N^{1/3}}\Big)^{2}\cdot e^{-\tau_{\beta}(0)N}\\
& \geq  N^{- 1/3} K^{M-1}c_{\beta}^{M+1}e^{- K^{2}\tau_{\beta}(0) N^{1/3} -2 C K N^{1/3}}\cdot e^{-\tau_{\beta}(0)N}\,.
\end{align*}
At this point, we recall from~\eqref{eq:q-monotonicity-infinite-domain} and~\eqref{eq:floor-point-to-point-bounds} that (see also pg.18 of~\cite{Velenik})   
\begin{align*}
\frac{Z_{\pm,0,N}}{Z_{+,0,N}} = \langle \sigma_{\vsw} \sigma_{\vse}\rangle^*_{\Lambda_N} \leq \langle \sigma_{\vsw} \sigma_{\vse}\rangle^*_{s.i.} \leq \frac{K_2}{N^{3/2}}e^{- \tau_\beta(0) N}\,.
\end{align*}
As such, recalling that $M= N^{1/3}$, we obtain, for some other $c(\beta,c_\lambda),C(\beta,c_\lambda)>0$, the inequality  
\begin{align*}
 \sum_{\gamma\subset \senlarge_{\mathcal{H}}}q_{N}(\gamma) & \geq  N^{7/6} \exp\Big(- C N^{1/3} \big(K^{2}\tau_{\beta}(0)+ CK +\log (K c_{\beta})\big)\Big)\cdot\frac{Z_{\pm,0,N}}{Z_{+,0,N}}\,,
\end{align*}
which when plugged in to~\eqref{eq:global-part-function-ratio-random-line} yields 
\begin{align*}
\frac{Z_{\pm,\lambda,N}}{Z_{+,\lambda,N}}\geq & \frac{Z_{\pm,0,N}}{Z_{+,0,N}}e^{-A N^{1/3}}\,,\qquad\mbox{for}\qquad A> C( c_\lambda K + K^{2}\tau_{\beta}(0)+K +\log (K c_{\beta}))\,,
\end{align*}
where we observe that this latter quantity only depends on $\beta,c_\lambda$.  \qed

\subsection{Local control of the tilt induced by an external field}\label{subsec:local-tilt}
We now extend the results of the previous section to local scales necessary for our proofs of Theorems 
\ref{thm:one-point-tail-bounds} 
 and 
\ref{thm:max-tightness}.
Before proceeding, recall that $\Lambda_{L,N}=\llb 0, L\rrb \times \llb 0, N\rrb$ and recall from~\eqref{eq:enlargement}--\eqref{eq:tshape-enlargement} that given $r$ and $H$, the T-shaped and strip enlargements of $\Lambda_{L,N}$ 
 are given by
\begin{align*}
\tshape:=\tshape_{r,H}(\Lambda_{L,N}) := \Lambda_{L,N} \cup \enlarge^\uparrow\qquad \mbox{where}\qquad \enlarge^\uparrow=\enlarge_{r,H}^\uparrow(\Lambda_{L,N}) := \llb -r, L+r\rrb \times \llb H,N\rrb\,.
\end{align*}
(Note that in the case $r=0$, $\tshape = \Lambda_{L,N}$.) The main result of this section establishes the key estimate on the change in surface tension at local scales (on $\tshape$) analogous to Proposition~\ref{prop:global-tilt1}.

\begin{proposition}\label{prop:local-tilt-enlargements}
Let $\beta>\beta_c$ and let $\lambda = \frac{c_\lambda}{N}$ for $c_\lambda>0$. There exist constants $K$ and $\eta$ such that the following holds. For every $H = RN^{1/3}$ with   $K \le R \le K^{-1}N^{2/3}$, every $L\geq K \sqrt{R}N^{1/3}$ and every $r\le K^{-1}H^2$,
\begin{align*}
\frac {Z_{(\pm,2H),0, \tshape}}{Z_{+,0,\tshape}} \geq \frac{Z_{(\pm,2H),\lambda,\tshape}}{Z_{+,\lambda,\tshape}} \geq \frac {Z_{(\pm,2H),0,\tshape}}{Z_{+,0,\tshape}}\exp\Big(-C_{\beta,\lambda}\Big(LN^{-2/3} + \Big({H N^{-1/3}}\Big)^{3/2}\Big)\,.
\end{align*}
\end{proposition}
We will obtain such a bound by restricting attention to the following set of interfaces: starting from a height $H$ possibly $\gg N^{1/3}$, the interface quickly (in distance $\sqrt R N^{2/3}$) descends down to height $O(N^{1/3})$ and stays at that height through the bulk of the $\tshape$, before returning to height $H$. 
\begin{proof}
As in the proof of Proposition \ref{prop:global-tilt1}, the first inequality is obtained by monotonicity. We turn to the second inequality. Fix $K$ to be a sufficiently large constant (as compared to constants depending on $\beta,c_\lambda$, but independent of $L,H$) and  $M$ to be chosen later ({it will be useful to think of it as being of order $C_H N^{2/3}$}). With these, to encode the class of interfaces as in Figure \ref{fig:local-tilt-intro}  define the following corridor $\corridor= \corridor_{r,H} \subset \tshape$  to which we will  restrict the interface: 
\begin{align*}
\corridor: = (\underbrace{\llb -r, M\rrb \times \llb 0,3H\rrb}_{\lshape_\west}) \cup (\underbrace{\llb M,L-M\rrb \times \llb 0,3KN^{1/3}\rrb}_{{\sf S}_{\textsc{mid}}}) \cup (\underbrace{\llb L-M, L+r\rrb \times \llb 0,3H\rrb}_{\lshape_\east}) \bigcap \tshape.
\end{align*}
where $\lshape_\west$ and $\lshape_\east$ are (upside-down) L-shaped regions on either side of $\tshape$, (see Figure~\ref{fig:local-tilt-corridor}). 

 An admissible interface $\gamma$ on $\tshape$ has 
$$\partial \gamma = \{v_{\west}^*,v_{\east}^*\} := \{(-r-\tfrac 12, 2H - \tfrac 12), (L+r+\tfrac 12, 2H- \tfrac 12)\}\,.$$
For such an interface $\gamma$, we use $\tshape^+= \tshape^+(\gamma)$ and $\tshape^-= \tshape^-(\gamma)$ to denote the subsets of $\tshape$ above and below the interface $\gamma$, respectively.  For ease of notation, we denote by $Z_{\pm,\lambda,\tshape}$, the partition function $Z_{(\pm,2H),\lambda,\tshape}$ corresponding to the $(\pm,2H)$ boundary conditions.
Now by \eqref{eq:low-temp-interface-probability}
\begin{align}\label{eq:localexp1}
Z_{\pm,\lambda,\tshape}  \geq \sum_{\substack{\gamma: v_\west^* \to\ve \\ \gamma \subset \corridor}}e^{ - 2\beta|\gamma|} Z_{+,\lambda,\tshape^+} Z_{-,\lambda,\tshape^-}
\end{align}
We can then write for every $\gamma \subset \corridor,$ by the same argument in \eqref{eq:keystep-global},  
\begin{align*}
Z_{+,\lambda,\tshape^+} Z_{-,\lambda,\tshape^-} 
& \geq \frac{Z_{+,\lambda,\tshape}}{Z_{+,0,\tshape}} Z_{+,0,\tshape^+} Z_{+,0,\tshape^-} e^{ - 2\lambda (6H(r+M)+3KLN^{1/3})}\,.
\end{align*}
Dividing both sides out by $Z_{+,\lambda,\tshape}$ and arguing as in~\eqref{eq:global-part-function-ratio-random-line} yields for some universal constant $C$, 
\begin{align}\label{eq:lowerbnd123}
\frac{Z_{\pm,\lambda,\tshape}}{Z_{+,\lambda,\tshape}} \geq e^{ - C\lambda (H(r+M)+K LN^{1/3})}\sum_{\gamma:\vw \to \ve, \gamma \subset \corridor} q_{\tshape}(\gamma)\,.
\end{align}
We now turn to this last sum, lower bounding its contribution as a fraction of $Z_{\pm,0,\tshape}/Z_{+,0,\tshape}$. 
Towards this, let us split the strip ${\sf S}_{\textsc{mid}}$ by dividing $\llb M-\frac 12,L-M+\frac 12\rrb$ into $\mathcal N= (L-2M+1)/(N^{2/3})$ many segments, each of length $N^{2/3}$ delineated by $M-\frac 12 = x_1,...,x_{\mathcal N} = L-M+\frac 12$, and for each of these define the window ${\sf W}_i := x_i \times [ KN^{1/3}, 2KN^{1/3}]$. We also define the boundary windows by $${\sf W}_\west := \{-\tfrac 12\}\times [ 2H - K\sqrt r, 2H+K\sqrt r] \qquad \mbox{and}\qquad {\sf W}_\east = \{L+\tfrac 12\} \times [ 2H - K\sqrt r, 2H+K\sqrt r]\,.$$ 
 and let $\Gamma_{\corridor}$ be the set of $\gamma\subset \corridor:\vw \to \ve$ having the additional properties that 
\begin{itemize}
\item For every $i=1,...,\mathcal N,$ the intersection $\gamma\cap (\{x_i\}\times \mathbb R)$ is a singleton in ${\sf W}_i$. 
\item The intersection $\gamma \cap (\{-\frac 12\}\times \mathbb R)$ is a singleton in ${\sf W}_\west$ and the intersection $\gamma\cap (\{L+\frac 12\}\times \mathbb R)$ is a singleton in ${\sf W}_\east$. 
\end{itemize}
\begin{figure}
\centering
\begin{tikzpicture}
\node at (0,0){
\includegraphics[width=.85\textwidth]{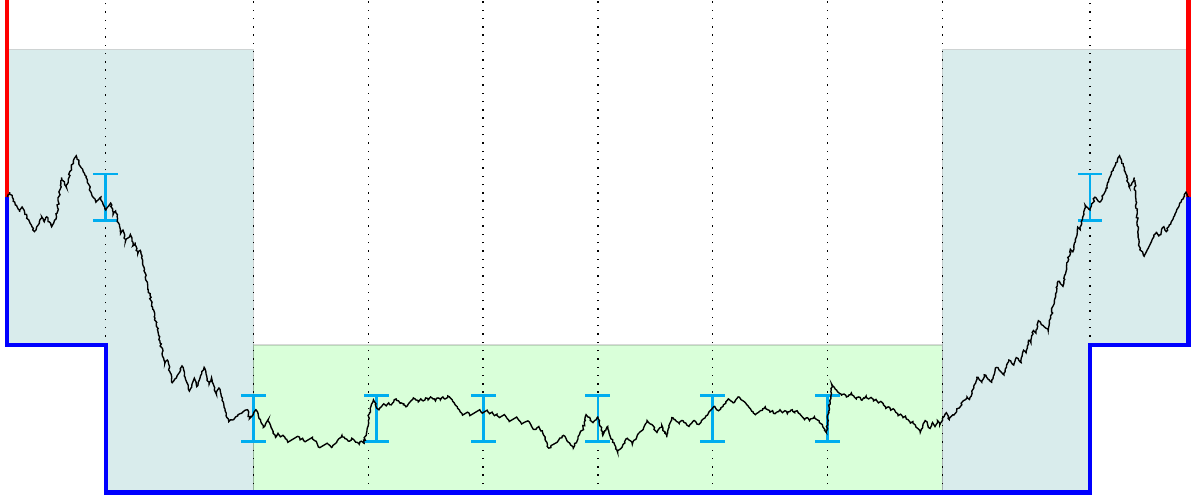}
};

\draw[<->, thick]  (-1.35,.5)--(0,.5);
\draw[<->, thick]  (-7,-2.83)--(-7,.7);
\draw[<->, thick]  (5.9,-2.83)--(5.9,-1.75);

\node at (-.75,.8) {$N^{2/3}$};
\node at (-7.3,-.95) {$2 H$};
\node at (6.7,-2.2) {$O(N^{1/3})$};

\end{tikzpicture}
\caption{The interfaces in $\Gamma_{\corridor}$ are those confined to $\corridor$---the union of ${\sf L}_\west, {\sf L}_\east$ (dark green, shaded) and $\senlarge_{\textsc{mid}}$ (green, shaded)---that intersect the mesh columns (black dashed) at singletons inside the windows $\window_\west, \window_1,...,\window_{\mathcal N},\window_\east$ (cyan).}\label{fig:local-tilt-corridor}
\end{figure}

We can now bound the sum over interfaces $\gamma \subset \corridor$ by a sum over $\gamma \in \Gamma_{\corridor}$. Enumerating over the dual-vertices $w^*$ at which $\gamma$ intersects ${\sf W}_\east, {\sf W}_\west$ and $({\sf W}_i)_{1\le i\le \cN}$,  we obtain 
\begin{align*}
\sum_{\gamma \in \Gamma_\corridor} q_{\tshape}(\gamma) & = \sum_{w_\west^* \in {\sf W}_\west, w^*_\east\in {\sf W}_\east} \sum_{w^*_i \in {\sf W}_i} \sum_{\gamma: \vw \to w_\west^* \to w_1^* \to \cdots \to w_{\mathcal N}^* \to w_\east^*\to \ve\,,\,\gamma\in \Gamma_{\corridor}} q_\tshape (\gamma) 
\end{align*}
Now by \eqref{eq:concatenation},  the above expression has the following lower bound, 
\begin{align*}  
&  \sum_{\ww \in {\sf W}_\west} \omega_\tshape(\Gamma_\corridor^\west; \ww) \times \sum_{w_1^*\in \window_1} \omega_\tshape(\Gamma_\corridor^\downarrow ; \ww,w_1^*) \times \sum_{w_i^*\in {\sf W}_i: 2\leq i\leq \mathcal N-1} \prod_{i=1}^{\mathcal N-1}\omega_{\tshape}(\Gamma^i_\corridor; w_i^*,w^*_{i+1})\\ 
& \quad\qquad\qquad \times \sum_{\we \in {\sf W}_{\mathcal N}} \omega_\tshape(\Gamma_\corridor^\uparrow ; w^*_{\mathcal N},\we) \times  \omega_\tshape(\Gamma_\corridor^\east; \we)\,, 
\end{align*}
where the weights of the various stretches of the interface are given by 
\begin{align*}
\omega_\tshape (\Gamma_\corridor^\west ; \ww) & := \sum_{\substack{\gamma_\west\in \nbt: \vw\to \ww\,, \\ \gamma_\west \subset \llb -r,0\rrb\times \llb H,3H\rrb}} q_{\tshape}(\gamma_\west)\,, \quad  &  \omega_\tshape (\Gamma_\corridor^\east ; \we) := \sum_{\substack{\gamma_\east\in \nbt: \we \to\ve\ \\ \gamma_\east \subset \llb L,L+r\rrb\times \llb H,3H\rrb}} q_{\tshape}(\gamma_\east)\,, \\ 
\omega_\tshape(\Gamma^\downarrow_\corridor ;\ww,w_1*) & := \sum_{\substack{\gamma_0\in \nbt: \ww\to w_1^*\\ \gamma_0 \subset \llb 0,x_1\rrb \times \llb 0,3H\rrb}} q_{\tshape}(\gamma_0)\,, \quad \quad  & \omega_\tshape(\Gamma^\uparrow_\corridor ;w_{\mathcal N}^*,\we)  := \sum_{\substack{\gamma_{\mathcal N} \in \nbt: w^*_{\mathcal N} \to \we \\ \gamma_{\mathcal N} \subset \llb x_{\mathcal N},L\rrb \times \llb 0,3H\rrb}} q_{\tshape} 
(\gamma_{\mathcal N})\,, \\
\omega_\tshape(\Gamma^i_\corridor; w_i^*,w^*_{i+1}) & := \sum_{\substack{\gamma_i\in \nbt: w_i^*\to w_{i+1}^* \\ \gamma_i\subset \llb x_i,x_{i+1}\rrb \times \llb 0,3KN^{1/3}\rrb}} q_\tshape(\gamma_i)\,. 
\end{align*}
The weights of $\Gamma_{\corridor}^i$ are bounded exactly as in~\eqref{eq:global-tilt-interior-weights-1}--\eqref{eq:global-tilt-interior-weights-2} from which we deduce from~\eqref{eq:point-to-point-l-r},
\begin{align*}
\omega_{\tshape}(\Gamma_\corridor^i; w^*_i,w^*_{i+1}) \geq \varepsilon_\beta \langle \sigma_{w_i^*} \sigma_{w_{i+1}^*}\rangle_{\mathbb Z^2}^* \geq \frac{\varepsilon_{\beta} e^{ - K^2 \kappa_\beta}}{N^{1/3}} e^{ - \tau_0 (\beta) N^{2/3}}\,.
\end{align*}
Similarly, as long as $r \leq K^{-1} H^2$, we have that for every $\ww\in {\sf W}_\west$, 
\begin{align*}
\omega_{\tshape}(\Gamma_\corridor^\west; \ww)  \geq \varepsilon_\beta \langle \sigma_{\vw} \sigma_{\ww}\rangle_{\mathbb Z^2}^* \geq \frac{\varepsilon_{\beta} e^{ - K^2 \tau_0(\beta)}}{\sqrt r} e^{ - r \tau_0 (\beta) }\,,
\end{align*}
and, by symmetry, the same inequality holds for $\omega_\tshape (\Gamma_\corridor^\east ; \we)$ for every $y_\east \in \window_\east$. It remains to control the stretches of the interface where it decreases from height $2H$ down to $KN^{1/3}$ and back up to $2H$. For these, we recall that for every $\ww\in \window_\west$ and $w_1^* \in \window_1$, arguing similarly to~\eqref{eq:global-tilt-interior-weights-1}--\eqref{eq:global-tilt-interior-weights-2}, we have by Proposition~\ref{prop:no-backtracking-domain} and Corollary~\ref{cor:max-height-fluctuation-not-straight}, that as long as $K$ is sufficiently large (depending only on $\beta$),  the restrictions $\gamma_0 \in \nbt, \gamma_0 \subset  \llb 0,x_1 \rrb \times \llb 0,3H\rrb$ cost a constant factor $\varepsilon_\beta>0$, and 
\begin{align*}
\omega_{\tshape} (\Gamma_\corridor^\downarrow ; \ww, w_1^*)  \geq \varepsilon_\beta \langle \sigma_{\ww}\sigma_{w_1^*}\rangle_{\mathbb Z^2}^*\,,
\end{align*}
which by~\eqref{eq:point-to-point-bounds} and~\eqref{eq:point-to-point-l-r} implies, uniformly over choices of $\ww,w_1^*$, 
\begin{align*}
\omega_{\tshape} (\Gamma_\corridor^\downarrow ; \ww, w_1^*)   \geq \frac{c}{\sqrt M} e^{ - M \tau_0 (\beta)} e^{ - \kappa_\beta |H- KN^{1/3}|^2/M}\,,
\end{align*}
as long as $H\leq M$. 

Putting the above bounds together, we see that if $H\ge KN^{1/3}$, $H\leq M$, and $r\leq K^{-1} H$, we have for some $c>0$  that 
\begin{align*}
\sum_{\gamma \in \Gamma_\corridor} q_{\tshape}(\gamma) &  \geq (K\sqrt r)^2 (KN^{1/3})^{\mathcal N} \Big(\frac{\varepsilon_{\beta} e^{ - K^2 \tau_0(\beta)- \tau_0 (\beta) N^{2/3}}}{N^{1/3}} \Big)^{\mathcal N-1} \Big( \frac{\varepsilon_{\beta} e^{ - K^2 \tau_0(\beta)- r \tau_0 (\beta) }}{\sqrt r}
\Big)^2 \\
& \quad \quad \times\Big(\frac{c e^{ - M \tau_0 (\beta)- \tau_0(\beta) |H- KN^{1/3}|^2/M}}{\sqrt M}   \Big)^2 \\ 
& \geq \frac{(\delta_\beta)^{\mathcal N} N^{1/3}}{M} e^{ - \tau_0(\beta)|L+2r|} e^{ - \tau_0(\beta)|H- KN^{1/3}|^2/M}\,,
\end{align*}
for some $\delta_\beta>0$ small (depending on $K$ which was taken to be large only in $\beta$). Using the fact that 
\begin{align*}
\frac{Z_{\pm,0,\tshape}}{Z_{+,0,\tshape}} = \langle \sigma_{\vw} \sigma_{\ve} \rangle_{\tshape}^* \leq \langle \sigma_{\vw} \sigma_{\ve}\rangle_{\mathbb Z^2}^* \leq \frac{C}{\sqrt{L+2r}} e^{ - \tau_{0}(\beta) |L+2r|}\,,
\end{align*}
(where the first equality is by~\eqref{eq:random-line-representation-partition-functions} and~\eqref{eq:random-line-representation-identity}, for the domain $\tshape$), and the following inequalities are by \eqref{eq:two-point-function-monotonicity} and
\eqref{eq:point-to-point-bounds} respectively),
we see that 
\begin{align*}
\sum_{\gamma \in \Gamma_\corridor} q_{\tshape}(\gamma)  \geq \frac{(\delta_\beta)^{\mathcal N} N^{1/3} \sqrt{L+2r}}{M} e^{- \tau_0(\beta) |H-KN^{1/3}|^2/M} \frac{Z_{\pm,0,\tshape}}{Z_{+,0,\tshape}}\,.
\end{align*}
Putting the above together with~\eqref{eq:lowerbnd123}, we see for some $C_{\beta}, C_{\beta,\lambda}>0$ that 
\begin{align}\label{eq:part-function-tilt-local}
\frac{Z_{\pm,\lambda,\tshape}}{Z_{+,\lambda,\tshape}} &\geq \frac{Z_{\pm,0,\tshape}}{Z_{+,0,\tshape}} \cdot \frac{N^{1/3}\sqrt L}{M} \exp \Big( - C_{\beta} \frac{L-2M}{N^{2/3}} - \tau_0(\beta)\frac{|H- KN^{1/3}|^2}{M} - C\lambda (HM + KLN^{1/3})\Big) \nonumber \\ 
& \geq \frac{Z_{\pm,0,\tshape}}{Z_{+,0,\tshape}} \cdot \frac{N^{1/3}\sqrt L}{M}\exp \Big( -C_{\beta,\lambda} \frac{L}{N^{2/3}} - \tau_0(\beta) \frac{H^2}{M} - C_{\beta,\lambda} \frac{HM}{N}\Big)
\end{align}
under the assumptions that $r\leq L$, $r\leq H$ and $H\geq KN^{1/3}$. Optimizing over choices of $M$ in the exponent, we make the choice of 
\begin{align*}
M = \sqrt{\frac{R^2 N^{2/3} N}{RN^{1/3}}}= \sqrt R N^{2/3}\,.
\end{align*}
Since $L \geq N^{2/3}$, the pre-factor in~\eqref{eq:part-function-tilt-local} is at least $R^{-1/2}$; using either the $H^2/M$ or $HM/N$ in the exponential, this pre-factor is absorbed by a sufficiently large  $C_{\beta,\lambda}$, giving 
\begin{align*}
\frac{Z_{\pm,\lambda,\tshape}}{Z_{+,\lambda,\tshape}} \geq \frac{Z_{\pm,0,\tshape}}{Z_{+,\lambda,\tshape}} \exp\Big( - C_{\beta,\lambda} \Big(\frac{L}{N^{2/3}} + R^{3/2}\Big)\Big)\,,
\end{align*}
 and concluding the proof of the proposition. 
\end{proof}

\section{Right tail for the one-point height distribution}\label{sec:upper-tail}
In this section, we show the $\exp(- \Theta(R^{3/2}))$ upper tail behavior for the normalized height of the interface, $N^{-1/3} \hgt_x^\pm$, above any point in the bulk of $\partial_\south \Lambda_N$.  Namely, the goal of this section is to prove the following two propositions, which combine to give Theorem~\ref{thm:one-point-tail-bounds}.  

\begin{proposition}\label{prop:right-tail-ub}
For every $\beta>\beta_c$ and $\lambda = \frac {c_\lambda}{N}$, there exists $C(\beta, c_\lambda)>0$ such that for every $x\in \partial_\south \Lambda_N$, and every $R=o(N^{2/9})$, we have 
\[
\mu_{\lambda,N}^\pm (\hgt_x^+ \geq RN^{1/3})\leq Ce^{-R^{3/2}/C}\,.
\]
\end{proposition}

\begin{proposition}\label{prop:right-tail-lb}
For every $\beta>\beta_c$ and $\lambda = \frac{c_\lambda}{N}$, there exists $C(\beta,c_\lambda)>0$ such that for every $R= o(N^{2/3})$ and every sequence $x=x_N$ at distance at least $\sqrt R N^{2/3}$ from $\{0,N\}$,  we have
\begin{align*}
\mu_{\lambda,N}^{\pm} \big(\hgt_x^- \geq R N^{1/3}\big) \geq  e^{ - CR^{3/2}}\,.
\end{align*}
\end{proposition}
In the upper bound, the constraint $R= o(N^{2/9})$ is such that the right-hand side is greater than $e^{ - cN^{1/3}}$, coming from certain enlargement couplings. In the lower bound, $R = o(N^{2/3})$ is so that the height is not macroscopic, and we do not feel the effect of the ceiling at height $N$.   

As indicated in Section \ref{sec:ideas-of-proofs}, the heart of this section is the proof of Proposition~\ref{prop:right-tail-ub}.  The proof strategy is sketched out in detail in Section~\ref{sec:right-tail-ideas-of-proof} and carried out in Section~\ref{sec:right-tail-ub}.
The proof of Proposition~\ref{prop:right-tail-lb} follows our general strategy for lower bounds and is presented in Section~\ref{sec:right-tail-lb}.

\subsection{Idea of proof of Proposition~\ref{prop:right-tail-ub}}\label{sec:right-tail-ideas-of-proof}
Recall the heuristic for the $e^{ - \Theta( R^{3/2})}$ upper bound from Section~\ref{sec:ideas-of-proofs}, where we considered the probability of $\hgt_x^+\geq RN^{1/3}$ given the \emph{stopping domain} $I_\star$---the innermost interval containing $x$ throughout which the interface stays above $RN^{1/3}/2$. However, there is no way of revealing the stopping domain without obtaining information about its interior, and a union bound over all possible stopping domains would induce a pre-factor of $N^{2/3}$.   We therefore use a delicate multi-scale analysis to localize the interface to a window of critical scaling---$\sqrt R N^{2/3}$ by $RN^{1/3}$, before using the enlargement machinery of Section~\ref{sec:enlargements} together with Gaussian upper tails on the interface fluctuations at this critical scale.

The localization step of this analysis goes as follows. Consider a mesh of columns separated by distances $N^{2/3}$ and consider a pair of consecutive columns on either side of $x$: for any fixed pair of columns, we consider the probability that the endpoints of the stopping domain are between the consecutive columns on either side.   
 By the a priori regularity estimates on the spikiness of $\mathscr I$ from Section~\ref{subsec:spikiness}, if an interval $I_\star$ is the stopping domain, then its approximation along the mesh of columns must be such that the entry and exit of the interface through those columns is below some $\mathcal H_I$ (chosen precisely for the regularity probability to sustain a union bound over the mesh of columns). In this manner, we bound the probability of any interval being the stopping domain, by the probability of some interval delimited by the $N^{2/3}$ column mesh being such that its entry and exit data are not too high, and such that the interfaces stays above $RN^{1/3}/2$ inside.   
 
The key step is then bounding the probability of a deterministic interval $I$ with entry and exit data below some $\mathcal H_I$, having its interface staying above height $RN^{1/3}/2$ throughout. The interface staying above $RN^{1/3}/2$ forces an atypically large $\Theta(RN^{1/3} |I|)$ number of sites, in the ``minus phase" of the Ising measure. Specifically, we find that the cost from the external field of placing so many sites in the ``minus phase" beats the tilt factor of Proposition~\ref{prop:local-tilt-enlargements}. This argument is complicated by the fact that the ``minus phase" is not stable under the presence of an external field; remedying this entails a local version of the comparability of $\mathcal C^-$ and $\Lambda^-$.
With the careful choice of the above parameters, the union bound over the deterministic mesh of columns only induces an $O(1)$ pre-factor, and we deduce that with probability $1-e^{- \Theta(R^{3/2})}$ the stopping domain $I_\star$ has width at most $C\sqrt R N^{1/3}$, allowing us to locally apply the Gaussian tail bound.

\subsection{Proof of Proposition~\ref{prop:right-tail-ub}}\label{sec:right-tail-ub}In this section, we prove the upper bound on the right-tail of $\hgt_x^+$.
The section is organized as follows: we begin by defining our deterministic mesh of columns, the stopping domain $I_\star$, and the rounding scheme from $I_\star$ to the fixed column mesh. Then, in Section~\ref{subsec:spikiness-right-tail}, we use Proposition~\ref{prop:spikiness-bound} to prove the regularity estimate we require on the local spikiness of $\mathscr I$, and describe our rounding scheme for $I_\star$.  In Section~\ref{subsubsec:localizing}, we bound the probability of the stopping domain having size larger than $C\sqrt R N^{1/3}$. Finally, in Section~\ref{subsubsec:localized-Gaussian-tail}, we show that if the stopping domain has width $O(\sqrt R N^{1/3})$, the right-tail probability is at most $e^{ - \Theta(R^{3/2})}$.

\medskip
Suppose without loss of generality that $x$ is in the left-half of $\partial_\south \Lambda_N$ and begin by labeling a mesh of points of separation $\pm N^{2/3}$ on either side of the point $x$, by considering the sequence of points 
\begin{align*}
x_i := x+iN^{2/3} \qquad \mbox{for}\qquad i = - ...,-1,0,1,...\,.
\end{align*}
Let $I_{ij}$ be the interval $[x_{-i}\vee 0,x_j \wedge N]$ for $i \leq \lceil x N^{-2/3}\rceil$ and $j\leq \lceil(N-x)N^{-2/3}\rceil$.  For any interval $I\subset \partial_\south \Lambda_N$, (not necessarily of the form $I_{ij}$) if $H=R N^{1/3}$, define the events
\begin{align}\label{eq:entry-exit-interval}
\mathcal E_{\leq H}^{I,-}:= \{\max \{\hgt_{\partial_\east \Lambda_{I}}^-, \hgt_{\partial_\west \Lambda_{ I}}^-\} \leq H\}\quad \mbox{and}\quad \mathcal E_{\leq H}^{I,+}:= \{\max \{\hgt_{\partial_\east \Lambda_{I}}^+, \hgt_{\partial_\west \Lambda_{ I}}^+\} \leq H\}\,,
\end{align}
where we recall that $\Lambda_I = I \times \llb 0,N\rrb$. 
The following event will be central to our argument:
\begin{align}\label{eq:upsilon-ij}
    \Upsilon_{i,j}  = \{\mathscr I :  \text{there exists an $I_\star$ such that $I_{i+1,j+1}\supset I_\star\supset I_{i,j}$ for which $\mathcal E_{\leq H}^{ I_\star,-}$ holds}\}.
\end{align}
Now let $L_0 = k_0 N^{2/3}$ where $k_0 := T_0 \sqrt R$ for $T_0$ sufficiently large, as specified later. We can bound 
\begin{align*}
\mu_{\lambda,N}^\pm (\hgt_x^+> 4RN^{1/3})\leq \mu_{\lambda,N}^\pm \Big(\hgt_x^+>4RN^{1/3} , \!\!\bigcup_{i,j:i+j\leq k_0} \Upsilon_{i,j}\Big)+\mu_{\lambda,N}^\pm \Big(\hgt_x^+>4RN^{1/3}, \!\!\bigcap_{i,j:i+j\leq k_0} \Upsilon_{i,j}^c\Big)\,.
\end{align*}
Since the interface is pinned at height zero on both endpoints of $\Lambda_N$, deterministically, $\Upsilon_{i,j}$ holds for some $i,j$. 
 Therefore  
\begin{align}\label{eq:one-point-upper-bound-split}
\mu_{\lambda,N}^\pm (\hgt_x^+> 4RN^{1/3}) & \leq   \mu_{\lambda,N}^\pm \Big(\bigcap_{i,j:i+j\leq k_0} \Upsilon_i^c, \bigcup_{i,j:i+j>k_0}  \Upsilon_{i,j}\Big) +  \mu_{\lambda,N}^\pm \Big(\hgt_x^+>4RN^{1/3} , \bigcup_{i,j:i+j\leq k_0} \Upsilon_{i,j}\Big)\,.
\end{align}

The core of the argument is dealing with the first term via a stopping domain argument.

\subsubsection{A priori regularity properties}\label{subsec:spikiness-right-tail}
We use the bounds of Section~\ref{subsec:spikiness} to establish a priori control on the regularity of the interface, and allow us to argue that 
 the entry/exit data along the mesh $\{x_i\}$ is a good proxy for the minimal/maximal behavior of the interface in corridors in $I_{i+1,j+1}\setminus I_{i,j}$. 

We begin by defining upward and downward spikiness events for an interval $I\subset \partial_\south \Lambda_N$, whereby the interface's maximal height in $\Lambda_I$ exceeds twice its entry and exit heights, or its minimal height in $\Lambda_I$ is half its entry and exit heights, respectively. Towards that, for every height $h$, define 
\begin{align}
\spiky_{I}^{\uparrow}(h) & := \{{\mathscr I}: \hgt_{x_\west}^- \leq h, \hgt_{x_\east}^-  \leq h\} \cap \{{\mathscr I}: \max_{x\in I} \hgt_{x}^+ \geq 2h\} \label{eq:up-spiky} \\ 
\spiky_I^{\downarrow}(h) &  := \{{\mathscr I}: \hgt_{x_\west}^+ \geq h, \hgt_{x_\east}^+ \geq h\} \cap \{{\mathscr I}: \min_{x\in I} \hgt_{x}^- \leq h/2\} \label{eq:down-spiky}
\end{align}

Note that there is a subtle difference between~\eqref{eq:up-spiky}--\eqref{eq:down-spiky} and the similar looking events in Proposition~\ref{prop:spikiness-bound}; e.g., in the upward spikiness event of the latter, we consider $\{\hgt_{x_\west}^+ \leq h, \hgt_{x_\east}^+  \leq h\}$ as opposed to $\{\hgt_{x_\west}^- \leq h, \hgt_{x_\east}^-  \leq h\}$ in~\eqref{eq:up-spiky}.  However it will be useful for us to work with the present definitions and it will be straightforward to  conclude from Proposition~\ref{prop:spikiness-bound} and bounds on sizes of overhangs of the interface $\sI,$ that spikiness satisfies Gaussian tail bounds. 

\begin{lemma}\label{lem:spikiness-bound}
Fix $\beta>\beta_c$ and $\lambda = \frac{c_\lambda}{N}$; there exists $C= C(\beta, c_\lambda)>0$ such that the following holds. For every $I: |I|\geq N^{2/3}$ and every $h \geq K\sqrt {|I|}$ for $K$ large (depending on $\beta,c_\lambda$), 
\begin{align*}
\mu_{\lambda,N}^{\pm} ( \spiky_I^\uparrow(h)) &  \leq C\exp(- h^2 /C|I|) + C e^{ - N^{1/3}/C}
\end{align*}
For every $I: N^{2/3}\leq |I|\leq AN^{2/3}$, for every $h\geq K\sqrt{|I|}$ for $K$ large (depending on $A, \beta, c_\lambda)$, ,
\begin{align*}
\mu_{\lambda,N}^{\pm} ( \spiky_I^\downarrow (h))  & \leq  C\exp(- h^2 /C|I|) + C e^{ - N^{1/3}/C}
\end{align*}
\end{lemma}
(We do not actually use the upwards spikiness estimate in this paper, but have included it for completeness.)
Before proving Lemma~\ref{lem:spikiness-bound}, let us discuss the manner in which we will apply it to the intervals $(I_{i,j})_{i,j}$. 
  For every $i$, define the height 
\begin{align}\label{eq:H-i}
\mathcal H_i :=  \max \{ i^{2/3} R^{1/3} N^{1/3}, 2 RN^{1/3}\}\,,
\end{align}
and for each $i,j$ pair, define the west and east segments of $I_{i+1,j+1} \setminus I_{i,j}$ as 
\begin{align*}
I^{\west}_{i,j} & = \llb x-(i+1)N^{2/3}, x-iN^{2/3}\rrb \times \llb 0,N\rrb\,, \quad \mbox{and} \quad  I^{\east}_{i,j} = \llb x+jN^{2/3}, x+ (j+1)N^{2/3}\rrb \times \llb 0,N\rrb\,.
\end{align*}
Using Lemma~\ref{lem:spikiness-bound}, we bound the probability that there are spiky oscillations in $\str_\west$ or $\str_\east$. Recall the definition of $\spiky_I^{\downarrow}(h)$ from~\eqref{eq:down-spiky} and for each $i,j$, let
\begin{align*}
\spiky_{i,j}^{\downarrow}& := \spiky_{I^\west_{i,j}}^\downarrow (\mathcal H_{i+j}) \cup \spiky_{I^\east_{i,j}}^{\downarrow} (\mathcal H_{i+j})\,.
\end{align*}
It follows immediately from Lemma~\ref{lem:spikiness-bound}, the fact that each of $I_{i,j}^\west$, $I_{i,j}^\east$ has width $N^{2/3}$, and the choice for $\mathcal H_{i+j}$,   
that the probability of spikiness in $I_{i+1,j+1}\setminus I_{i,j}$ is appropriately small. 

\begin{corollary}\label{cor:right-tail-spikiness}
Fix $\beta>\beta_c$ and let $\lambda = \frac{c_\lambda}{N}$. There exists $C(\beta,c_\lambda)>0$ such that for every $i,j,R$,
\begin{align}\label{eq:right-tail-spikiness}
\mu_{\lambda, N}^{\pm} (\spiky_{i,j}^{\downarrow}) \leq C\exp( - [(i+j)^{4/3} R^{2/3}\vee R^2]/C) + Ce^{ - N^{1/3}/C}
\end{align}
\end{corollary}
\begin{proof}[\textbf{\emph{Proof of Lemma~\ref{lem:spikiness-bound}}}]
We apply Proposition~\ref{prop:spikiness-bound} to obtain the lemma, using the exponential tails on overhangs of $\mathscr I$. However, the random-line representation, and therefore Proposition~\ref{prop:overhang-height-tail} are in the absence of an external field. From Theorem~\ref{thm:effect-of-global-tilt}, it follows that the exponential tail bound extends to overhangs of size at least $KN^{1/3}$ (for $K$ large in $\beta, \lambda$) in the presence of a field. 
\begin{claim}\label{clm:overhangs}
Fix $\beta >\beta_c$ and $\lambda = \frac{c_\lambda}{N}$. There exist $K_0(\beta,c_\lambda) , C(\beta,c_\lambda)>0$ such that for all $K\geq K_0$, 
\begin{align}\label{eq:overhang-tail-with-field}
\max_{x\in \partial_\south \Lambda_N} \mu_{\lambda,N}^\pm ( \hgt_x^+ - \hgt_x^-> K N^{1/3}) \leq C e^ { - KN^{1/3}/C} + C e^{ - N/C}\,.
\end{align} 
\end{claim}
Let $K_0$ be as above and let $\tilde h = h+K_0 N^{1/3}$. We split the event in question as follows: 
\begin{align*}
\spiky_I^\uparrow (h) &   \subset \big\{ \max_{i\in \{\east,\west\}} |\hgt_{x_i}^{+}  - \hgt_{x_i}^-|\geq K_0 N^{1/3} \big\}   \cup  \big\{\hgt_{x_\west}^+ \leq \tilde h, \hgt_{x_\east}^+ \leq \tilde h, \max_{x\in I}\hgt_x^+ \geq 3\tilde h /2\big\} 
\end{align*}
as long as $h > 3K_0 N^{1/3}$. 
By~\eqref{eq:overhang-tail-with-field}, the probability of the first term above is at most $C e^{ - K_0 N^{1/3}/C}$. 
Combining that bound with the bound of Proposition~\ref{prop:spikiness-bound} on the second term, we conclude the desired.  
In the other direction, we let $\tilde h = h- K_0 N^{1/3}$ and write, 
\begin{align*}
\spiky_I^\downarrow ( h ) \subset \big\{ \max_{i\in \{\east,\west\}} |\hgt_{x_i}^{+}  - \hgt_{x_i}^-|\geq K_0 N^{1/3}\big\} \cup \big\{\hgt_{x_\west}^+ \geq  \tilde h, \hgt_{x_\east}^+ \geq \tilde h, \min_{x\in I}\hgt_x^- \leq 3\tilde h/4\big\}  
\end{align*}
as long as $h> 3K_0 N^{1/3}$.  Eq.~\eqref{eq:overhang-tail-with-field} bounds the probability of the first term here. For the second term, if $h = KN^{1/3}$ for a large enough $K$, Proposition~\ref{prop:spikiness-bound} bounds the second term. 
\end{proof}

Having established Corollary~\ref{cor:right-tail-spikiness}, we next bound the terms on the right-hand side in \eqref{eq:one-point-upper-bound-split}.

\subsubsection{Localizing the interface to a critically sized window: first term in~\eqref{eq:one-point-upper-bound-split}}\label{subsubsec:localizing}
The proof proceeds by identifying a stopping domain.
On the intersection of $\bigcap_{i,j:i+j\leq k_0} \Upsilon_{i,j}^c$ and $\bigcup_{i,j:i+j >k_0} \Upsilon_{i,j}$, there {exists some smallest pair of indices} $i_\star,j_\star: i_\star + j_\star >k_0$ such that $\Upsilon_{i_\star,j_\star}$ holds and for every $i<i_\star$ and $j\leq j_\star$, or every $i\leq i_\star$ and $j<j_\star$, $\Upsilon_{i,j}$ did not hold. To this end, define 
\[ {\Upsilon}^{\sf{stop}}_{i,j} := \Upsilon_{i,j} \cap \bigcap_{k< i} \Upsilon_{k,j}^c\cap \bigcap_{k<j} \Upsilon_{i,k}^c\,.
\] 
(See Figure~\ref{fig:stopping-domain} for a depiction of this event.) By a union bound, we have
\begin{align}\label{eq:stopping-domain-decomposition}
\mu_{\lambda,N}^\pm \Big(\bigcap_{i,j:i+j\leq k_0} \Upsilon_{ij}^c, \bigcup_{i,j:i+j>k_0}  \Upsilon_{i,j}\Big)  \leq \sum_{i,j:i+j>k_0} \mu_{\lambda,N}^\pm \big( {\Upsilon}^{\sf stop}_{i,j}\big)\,.
\end{align}

\begin{figure}
\centering
\begin{tikzpicture}
\node at (0,0){
\includegraphics[width=.85\textwidth]{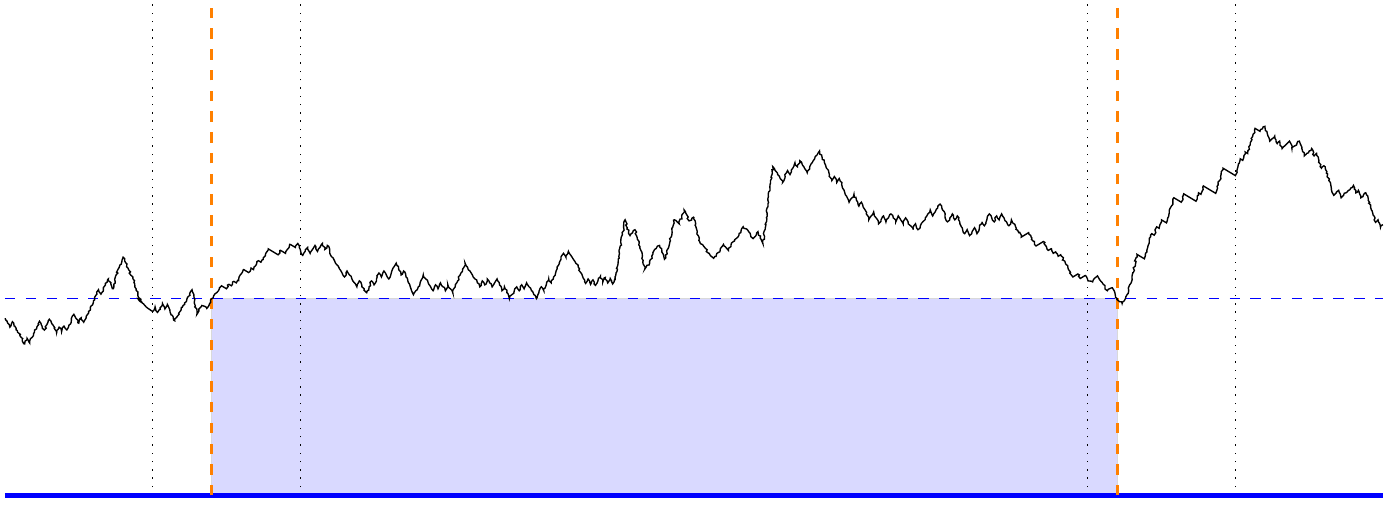}
};

\draw[<->, thick]  (-5.45,.75)-- (-4,.75) ;
\draw[<->, thick]  (5.45,.75)-- (4,.75) ;

\draw[<->, thick]  (6,-2.45)-- (6,-.5) ;

\node at (6.25, -1.5) {$H$};

\draw[decorate, decoration = {brace, amplitude = .1cm}, thick] (-3.95,-2.35)--(3.95,-2.35);
\draw[decorate, decoration = {brace, mirror, amplitude = .1cm}, thick] (-4.9,-2.6)--(4.25,-2.6);

\node at (0,-2.0) {$I_{ij}$};

\node at (-.3,-3) {$I_{\star}$};

\node at (-4.725, 1) {$N^{2/3}$};
\node at (4.725, 1) {$N^{2/3}$};


\end{tikzpicture}
\vspace{-.2cm}
\caption{The stopping domain $I_\star$ (delimited by the dashed orange lines) is the inner-most interval containing $x$ on which $\mathscr I$ stays above height $RN^{1/3}$ (shaded blue). This gets rounded to one of the deterministic approximations $I_{ij}, I_{i+1,j+1}, I_{i,j+1}$ or $I_{i+1,j}$ (intervals delimited by the dotted black lines).}\label{fig:stopping-domain}\vspace{-.2cm}
\end{figure}
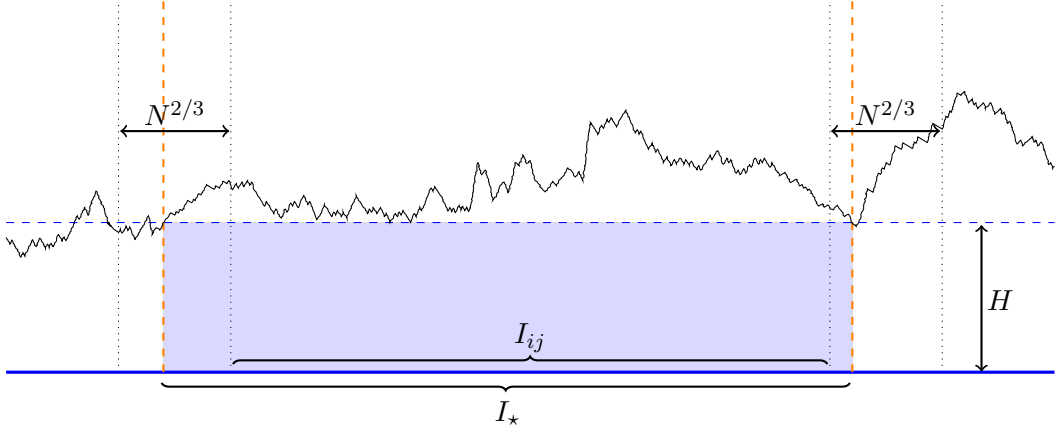

If we consider each of these summands individually, we then have the key lemma.
\begin{lemma}\label{lem:bound-on-long-high-interval}
Fix $\beta>\beta_c$ and $\lambda = \frac{c_\lambda}{N}$. There exists $C(\beta,c_\lambda)>0$ and $T_0 (\beta, c_\lambda)$ such that for every $x\in \partial_\south \Lambda_N$, every $R \leq N^{2/3}$ and every $i,j: i+j>k_0 = T_0 \sqrt R$,  
\begin{align*}
\mu_{\lambda,N}^\pm \big( \Upsilon^\s_{i,j}) \leq C\exp ( - (i+j) R/C) + Ce^{ - N^{1/3}/C}\,.
\end{align*}
\end{lemma}
\begin{proof}
We split the event $\Upsilon^\s_{i,j}$ into various sub-events which we bound separately.  While the interval $I_\star$ satisfying $\Upsilon_{i,j}$ is random, we  use~\eqref{eq:right-tail-spikiness} to control the entry and exit data for one of $I_{i,j}, I_{i+1,j}, I_{i,j+1}, I_{i+1,j+1}$. 
Begin by decomposing,   
\begin{align}\label{eq:spikiness-split}
\mu_{\lambda,N}^{\pm}( \Upsilon^{\s}_{i,j}) \leq \mu_{\lambda,N}^{\pm}(\spiky^{\downarrow}_{i,j}) + \mu_{\lambda_N}^{\pm}(\Upsilon_{i,j}^{\s} \cap (\spiky^{\downarrow}_{i,j})^c)\,.
\end{align} 
By~\eqref{eq:right-tail-spikiness}, for every $R$, 
\begin{align}\label{eq:right-tail-spikiness-2}
\mu_{\lambda,N}^{\pm}(\spiky^{\downarrow}_{i,j}) &  \leq Ce^{- ((i+j)^{4/3} R^{2/3} \vee R^2)/C} + Ce^{ - N^{1/3}/C} \nonumber \\
&  \leq C e^{- (i+j)R/C} + Ce^{ - N^{1/3}/C}\,.
\end{align}
Here the second inequality is because if $(i+j)\leq R$, the middle term is exponentially small in $R^2$,  and if $(i+j)\geq R$, we can split off $(i+j)^{1/3}\geq R^{1/3}$: in both of these cases the inequality holds. 

Turning to the next term in~\eqref{eq:spikiness-split}, recall the definitions of $\mathcal E_{\leq H}^{I,\pm}$ from~\eqref{eq:entry-exit-interval}; we have 
\begin{align}\label{eq:four-upsilons}
 \Upsilon^\s_{i,j} \cap (\spiky^{\downarrow}_{i,j})^c  & \implies \widetilde \Upsilon_{i+1,j+1} \cup \widetilde \Upsilon_{i+1,j} \cup \widetilde \Upsilon_{i,j+1} \cup \widetilde\Upsilon_{i,j}\,, \qquad \mbox{where} \\ 
\widetilde \Upsilon_{i,j} & = \mathcal E_{\leq \mathcal H_{i+j}}^{I_{i,j},+} \cap \Big\{\min_{x\in I_{i,j}}\hgt_x^- \geq RN^{1/3}\Big\} \nonumber\,,
\end{align}
where for the other index pairs, e.g., $\widetilde \Upsilon_{i+1,j+1}$, we change the interval in $\mathcal E^{I_{i,j},+}_{\le \mathcal H_{i+j}}$ e.g., to $\mathcal E^{I_{i+1,j+1},+}_{\le \mathcal H_{i+j}}$, but leave the other instances of $i,j$ (in $\mathcal H_{i+j}$ and in $\min_{x\in I_{i,j}}$) unchanged. The proof of Lemma~\ref{lem:bound-on-long-high-interval} concludes with the following Lemma~\ref{lem:bound-confining-large-area}, when it is combined with~\eqref{eq:spikiness-split}--\eqref{eq:right-tail-spikiness-2}. 
\end{proof}

\begin{lemma}\label{lem:bound-confining-large-area}
Fix $\beta>\beta_c$ and $\lambda = \frac{c_\lambda}{N}$. There exists $C(\beta,c_\lambda)>0$ and $T_0 (\beta, c_\lambda)$ such that for every $x\in \partial_\south \Lambda_N$, every $R \leq N^{2/3}$ and every $i,j: i+j>k_0 = T_0 \sqrt R$,
\begin{align*}
\mu_{\lambda,N}^{\pm} (\widetilde \Upsilon_{i,j}) \leq C \exp( - (i+j) R/C) + Ce^{-N^{2/3}/C}\,.
\end{align*}
The same bound also holds for $\widetilde \Upsilon_{i+1,j}, \widetilde \Upsilon_{i,j+1}, \widetilde\Upsilon_{i+1,j+1}$. 
\end{lemma}

\begin{proof}
We bound the probability of $\widetilde \Upsilon_{i,j}$, as the other proofs are similar. 
We leverage the fact that under $\widetilde \Upsilon_{i,j}$ the interface  confines an excess of minuses (costing in the external field) in a way that overwhelms the tilt on the surface tension bounded in Section~\ref{subsec:local-tilt}. 
Recall $H = RN^{1/3}$, and let 
$$\Lambda_{ij} = I_{ij} \times \llb 0,N\rrb\,, \qquad \mbox{and}\qquad \Lambda_{ij}^H := I_{ij} \times \llb 0,H\rrb\,.$$

We apply Proposition~\ref{prop:minus-enlargement} with the choices $r= N^{2/3}, h := \mathcal H_{i+j}$ and $\tshape = \tshape_{r,h}(\Lambda_{i,j})$. For an interface $\mathscr I$ in $\tshape$, let $\tshape^-$ be the set of sites below $\mathscr I$ and $\tshape^+$ is those above.  We then have  
\begin{align*}
\mu_{\lambda, N}^{\pm} (\widetilde \Upsilon_{i,j}) \leq \mu^\pm_{\lambda,N}(\Lambda^- \supset \Lambda_{ij}^H \mid \mathcal E_{\leq {h}}^{I_{ij},+}) \leq \mu_{\lambda,\tshape}^{(\pm,2 h)} (\tshape^- \supset \Lambda_{ij}^H) + Ce^{-N^{2/3}/C}\,,
\end{align*}
where by monotonicity we only increased the probability on the right-hand side by decreasing the boundary conditions to $(\pm,2h)$. 
(The application of Proposition~\ref{prop:minus-enlargement} is justified as follows: given $\mathcal E_{\leq h}^{I_{ij},+}$, the event $\Lambda^- \supset \Lambda_{ij}^H$ implies there is no vertical plus crossing of $I_{ij}\times \llb H,h\rrb$ (a decreasing event measurable w.r.t.\ $\Lambda_{ij}$), and in $\tshape$ with $(\pm,2h)$ boundary conditions, this in turn implies $\tshape^-\supset \Lambda_{ij}^H$.) The proof of Lemma~\ref{lem:bound-confining-large-area} then evidently concludes from the following lemma,  which we isolate as we will apply it again in bounding the multi-point height oscillations. 
\end{proof}

\begin{lemma}\label{lem:tshape-enclosing-large-area}
Fix $\beta>\beta_c$, and $\lambda = \frac{c_\lambda}{N}$. There exists $C(\beta,c_\lambda)>0$ and $T_0(\beta, c_\lambda)$ such that for every $x\in \partial_\south \Lambda_N$, every $R\leq N^{2/3}$, and every $i,j: i+j >T_0 \sqrt R$, if $h= \mathcal H_{i+j}$ is as in~\eqref{eq:H-i}, $H = RN^{1/3}$, $r= N^{2/3}$, and $\tshape^-$ is the set of sites below $\sI$ in $\tshape = \tshape_{r,h}(\Lambda_{i,j})$, then 
\begin{align*}
\mu_{\lambda, \tshape}^{(\pm,2h)} (\tshape^- \supset \Lambda_{ij}^H) \leq C \exp( - (i+j)R /C)\,.
\end{align*} 
\end{lemma}
\begin{proof}
We begin by localizing the interface in $\tshape$ to not be too high. Let $\e>0$ be sufficiently small in $\beta$ and $c_\lambda$ (to be chosen later) and define the event 
\begin{align*}
\Psi_\e: = \{\mathscr I: \max_{x\in \partial_\south \tshape} \hgt_x^+ \geq \e (i+j)\sqrt R N^{1/3}\}\,.
\end{align*}
(Notice that for $T_0$ sufficiently large depending on $\e$, the height $\e(i+j)\sqrt R N^{1/3}$ is at least $4h$.)
By monotonicity and Corollary~\ref{cor:max-height-fluctuation-floor}, there exists $C(\beta)>0$ such that for every $\e>0$, 
\begin{align*}
\mu_{\lambda,\tshape}^{(\pm,2h)} (\Psi_\e) \leq C \exp ( - \e^2 (i+j) R/C)
\end{align*}
Assume that 
\begin{align*}
\mu_{\lambda,\tshape}^{(\pm,2h)} (\tshape^- \supset\Lambda_{ij}^H) > C \exp ( - \e^2 (i+j) R /2C )\,,
\end{align*}
as otherwise we could conclude the proof for a different $C$. Then for every fixed $\e>0$, we have 
\begin{align}\label{eq:conditional-Psi-eps-probability}
\mu_{\lambda,\tshape}^{(\pm,2h)} ( \Psi_\e \mid \tshape^- \supset\Lambda_{ij}^H) \leq C \exp ( - \e^2 (i+j) R/2C)\,.
\end{align}

With that bound in hand, let us return to the main goal, and rewrite the probability 
\begin{align}\label{eq:probabilities-to-part-functions}
\mu_{\lambda,\tshape}^{(\pm,2h)} (\tshape^- \supset \Lambda_{ij}^H) = \mu_{0,\tshape}^{(\pm,2h)} (\tshape^- \supset \Lambda_{ij}^H) \frac{Z_{(\pm,2h),\lambda,\tshape}(\tshape^- \supset \Lambda_{ij}^H) Z_{(\pm,2h),0,\tshape}}{Z_{(\pm,2h),0,\tshape}(\tshape^-\supset\Lambda_{ij}^H) Z_{(\pm,2h),\lambda\tshape}}  \,.
\end{align}
Using Proposition~\ref{prop:local-tilt-enlargements}, the ratio of partition functions on the right-hand side is at most  
\begin{align}\label{eq:part-functions-applying-local-tilt}
\frac{Z_{{(\pm,2h)}, \lambda,\tshape}(\tshape^- \supset \Lambda_{ij}^H)}{Z_{{(\pm,2h)},0,\tshape}(\tshape^- \supset \Lambda_{ij}^H)}\frac{Z_{+,0,\tshape}}{Z_{+,\lambda,\tshape}} e^{C_{\beta,\lambda}(|I_{ij}|N^{-2/3} + ({h N^{-1/3}})^{3/2})}
\end{align}
From there, by writing 
\begin{align*}
\frac{Z_{+,0,\tshape}}{Z_{+,\lambda,\tshape}} & = \exp \Big( -  \int_0^\lambda d\lambda' \sum_{v\in \tshape} \langle \sigma_v \rangle_{+,\lambda',\tshape}\Big) \qquad \mbox{and} \\
\frac{Z_{{(\pm,2h)}, \lambda,\tshape}(\tshape^- \supset \Lambda_{ij}^H)}{Z_{{(\pm,2h)},0,\tshape}(\tshape^- \supset \Lambda_{ij}^H)} & = \exp \Big(\int_0^\lambda d\lambda'\sum_{v \in \tshape} \langle \sigma_v \mid \tshape^-\supset \Lambda_{ij}^H\rangle_{\pm,\lambda',\tshape}\Big)
\end{align*}
(see also~\cite[Section 2.3.3]{Velenik}) we have
\begin{align}\label{eq:part-function-ratio-magnetization-integral}
\frac{Z_{{(\pm,2h)}, \lambda,\tshape}(\tshape^- \supset \Lambda_{ij}^H)}{Z_{{(\pm,2h)},0,\tshape}(\tshape^- \supset \Lambda_{ij}^H)}\frac{Z_{+,0,\tshape}}{Z_{+,\lambda,\tshape}} = \exp\Big(\int_0^\lambda d\lambda' \sum_{v\in \tshape} \langle \sigma_v \mid \tshape^- \supset \Lambda_{ij}^H \rangle_{\pm,\lambda',\tshape} -\langle \sigma_v \rangle _{+,\lambda',\tshape}\Big)\,.
\end{align}
Our aim is to show there exists $\bar C(\beta,c_\lambda)>0$ such that the following bound holds, expressing the above cost as exponential in $\lambda |\Lambda_{ij}^H|$, as one might expect the external field to induce an area tilt: 
\begin{align}\label{eq:wts-cost-large-area-right-tail}
\exp\Big(\int_0^\lambda d\lambda' \sum_{v\in \tshape} \langle \sigma_v \mid \tshape^- \supset \Lambda_{ij}^H \rangle_{\pm,\lambda',\tshape} -\langle \sigma_v \rangle _{+,\lambda',\tshape}\Big)\leq \bar C\exp ( - c_\lambda (i+j) R/ \bar C)
\end{align}
Let us first conclude the proof from here. Plugging this bound back in to~\eqref{eq:probabilities-to-part-functions}--\eqref{eq:part-functions-applying-local-tilt} would yield the desired, since we can take $T_0$ sufficiently large (depending on $\beta, c_\lambda$) that for all $i,j: i+j>k_0$, 
\begin{align*}
c_\lambda (i+j)R /\bar C > C_{\beta,\lambda} (|I_{ij}| N^{-2/3} + (hN^{-1/3})^{3/2}) = C_{\beta,\lambda} ((i+j) + \max\{i \sqrt R, (2R)^{3/2}\})\,.
\end{align*}

We now turn to proving~\eqref{eq:wts-cost-large-area-right-tail}. 
By the FKG inequality,  there exists a monotone coupling $\nu_{\lambda'}$ of  $\tilde \sigma \sim \mu_{\lambda',\tshape}^{\pm} (\cdot \mid \tshape^- \supset \Lambda_{ij}^H )$ and $ \sigma \sim \mu_{\lambda',\tshape}^+$ such that $\nu_{\lambda'}$-a.s.\ $\tilde \sigma \leq \sigma$. We can then express  
\begin{align}\label{eq:magnetization-integral-coupling}
\exp\Big(\int_0^\lambda d\lambda' \sum_{v\in \tshape} \langle \sigma_v \mid \tshape^- \supset \Lambda_{ij}^H \rangle_{\pm,\lambda',\tshape} -\langle \sigma_v \rangle _{+,\lambda',\tshape}\Big) \leq \exp\Big(\int_0^\lambda d\lambda' \mathbb E_{\nu_{\lambda'}} \Big[ \sum_{v\in \tshape} \tilde \sigma_v - \sigma_v\Big]\Big)\,.
\end{align}

In particular, under this coupling, revealing any interface $\mathscr I$ in $\tilde \sigma$, the configuration $\tilde \sigma$ is all-plus on $\partial_\interior \tshape^+$ (all sites frozen by $\mathscr I$ and in $\tshape^+$), all-minus on $\partial_\interior \tshape^-$, and  distributed according to $\mu_{\lambda,\tshape^+}^+$ on $\tshape^+$ and $\mu_{\lambda,\tshape^-}^-$ on $\tshape^-$. By the monotonicity of the coupling $\sigma$ is also all-plus on $\partial_\interior \tshape^+$, the two configurations agree on all of $\tshape^+$, and the configuration $\sigma$ is pointwise larger than $\tilde \sigma$ on all of $\tshape^- \cup \partial_\interior \tshape^-$. As such, we bound the sum in the exponential in~\eqref{eq:magnetization-integral-coupling} as   
\begin{align*}
\mathbb E_{\nu_{\lambda'}}\Big[ \sum_{v\in \tshape} \tilde \sigma_v - \sigma_v\Big] & \leq \mathbb E_{\nu_{\lambda'}} \Big[\Big(\sum_{v\in \tshape} \tilde \sigma_v - \sigma_v\Big) \mathbf 1_{\{\Psi_\e^c\}} \Big] \\ 
&  \leq \mu_{\lambda',\tshape}^{(\pm,2h)}( \Psi_\e^c \mid \tshape^- \supset \Lambda_{ij}^H)  \max_{\mathscr I\in \Psi_\e^c: \tshape^-(\mathscr I) \supset \Lambda_{ij}^H} \sum_{v\in \Lambda_{ij}^H} \mathbb E_{\mu_{\lambda',\tshape^-}^-}[ \sigma_v]  - \mathbb E_{\lambda',\tshape^-}^{(-,\partial_\interior \tshape^-, +)}[\sigma_v]
\end{align*}
where the $(-,\partial_\interior \tshape^-,+)$ boundary conditions are those that are $+$ on $\partial \tshape$ and $-$ on $\partial_\interior \tshape^-$, and we used the monotonicity of the coupling to drop the sum on $\Psi_\e$, as well as to drop the sum over sites in $\tshape^- \setminus \Lambda_{ij}^H$. By~\eqref{eq:conditional-Psi-eps-probability} for all $R$ large enough (depending on $\beta,c_\lambda$ through $\e$),  this is at most 
\begin{align}\label{eq:diff-in-expectations-localizing}
{\frac{1}{2}}\max_{\mathscr I\in \Psi_\e^c: \tshape^-(\mathscr I) \supset \Lambda_{ij}^H} \sum_{v\in \Lambda_{ij}^H} \mathbb E_{\mu_{\lambda',\tshape^-}^-}[ \sigma_v]  - \mathbb E_{\lambda',\tshape^-}^{(-,\partial_\interior \tshape^-, +)}[\sigma_v]\,.
\end{align}
(using non-positivity of the summands). 
Now fix any $\mathscr I \in \Psi_\e^c \cap \{\tshape^- \supset \Lambda_{ij}^H\}$ and consider these two expectations individually. We will show that since $\mathscr I \in \Psi_{\e}^c$, the tilt from the external field is not too large, and the difference above is indeed comparable to the area of $\Lambda_{ij}^H$ as one would expect.

Start with the second expectation. By monotonicity, the induced measure of $\mu_{\lambda',\tshape^-}^{(-,\partial_\interior,\tshape^-,+)}$ on $\Lambda_{ij}^H$ stochastically dominates the measure $\mu_{\lambda',\Lambda_{ij}^H}^{(\pm,H)}$. Then, by monotonicity, we have 
\begin{align*}
\sum_{v\in \Lambda_{ij}^H} \mathbb E_{\lambda',\tshape^-}^{(-,\partial_
 \tshape^-,+)} [\sigma_v] \geq \sum_{v\in \Lambda_{ij}^H} \mathbb E_{\lambda',\Lambda_{ij}^H}^{(\pm,H)} [\sigma_v] \geq \sum_{v\in \Lambda_{ij}^H} \mathbb E_{\lambda',\Lambda_{ij}^H}^{(\pm,\lceil H/2\rceil)} [\sigma_v]
\end{align*} 
and by spin-flip and reflection symmetry and $\lambda' \geq 0$, this last sum is non-negative.  

\begin{figure}
\vspace{-.2cm}
\centering
\begin{tikzpicture}
\node at (0,0){
\includegraphics[width=.85\textwidth]{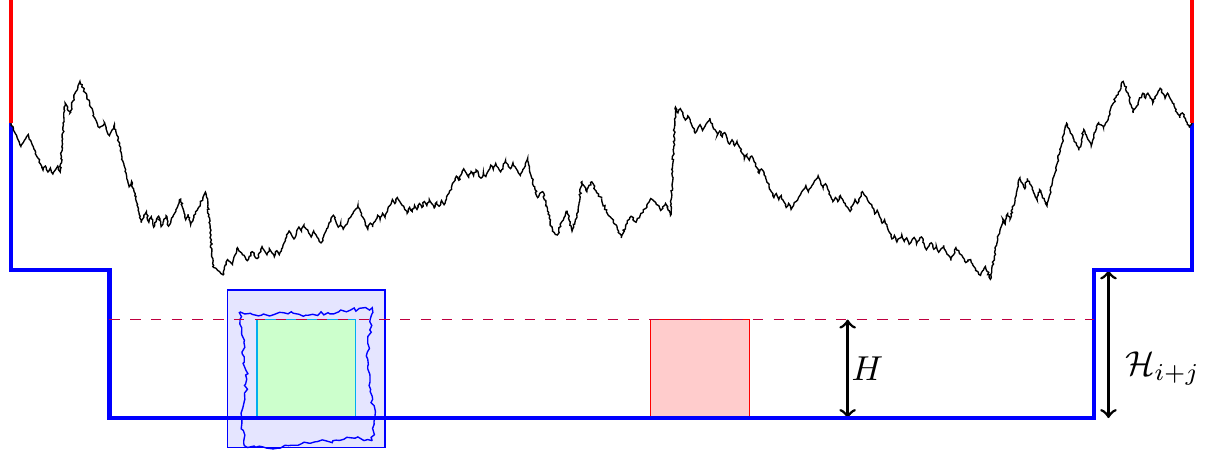}
};
\end{tikzpicture}
\vspace{-.2cm}
\caption{The event $\Upsilon_{ij}$, that $I_{ij}$ is the approximation to the stopping domain $I_\star$, requires $\mathscr I$ to confine $\Lambda_{ij}^H$ (below the dashed red). This should force an atypically negative magnetization in $\Lambda_{ij}^H$; to show this we use a tiling of $\Lambda_{ij}$, by disjoint tiles (green, red). We then couple the configuration on the tile to the infinite-volume minus measure using a minus circuit (blue) in the annulus $B_{\rho}^k \setminus \bmB_k$ (blue, shaded).}\label{fig:localizing-right-tail-tiling}\vspace{-.2cm}
\end{figure}

Turning to the expectation under the minus measure on $\tshape^-$, we coarse grain $\{v\in \Lambda_{ij}^H\}$ by tiles,  
$$\bmB_{k} : =  (x- iN^{2/3} ,0) + \llb k H, (k+1)H-1 \rrb \times \llb 0,H\rrb \qquad \mbox{for} \qquad k\leq  \frac{(i+j) N^{2/3}}{H}\,.$$
Observe that $\bmB_k$ cover $\Lambda_{ij}^H$ and each are fully contained in $\Lambda_{ij}^H \subset \tshape$. 
Consider the enlargement of $\bmB_k$ by a distance $\rho$, i.e., $B_\rho^k = \{v\in \mathbb Z^2:  d(v,\bmB_k)<\rho\}$. Let $\circuit^-(B_\rho^k\setminus\bmB_k)$ be the event that there is a circuit of minuses in the annulus $B_\rho^k \setminus \bmB_k$: see Figure~\ref{fig:localizing-right-tail-tiling}. Observe that the measure $\mu_{\lambda',B_\rho^k\cap \tshape^-}^-$ is stochastically below $\mu_{\lambda',\tshape^-}^-$ and therefore under the monotone coupling of these two, if the sample from the latter satisfies $\circuit^-(B_\rho^k\setminus \bmB_k)$ (possibly using the minus boundary sites of $\partial_\interior \tshape^-$), then so does the sample from the former. One can then expose the outer-most such minus circuit from $\partial B_\rho^k$ inwards, so that by the domain Markov property the samples from $\mu_{\lambda',B_\rho^k \cap \tshape^-}^-$ and $\mu_{\lambda',\tshape^-}^-$ would  agree on all of $\bmB_k$. Namely, for every $\delta>0$, if we upper bound
\begin{align*}
\mathbb E_{\mu_{\lambda',\tshape^-}^-} \Big[ \sum_{v\in \bmB_k} \sigma_v\Big] & \leq  -\delta |\bmB_k| + |\bmB_k| \mu_{\lambda',\tshape^-}^- \Big(\sum_{v\in \bmB_k} \sigma_v >-\delta |\bmB_k| \Big)  
\end{align*} 
the second term satisfies,  
\begin{align}\label{eq:coupling-on-tiles}
\mu_{\lambda',\tshape^-}^- \Big(\sum_{v\in \bmB_k} \sigma_x > - \delta |\bmB_k| \Big)  \leq   \mu_{\lambda',B_\rho^k \cap \tshape^-}^-\Big(\sum_{v\in \bmB_k} \sigma_x >-\delta |\bmB_k| \Big) + \mu_{\lambda',\tshape^-}^-\big(( \circuit^-(B_\rho^k))^c\big)\,.
\end{align}
We bound the two probabilities in~\eqref{eq:coupling-on-tiles}, and show that there exists a choice of  $\delta_0 >0$ such that the above is at most $\delta_0/2$.  For some $\eta$ to be chosen small depending on $\beta,c_\lambda$ later,  let
\begin{align*}
\rho: = \eta (i+j)^2 \sqrt R\,.
\end{align*}
(Note that $B_\rho^k$ does not necessarily fit in $\tshape^-$, but the minus circuit may include the minus sites of $\partial_\interior \tshape^-$.)  
Begin with the first term in~\eqref{eq:coupling-on-tiles}. By~\eqref{eq:radon-nikodym-bound}, and monotonicity, it is at most 
\begin{align*}
e^{ 2\lambda |B_\rho^k| } \mu_{0,\mathbb Z^2}^- \Big(\sum_{v\in \bmB_k} \sigma_x >- \delta |\bmB_k|\Big) \leq e^{ 4c_\lambda [ \eta^2 (i+j)^4 R + R^2 N^{2/3}]/N} \mu_{0,\mathbb Z^2}^- \Big(\sum_{v\in \bmB_k} \sigma_x >- \delta |\bmB_k|\Big)\,.
\end{align*}
By~\eqref{eq:surface-order-ld-component-of-infinity}, there exists a $\delta_0(\beta)>0$ together with a $C_0(\beta,\delta_0)>0$, such that  
\begin{align*}
\mu_{0,\mathbb Z^2}^- \Big( \sum_{v\in \bmB_k} \sigma_x > - \delta_0 |\bmB|_k\Big) \leq C_0 \exp ( - H /C_0)\,.
\end{align*}
Multiplying this by the tilt factor, we see that the first term of~\eqref{eq:coupling-on-tiles} is at most  
\begin{align*}
C_0 e^{ 4c_\lambda [ \eta^2 (i+j)^4 R + R^2 N^{2/3}]/N} e^{ - RN^{1/3}/C_0 } \leq C_0  \exp(- RN^{1/3} /2C_0 )\,.
\end{align*}
so long as $\eta$ is small depending on $c_\lambda$ and $C_0 (\beta,\delta_0)$ (using the bounds $i+j\leq N^{1/3}$ and $R\leq N^{2/3}$).  

On the other hand using $\mathscr I \in \Psi_\e$, we have $|\tshape^-|\leq 3\e (i+j)^2 \sqrt R N$, so that by~\eqref{eq:radon-nikodym-bound} together with monotonicity, for every $\delta$, the second term in~\eqref{eq:coupling-on-tiles} is at most   
\begin{align*}
e^{ 6 c_\lambda \e (i+j)^2 \sqrt R} \mu_{0,\mathbb Z^2}^- \big((\circuit^- (B_\rho^k))^c\big)\,.
\end{align*}
Using~\eqref{eq:circuit-probability-bound}, we find that this is at most 
\begin{align*}
C e^{ 6c_\lambda \e (i+j)^2 \sqrt R} e^{-\rho/C} \leq C e^{ - (i+j)^2 \sqrt R/2C}\,,
\end{align*}
as long as $\e$ is sufficiently small (depending on $\eta$ as well as on $\beta,c_\lambda$). Combining the above, we find that (for $N$ large) there exists $\delta_0 (\beta)>0$ such that as long as $\e$ was sufficiently small, uniformly over $\mathscr I\in \Psi_\e^c$, for $R$ sufficiently large, the right-hand side of~\eqref{eq:coupling-on-tiles} is at most $\delta_0/2$; in particular, we deduce that uniformly over $\lambda '\in [0,c_\lambda/N]$,
\begin{align*}
\mathbb E_{\mu_{\lambda',\tshape^-}^-} \Big[ \sum_{v\in \bmB_k} \sigma_v\Big] \leq - \frac{\delta_0}{2} |\bmB_k|\,.
\end{align*}
Summing over $k$, and using the non-negativity of the expectations under the $(-,\partial_\interior \tshape^-,+)$-measure, we find for some $\delta(\beta)>0$ and sufficiently small $\e$ (depending on $\beta,c_\lambda$), for all $\lambda' \in [0,c_\lambda/N]$,~\eqref{eq:diff-in-expectations-localizing} is at most $-\delta |\Lambda_{ij}^H|$, and therefore so is $\mathbb E_{\nu_{\lambda'}}[\sum_{v\in \tshape} \tilde \sigma_v - \sigma_v]$. Plugging this into~\eqref{eq:magnetization-integral-coupling} yields the desired~\eqref{eq:wts-cost-large-area-right-tail}, concluding the proof. 
\end{proof}

\subsubsection{Upper tail after localizing to critical scale: second term in~\eqref{eq:one-point-upper-bound-split}}\label{subsubsec:localized-Gaussian-tail}
It now suffices for us to take the $T_0$ sufficiently large from Lemma~\ref{lem:bound-confining-large-area}, and upper bound the probability of exceeding a height $4RN^{1/3}$ above $x$, together with $\bigcup_{i,j\leq k_0} \Upsilon_{i,j}$. We can bound 
\begin{align*}
\mu_{\lambda,N}^\pm \Big(\hgt_x^+>4RN^{1/3} & , \bigcup_{i,j:i+j\leq k_0} \Upsilon_{i,j}\Big)  \\
&  \leq \sum_{i,j:i+j\leq k_0} \Big( \mu_{\lambda,N}^{\pm} (\hgt_x^+>4RN^{1/3} , \mathcal E_{\leq h}^{I_{ij},+}) + \mu_{\lambda,N}^{\pm} (\spiky^{\downarrow}_{ij})\Big)\,.
\end{align*}
But in the regime where $i+j \leq k_0$, we have $h= \mathcal H_{i+j}= 2RN^{1/3}$ as long as $R$ is sufficiently large. Therefore, by the bound of~\eqref{eq:right-tail-spikiness}, the spikiness probability is at most $e^{ - R^2/C} + e^{ - N^{1/3}/C}$. For the other term in the summands, for every $i,j: i+j \leq k_0$, by Proposition~\ref{prop:minus-enlargement} with the choices $r = \frac{1}{2} |I_{ij}|$ and $h = \mathcal H_{i+j}$ and $\enlarge^\uparrow = \enlarge^\uparrow_{r,h}(\Lambda_{ij})$, this is at most 
\begin{align*}
\mu_{\lambda, \enlarge^\uparrow}^{(\pm, h)}   \big(\co_v^-( I_{ij} \times \llb h  & , 4 R N^{1/3}\rrb)\big) + CNe ^{ - r/C}  \leq \mu_{\lambda, \enlarge^\uparrow}^{(\pm, h)} \Big(\max_{x\in \partial_\south \enlarge^\uparrow} \hgt_x^+ \geq 4  R N^{1/3}\Big) + CNe ^{ - r/C}\,.
\end{align*}
By~\eqref{eq:max-height-fluctuation-with-field}, and $4RN^{1/3} - h = 2RN^{1/3}$,  we see that the first term on the right-hand side is at most 
\begin{align*}
Ce^{ -  R^2 / C (i+j)} + CNe^{ - N^{2/3}/C} \leq Ce^{- R^{3/2}/C T_0}+ Ce^{ - N^{2/3}/C}\,.
\end{align*} 
The second error term only affects the constant $C$ in the above maximum height bound as long as $R\leq N^{4/9}$.  Finally, there are at most $T_0^2 R$ many such summands, which gets absorbed into the constant for $R$ large, so that altogether, for every $T_0$ large enough,  and every $R\leq N^{2/9}$, 
\begin{align}\label{eq:after-localizing-bound}
\mu_{\lambda,N}^\pm \Big(\hgt_x^+>4RN^{1/3}  , \bigcup_{i,j:i+j\leq k_0} \Upsilon_{i,j}\Big) & \leq  C T_0^2 R( e^{ - R^{3/2}/CT_0} + e^{ - R^2/C} + e^{ - N^{1/3}/C}) \nonumber \\
& \leq C\exp(- R^{3/2}/C T_0)\,.
\end{align}

\subsubsection{Concluding the proof of Proposition~\ref{prop:right-tail-ub}}  Recall that $k_0 = T_0 \sqrt R$ for a sufficiently large constant $T_0$ (depending on $\beta,c_\lambda$) as dictated by Lemma~\ref{lem:bound-confining-large-area}. Combining~\eqref{eq:one-point-upper-bound-split} and~\eqref{eq:stopping-domain-decomposition}, and using Lemma~\ref{lem:bound-on-long-high-interval} and~\eqref{eq:after-localizing-bound} on the  terms therein, we see that for every $R\leq N^{2/9}$, 
\begin{align*}
\mu_{\lambda,N}^{\pm}(\hgt_x^+ >4RN^{1/3}) & \leq  Ce^{ - R^{3/2}/CT_0} + \sum_{i,j: i+j>k_0} \Big(Ce^{ - (i+j) R/C} + Ce^{ - N^{1/3}/C}\Big)  \leq C' e^{ - R^{3/2} /C'}\,,
\end{align*}
for some $C' >0$ depending on $\beta, c_\lambda$, concluding the proof. \qed

\subsection{Proof of Proposition~\ref{prop:right-tail-lb}}\label{sec:right-tail-lb}
Fix any $R =N^{2/3}/2$ and fix $x$ at distance at least $\sqrt R N^{2/3}$ from $\partial_{\east,\west} \Lambda_N$.  Let $I = \llb x-\sqrt RN^{2/3}, x+\sqrt R N^{2/3} \rrb$  be the interval in $\partial_\south \Lambda_N$ centered about $x$ of width $\sqrt R N^{2/3}$. Let $H = RN^{1/3}$ and let  $\Lambda_{I,3H} = I\times \llb 0,3 H\rrb$ for $3$. Notice that the measure $\mu_{\lambda,N}^\pm$ induces on $\Lambda_{I,3H}$ is stochastically below the measure $\mu_{\lambda, \Lambda_{I,3H}}^\pm$ (boundary conditions that are minus on $\partial_\south \Lambda_{I,3H}$ and plus on $\partial_{\east,\west,\north} \Lambda_{I,3H}$). By monotonicity and~\eqref{eq:radon-nikodym-bound},   
\begin{align}\label{eq:tail-lower-bound-after-tilt}
\mu_{\lambda,N}^\pm (\hgt_x^- > RN^{1/3}) \geq \mu_{\lambda,\Lambda_{I,3H}}^\pm (\hgt_x^->R N^{1/3}) \geq e^{-12  c_\lambda R^{3/2}} \mu_{0,\Lambda_{I,3H}}^\pm (\hgt_x^- >R N^{1/3}) \,.
\end{align}
To bound this latter probability, observe by monotonicity that we can send the bottom minus boundary of $\Lambda_{I,3H}$ to $-\infty$ to obtain $\Lambda_{I,3H}^\downarrow := I \times \llb -\infty, 3H\rrb$. 
We next apply the following Gaussian lower bound on $\hgt_x^-$: we expect this is in the literature, but include a proof for completeness.

\begin{lemma}\label{lem:Gaussian-lower-bound}
Fix $\beta>\beta_c$. There exist $\e_\beta, C(\beta), K_0 (\beta)>0$ such that for every $\ell$ and $K_0 \sqrt \ell \leq r\leq \ell$, 
\begin{align*}
\mu_{0,\Lambda_{\ell,\infty}^\downarrow}^{\pm} (\hgt_{(\ell/2,0)}^- > r) \geq \varepsilon_\beta \exp \big( -  r^2 / C \ell\big)\,, \qquad \mbox{where} \qquad \Lambda_{\ell,\infty}^{\downarrow} = \llb 0 ,\ell \rrb \times \llb -\infty, 3 r\rrb\,.
\end{align*} 
\end{lemma} 

\begin{proof}
Define a subset of interfaces $\Gamma$ which have the property that $\hgt_{(\ell/2,0)}^->r$ as follows: $\Gamma$ is the set of contours between $\vsw = (-\frac 12, -\frac 12)$ and $\vse= (\ell+\frac 12,-\frac 12)$ that are $\Lambda_{\ell,\infty}^{\downarrow*}$ admissible and intersect the column $\{\frac \ell 2 - \frac 12\}\times \mathbb R$ in a singleton in ${\sf W}:= \{(\frac \ell 2 -\frac 12, y): y\geq r\}$. Then we have 
\begin{align*}
\mu_{0,\Lambda_{\ell,\infty}}^\pm ( \hgt_{(\ell/2,0)}^- >r ) \geq \frac{1}{\langle \sigma_{\vsw} \sigma_{\vse} \rangle_{\Lambda_{\ell,\infty}}^*}\sum_{\substack{\gamma \in \Gamma}} q_{\Lambda_{\ell,\infty}}(\gamma)\,.
\end{align*}
By~\eqref{eq:q-monotonicity-domain}, along with~\eqref{eq:concatenation} (using that $|\gamma \cap (\{\frac \ell 2 - \frac 12\}\times \mathbb R)|=1$), this is at least
\begin{align*}
\frac{1}{\langle \sigma_{\vsw} \sigma_{\vse}\rangle_{\mathbb Z^2}^*} \sum_{w^*\in {\sf W}} \Big(\sum_{\substack{\gamma_\west \in \nbt: \vsw\to w^* \\ \gamma_\west \subset [0,\ell/2] \times [-\infty, 3r]}} q_{\mathbb Z^2} (\gamma_\west)\Big) \Big( \sum_{\substack{\gamma_\east \in \nbt: w^*\to \vse \\ \gamma_\east \subset [\ell/2,\ell] \times [-\infty, 3r]}} q_{\mathbb Z^2} (\gamma_\east)\Big)\,.
\end{align*}
Using Proposition~\ref{prop:no-backtracking-domain} and Corollary~\ref{cor:max-height-fluctuation-not-straight} as in~\eqref{eq:global-tilt-interior-weights-1}, as long as $r^2/\ell$ is sufficiently large (depending on $\beta$), there exists $\varepsilon_\beta >0$ such that for every half-integer $y\in [r,2r\wedge \ell]$
\begin{align*}
\sum_{\substack{\gamma_\west \in \nbt : \vsw \to (\frac \ell 2 - \frac 12,y) \\ \gamma_\west \subset [0,\ell/2] \times [-\infty,3r]}} q_{\mathbb Z^2} (\gamma_\west) \geq \varepsilon_\beta \langle \sigma_{\vsw} \sigma_{(\frac \ell 2 - \frac 12,y)} \rangle_{\mathbb Z^2}^*\,.
\end{align*}
The same bound of course also holds for $\gamma_\east: (\frac \ell 2 - \frac 12,y)\to\vse$. By~\eqref{eq:point-to-point-bounds} and~\eqref{eq:point-to-point-l-r}, since $r \leq \ell$, the probability we wish to bound is then at least 
\begin{align*}
\varepsilon_\beta^2 \sum_{\substack{w^* = (\frac \ell 2- \frac 12, y) \\ y\in \mathbb Z+\frac 12: r\leq y\leq 2 r\wedge \ell}} \frac{\langle\sigma_{\vsw} \sigma_{w^*} \rangle_{\mathbb Z^2}^* \langle \sigma_{w^*} \sigma_{\vse} \rangle_{\mathbb Z^2}^*}{\langle \sigma_{\vsw} \sigma_{\vse} \rangle_{\mathbb Z^2}^*} &   \geq \varepsilon_\beta^2 \sum_{y: r\leq y \leq 2r \wedge \ell}  \frac{K_1^2 }{K_2\sqrt \ell} e^{ - \kappa_\beta y^2 / \ell} \,,
\end{align*} 
yielding the desired bound. 
\end{proof}

We can apply the lemma to right hand side of~\eqref{eq:tail-lower-bound-after-tilt}, since $K_0 R^{1/4}N^{1/3} \leq RN^{1/3} \leq 2\sqrt R N^{2/3}$ for every $R\leq N^{2/3}$ sufficiently large. Therefore, by Lemma~\ref{lem:Gaussian-lower-bound}, there is a $C(\beta)>0$, 
\begin{align*}
\mu_{0,\Lambda_{I,3H}}^{\pm} (\hgt_x^- >RN^{1/3}) \geq \exp ( - CR^{3/2})\,,
\end{align*}
and
$\mu_{\lambda,N}^\pm (\hgt^-_x>RN^{1/3}) \geq  e^{ - 4c_\lambda R^{3/2}}  e^{ - CR^{3/2}}$
concluding the proof. \qed

\section{Tail estimates for multi-point height oscillations}\label{sec:multi-point-oscillations}
We now extend the one-point upper tail of the previous section to a large deviation tail estimate for attaining height $RN^{1/3}$ at many well-separated points, formalized as follows. 

\begin{theorem}\label{thm:multipoint-height-bounds}
Fix $\beta>\beta_c$, let $\lambda = \frac{c_\lambda}{N}$ for $c_\lambda>0$. There exists $R_0, T_0$ and $C>0$ such that for every $R\geq R_0$, the following holds. For every collection of $\mathcal N$ points $x_1,...,x_{\mathcal N}$ such that their distances are at least $T_0 \sqrt R N^{2/3}$, for every $\alpha_R$ s.t.\ $|\log \alpha_R| = o(\sqrt R)$,
\begin{align*}
\mu_{\lambda,N}^{\pm} \Big( \sum_{1 \le i\leq \mathcal N} \mathbf 1\{\hgt_{x_i}^+ > RN^{1/3}\} > \alpha_R  \mathcal N \Big) \leq  \exp ( - R^{3/2} \alpha_R \mathcal N /C ) + e^{ - N^{1/3}/C}\,.
\end{align*}
\end{theorem}
\medskip
\noindent 
\textbf{Idea of proof.} Before proceeding, we give a roadmap for the proof and describe some of the key ideas.  To avoid fractions in the notation, here and in the proof, we work with $4R$ instead of $R$; this proves the desired up to a corresponding change in the constant $C$.  Following the proof of the one-point estimate, we consider a potential collection of $M = \alpha_R \mathcal N$ points, say $(x_\ell)_{1\le \ell \leq M}$ on which the interface exceeds $4RN^{1/3}$. As before, each of these comes from an excursion on which the interface is consistently above height $RN^{1/3}$, which we call the stopping domain. We use these excursions to capture the dependencies between the events $\hgt_{x_{\ell}}^+ \geq 4RN^{1/3}$. 
More precisely, we use the stopping domains to group together the elements of $(x_\ell)_{\ell \leq M}$ according to whether they are part of the same excursion or different excursions. We then consider three events whose union contains the event $\min_{\ell\leq M} \hgt_{x_\ell}^{+}\geq 4RN^{1/3}$: namely, that a constant fraction of $(x_\ell)_{\ell\leq M}$ are such that, either,
\begin{enumerate}
\item Their stopping domains are \emph{spiky}, and the entry/exit heights of $\sI$ across a mesh of separation $N^{2/3}$ nearest to the boundary of the stopping domain are much larger than $O(RN^{1/3})$.  
\item Their stopping domains are not spiky, and they are of width $O(\sqrt R N^{2/3})$.
\item Their stopping domains are not spiky, and they are of width greater than $\sqrt R N^{2/3}$. 
\end{enumerate}
Roughly speaking, the first case consists of irregular interfaces, which are unlikely due to the locally Brownian oscillations at widths of order $N^{2/3}$: this entails a multi-point analogue of Corollary~\ref{cor:right-tail-spikiness}. The second case requires order $M$ essentially independent excursions from height $RN^{1/3}$ to $4RN^{1/3}$ in a width of $O(\sqrt RN^{2/3})$, this is unlikely due to the Gaussian upper tail behavior of the interface, even absent an external field and the monotone effect of the field on the interface. The third case entails an interface that remains above height $RN^{1/3}$ on a distance of order $M \sqrt RN^{2/3}$, which is exponentially unlikely in $MR^{3/2}$ due to the external field effect per a multi-point version of Lemma~\ref{lem:tshape-enclosing-large-area}. See Figure~\ref{fig:multi-height} for a rough depiction. We now proceed to making the above formal.

\begin{figure}
\centering
\begin{tikzpicture}
\node at (0,0){
\includegraphics[width=.95\textwidth]{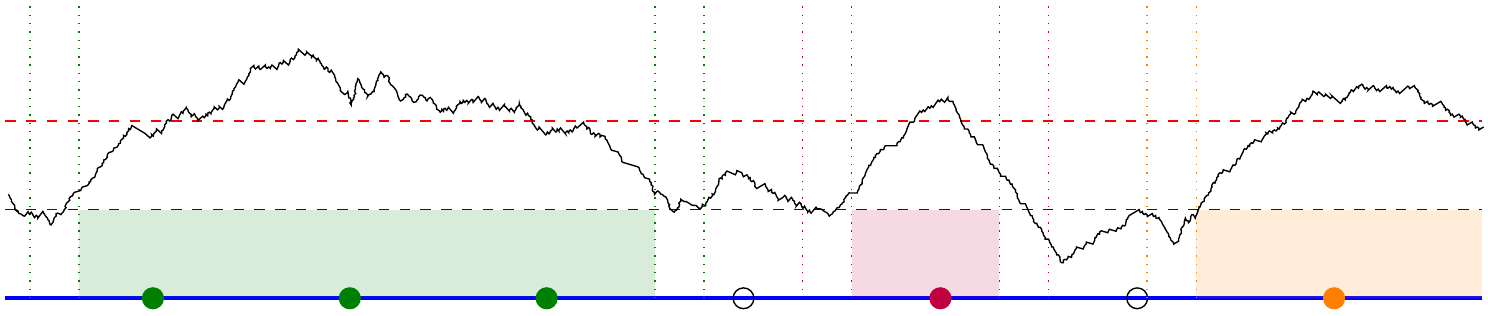}
};
\end{tikzpicture}
\caption{The depicted interface exceeds $4RN^{1/3}$ (dashed, red) above $(x_\ell)_{\ell \leq M}$ (colored nodes). The nodes are colored according to their partition element, induced by the stopping domain decomposition. For each partition element, its rounded stopping domain is delimited by the dotted columns of the matching color. The green element of the partition forces the interface to remain above height $RN^{1/3}$ (dashed, blue) for a width greater than $\sqrt RN^{2/3}$, and therefore encompass a large area (shaded, green). The purple and orange elements of the partition are singletons and require independent oscillations of $RN^{1/3}$ in a width less than $\sqrt RN^{2/3}$.}\label{fig:multi-height}
\end{figure}

\begin{proof}[\textbf{\emph{Proof of Theorem~\ref{thm:multipoint-height-bounds}}}]
Let $M = \alpha_R \mathcal N$ and union bound over the subsets of $M$ points amongst $x_1,...,x_{\mathcal N}$ 
at which the interface exceeds height $4 RN^{1/3}$. There are $\binom{\mathcal N}{M}$ many options of such subsets so we fix a subset $\chi$ of $(x_i)_{i\leq \mathcal N}$ of size $M$, and bound the probability of $\min_{x\in \chi} \hgt_x^+ > 4 RN^{1/3}$.  Abusing notation, let us label the points in $\chi$ as  $(x_\ell)_{\ell \leq M}$. We now define stopping domains about each $x_\ell$, throughout which the interface exceeds $H: =  RN^{1/3}$.   

For each $x_\ell$, we coarse-grain the interval $\llb 0,N\rrb$ into columns at separation $N^{2/3}$, to which we will round the stopping domain containing $x_\ell$, as in Section~\ref{subsec:spikiness-right-tail}. Namely, for each $\ell \leq M$, let 
\begin{align*}
I_{i,j}^\ell = [ x_\ell - iN^{2/3}\vee 0, x_\ell + jN^{2/3} \wedge N] \qquad \mbox{for $i\leq \lceil x_\ell N^{-2/3}\rceil$ and $j\leq \lceil (N-x_\ell)N^{-2/3}\rceil$}\,.
\end{align*}
Recall that for every interval $I \subset \partial_\south \Lambda_N$, we define the events
\begin{align*}
\mathcal E_{\leq H}^{I,-}  = \{\max\{\hgt_{\partial_\west I}^-, \hgt_{\partial_\east I}^-\} \leq H\}\qquad \mbox{and}\qquad \mathcal E_{\leq H}^{I,+} = \{\max\{\hgt_{\partial_\west I}^+ ,\hgt_{\partial_\east I}^+\} \leq H\}\,.
\end{align*} 
We now define the analogue of the event $\Upsilon_{i,j}$ from~\eqref{eq:upsilon-ij}: for each $\ell \leq M$, let
\begin{align*}
\Upsilon_{i,j}^{\ell} : = \{\sI : \mbox{ there exists an $I_\star^\ell$ such that $I_{i+1,j+1}^\ell\supset I_{\star}^\ell \supset I_{i,j}^\ell$ for which $\mathcal E_{\leq H}^{I_\star^\ell,-}$ holds}\}.
\end{align*} 
Since the interface is pinned at the boundary, it must be the case that for every $x_\ell,$ the event $\Upsilon_{i,j}^{\ell}$ holds for some $i,j\geq 0 $. For each $\ell$, we can thus define the inner-most such occurrence using
\begin{align*}
\Upsilon_{i,j}^{\ell,\sf{stop}}  : = \Upsilon_{i,j}^\ell \cap \bigcap_{0\leq k<i} (\Upsilon_{k,j}^{\ell})^c \cap \bigcap_{0\leq k<j} (\Upsilon_{i,k}^{\ell})^c
\end{align*} 
Evidently the union  over $i,j$ of $\Upsilon_{i,j}^\ell$ is a disjoint union of $\Upsilon_{i,j}^{\ell,\sf{stop}}$. We will control the correlations of heights at $(x_\ell)_{1\le \ell \le M}$ by partitioning them according to a stopping domain decomposition, which we define as follows: for every $x_\ell$, define its rounded stopping domain $I_{i,j}^{\ell,\sf{stop}}$ to be $I_{i,j}^{\ell}$ for the unique $(i,j)$ pair for which $\Upsilon_{i,j}^{\ell, \sf{stop}}$ holds: see Figure~\ref{fig:multi-height}. This partitions $(x_\ell)_{\ell\leq M}$  by letting  $x_{\ell}$ and $x_{\ell'}$ be in the same partition if their rounded stopping domains intersect. Notice that if $x_\ell,x_{\ell'}$ are in the same element of the partition, then $I_{i,j}^{\ell,\sf{stop}}$ contains $x_{\ell'}$ and vice versa. In particular, if the stopping domains---the intervals $I_\star^\ell$ and $I_\star^{\ell'}$ for which $\Upsilon_{i,j}^{\ell,\sf{stop}}$ and $\Upsilon_{i',j'}^{\ell',\sf{stop}}$ hold, respectively---intersect on more than one (boundary) point, then they are the same as subsets of $\partial_\south \Lambda$.

We wish to reduce this stopping domain decomposition into well-separated stopping domains to avoid boundary effects. Towards this, if we order the elements of the partition from left to right, we can define three subsets of the partition by assigning elements to a subset in an alternating manner, e.g., the first subset gets the first, fourth, etc... elements of the partition. One of these three sub-partitions necessarily contains at least $M /3$ of the points $(x_\ell)_\ell$; call this $\mathscr P$. Because the distance between $x_\ell$ and $x_{\ell'}$ exceeds $T_0 \sqrt R N^{2/3}$ for every $\ell\neq \ell'$,  the stopping domains of different elements of $\mathscr P$ are more than distance $T_0 \sqrt R N^{2/3}$ apart.  

For each element of $\mathscr P = (P_1 ,..., P_T)$, we can choose a representative, $x_{\ell_1} ,...,x_{\ell_T}$, say via some fixed enumeration over the vertices of $\partial_\south \Lambda_N$. We then associate to each element of $\mathscr P$, a rounded stopping domain via the representative $x_{\ell_s}$ and the index pair  $(i_s,j_s)_{s\leq t}$ for which $\Upsilon_{i_s,j_s}^{\ell_s,\sf{stop}}$ holds. 

For any subset of $\mathscr P$ of size $t \leq T$, and its representatives, which abusing notation we denote by $(x_{\ell_s})_{s\le t}$, and any choice of $(i_s,j_s)_{s\leq t}$, we define the following event that all of their stopping domains are spiky in the sense that the rounding approximation of the stopping interval has entry/exit data above $\mathcal H_{i_s + j_s}$, where we recall $\mathcal H_{i+j}= \max\{(i+j)^{2/3} R^{1/3} N^{1/3}, 2R N^{1/3}\}$ as in~\eqref{eq:H-i}. 
Recalling~\eqref{eq:up-spiky}--\eqref{eq:down-spiky}, we say the sequence of stopping domains is spiky if the following event holds:
\begin{align*}
\spiky_{(i_s,j_s)_{s\leq t}}^{\downarrow} = \bigcap_{s \leq t}  \spiky_{i_s,j_s}^{\downarrow} \quad \mbox{where}\quad \spiky_{i_s,j_s}^\downarrow = \spiky_{I_{i_s,j_s}^{s,\west}}^\downarrow (\mathcal H_{i_s + j_s}) \cup \spiky_{I_{i_s,j_s}^{s,\east}}^\downarrow (\mathcal H_{i_s + j_s}) 
\end{align*}
where $I_{i_s,j_s}^{s, \west}$ and $I_{i_s,j_s}^{s,\east}$ are the left and right strips constituting $I_{i_s+1,j_s+1}^{s} \setminus I_{i_s,j_s}^s$.
We can now split the event that $\mathscr I$ was such that the above procedure yielded the sub-partition $\mathscr P$ with representatives $(x_{\ell_i})_{i\leq T}$ and associated index pairs $(i_s,j_s)_{s\leq t}$, into the union of three events: 
\begin{enumerate}
    \item Event $E_{1}$: the sub-partition $\mathscr P$ has a subset of size $t$, say represented by $(x_{\ell_s})_{s\leq t}$, containing at least $M/6$ points amongst $(x_\ell)_{\ell\leq M}$ and such that $\spiky_{(i_s,j_s)_{s\leq t}}^\downarrow$ holds. 
    \item Event $E_2$: the sub-partition $\mathscr P$ contains at least $M/12$ singletons, say $(x_{\ell_1},...,x_{\ell_{M/12}})$, such that $(\spiky^\downarrow_{i_s,j_s})^c$ holds for all $s\leq M/12$ (purple and orange nodes in Figure~\ref{fig:multi-height}). 
    \item Event $E_3$: the sub-partition $\mathscr P$ has a subset of size $t$, say represented by $(x_{\ell_s})_{s\leq t}$ containing at least $M/12$ of $(x_\ell)_{\ell \leq M}$, none of whose elements are singletons, and such that $(\spiky^\downarrow_{i_s,j_s})^c$ holds for all $s\leq M/12$  (green nodes in Figure~\ref{fig:multi-height}). 
\end{enumerate}
Our aim is to show that
\begin{align}\label{eq:wts-cases}
    \mu_{\lambda,N}^{\pm} (E_i) \leq \exp ( - M R^{3/2}/C) + e^{ - N^{1/3}/C}\quad \mbox{for $i = 1,2,3$}\,.
\end{align}
Before proving~\eqref{eq:wts-cases}, let us complete the proof of the theorem assuming~\eqref{eq:wts-cases}. 
By union bounds over the choices of $M$ points $(x_\ell)_{\ell \leq M}$, we have, for $M = \alpha_R \mathcal N$, that
\begin{align*}
    \mu_{\lambda,N}^\pm \Big(\sum_{i\leq \mathcal N} \mathbf 1\{\hgt_x^+ >RN^{1/3}\} >\alpha_R \mathcal N\Big) &  \leq  \binom{\mathcal N}{M} \sum_{i=1,2,3} \mu_{\lambda,N}^{\pm} (E_i) \\
    & \leq   e^{ C  M |\log \alpha_R|} \big(e^{ - MR^{3/2}/C} + e^{ - N^{1/3}/C}\big)
\end{align*}
Since $M\leq (T_0 \sqrt R)^{-1} N^{1/3}$ and  $|\log \alpha_R|=o(\sqrt R)$, for $R$ sufficiently large (depending on $\beta, c_\lambda$), the exponentially decaying terms dominate $e^{ CM|\log \alpha_R|}$, and we deduce the desired bound, concluding the theorem for some other $C(\beta,c_\lambda)>0$. 
It remains to prove~\eqref{eq:wts-cases} for each of the events $E_1,E_2,E_3$. 

\medskip
\noindent \textbf{Proof of~\eqref{eq:wts-cases}, event $E_1$:} 
We union bound over the at most $2^{2M}$ choices of the sub-partition $\mathscr P$, and at most $2^M$ choices of the subset of $\mathscr P$, which abusing notation we call $(P_1,...,P_t)$, with representatives $(x_{\ell_s})_{s\leq t}$; it must be the case that $\sum_{i\leq t} |P_i| \geq M/6$. 

With the choice of $(P_1,...,P_t)$ fixed, we bound the probability of $\spiky_{(i_s,j_s)_{s\leq t}}^\downarrow$ via a union bound over the choices of $(i_s,j_s)_{s\leq t}$. Towards that, fix any choice of $(i_s,j_s)_{s\leq t}$ compatible with the fact that the stopping domains are separated by distances of at least $T_0 \sqrt R N^{1/3}$. We can further union bound over the $2^t$ choices of $\west,\east$ on which $\spiky_{i_s,j_s}^\downarrow$ holds for each $s\leq t$, and fix any such collection. For ease of notation, we take all $\west$; there is no differences in the argument for other choices.

We now use the multi-strip domain enlargement of Section~\ref{sec:multi-strip-domains}, to bound the probability of 
\begin{align*}
    \bigcap_{s\leq t} \spiky_{I_{i_s,j_s}^{s,\west}}^\downarrow (\mathcal H_{i_s+j_s})\,.
\end{align*}
For ease of notation, let $I_s^\west = I_{i_s,j_s}^{s,\west}$ and $I_s^\east = I_{i_s,j_s}^{s,\east}$, and let $\mathcal H_s = \mathcal H_{i_s+j_s}$. By Claim~\ref{clm:overhangs}, with probability $1-e^{ - N^{1/3}/C},$ we have $\max_x |\hgt_x^+ - \hgt_x^-|\leq K_0 N^{1/3}$ for a $K_0(\beta,c_\lambda)>0$ large, so it suffices to work on this event. On this event, as long as $R$ is sufficiently large relative to $K_0$, we have that $\mathcal E_{\geq \mathfrak h}$---defined as in~\eqref{eq:entry-exit-conditioning-event-multi-height} w.r.t.\ $\mathfrak R = (R_s)_{s\leq t} = (I_s^\west\times \llb 0,N\rrb)_{s\leq t}$ and $\mathfrak h = (h_s)_{s\leq t} =  (3\mathcal H_s/4)_{s\leq t}$---contains $\bigcap_{s\leq t}\spiky_{I_s^\west}^{\downarrow} (\mathcal H_s)$.  As such, we  would have 
\begin{align*}
\mu_{\lambda,N}^{\pm} \Big(\bigcap_{s\leq t}\spiky_{I_s^\west}^\downarrow(\mathcal H_s)\Big)  \leq \mu_{\lambda,N}^{\pm} \Big( \bigcap_{s\leq t}\spiky_{I_s^\west}^\downarrow(\mathcal H_s) \mid \mathcal E_{\geq \mathfrak h}\Big) = \mu_{\lambda,N}^{\pm} \Big( \bigcap_{s \leq t} \{\min_{x\in I_s^\west} \hgt_x^- \leq H \} \mid \mathcal E_{\geq \mathfrak h}\Big)\,.
\end{align*}

Notice that the event $\bigcap_{s\leq t} \{\min_{x\in I_s^\west} \hgt_x^- <H\}$ is an increasing event w.r.t.\ $\mathfrak R \cap \mathfrak E^\downarrow$, where we set $r= N^{2/3}$ and recall that $\mathfrak E^\downarrow = \bigcup\enlarge^\downarrow_{r,h_s}(R_s)$ as in Section~\ref{sec:multi-strip-domains}. Since $(I_s^\west)_{s\le t}$ are all a distance greater than $T_0 \sqrt R N^{1/3}$ apart, by Proposition~\ref{prop:multi-strip-plus-enlargement} with the above choice for $r$ and $\mathcal M = \e N^{2/3}$, for some sufficiently small $\varepsilon(\beta,c_\lambda)>0$, as long as $R$ is sufficiently large, since $t\le M$, the conditional probability above is at most 
\begin{align*}
\mu_{\lambda,\mathfrak E^\downarrow}^{\pm,\mathfrak h} \Big(\bigcap_{s\leq t} & \{\min_{x\in I_s^\west} \hgt_x^- < \mathcal H_s\} \Big) + CN t e^{ 8 c_\lambda (\varepsilon N^{2/3}+  t N^{1/3} + t N^{1/3})} e^{ - N^{2/3}/C}+ e^{ - \varepsilon^2 N^{1/3}/C} \\ 
&  \leq \mu_{\lambda,\mathfrak E^\downarrow}^{\pm,\mathfrak h} \Big(\bigcap_{s\leq t} \{\min_{x\in I_s^\west} \hgt_x^- < \mathcal H_s\} \Big) + e^{ - N^{1/3}/C}\,.
\end{align*}
where the constant $C$ changed between the two lines. 
The first term above is exactly 
\begin{align*}
\prod_{s \leq t} \mu_{\lambda,\enlarge^\downarrow_{r,h_s}(I_{s}^{\west})}^{\pm,h_s} \Big(\min_{x\in I_{s}^{\west}} \hgt_x^- <H\Big)\,.
\end{align*}
Recall that $(I_s^{\west})_{s\le t}$ are all of size $N^{2/3}$, and that $H\leq \mathcal H_s /2$. Therefore, using~\eqref{eq:radon-nikodym-bound} to remove the external field, together with Corollary~\ref{cor:max-height-fluctuation-floor} (spin flipped and vertically reflected), this is at most 
\begin{align*}
\prod_{s \leq t} \big[\exp ( c_\lambda \mathcal H_s / N^{1/3}) \exp ( - \mathcal H_s^2 / N^{2/3})\big]\leq \prod_{s\leq t} \exp( - [(i_s+j_s) R\vee R^2] /C) \,,
\end{align*}
for $R$ sufficiently large, where the inequality is, as before, by the choice of $\mathcal H_s$. 
Now observe that by construction, $i_s +j_s$ is at least $T_0 (|P_s|-1) \sqrt R$ where $|P_s|$ is the number of sites in the element of the partition represented by $x_{\ell_s}$. Altogether, we deduce that $\mu_{\lambda,N}^{\pm}(E_1)$ is at most 
\begin{align*}
    2^{3M} \max_{t\leq M}2^{t}   \sum_{\substack{(i_s,j_s)_{s\leq t}: \\ i_s+j_s \geq T_0 (|P_s|-1)\sqrt R}}  \Bigg[\prod_{s\leq t}\exp ( - [(i_s+j_s) R\vee R^2]/C) + e^{- N^{1/3}/C}\Bigg]\,.
\end{align*}
Using the fact that $M\leq (T_0 \sqrt R)^{-1} N^{1/3}$, and the assumption that $\sum_{s\leq t} |P_s| \geq M/6$, we obtain  
\begin{align}
    \mu_{\lambda,N}^{\pm}(E_1)\leq 2^{2M} ( R e^{ -  T_0 M R^{3/2}/C} + e^{ - N^{1/3}/C}) \leq \exp ( - M R^{3/2} /C)+ e^{ - N^{1/3}/C}\,.
\end{align}
for all $R\geq R_0$ where $R_0$ is a sufficiently large constant depending on $\beta, c_\lambda$. 

 \medskip
\noindent \textbf{Proof of~\eqref{eq:wts-cases}, event $E_2$:} 
We union bound over the at most $2^M$ choices of $M/12$ singletons, call them $(x_{\ell_s})_{s\leq M/12}$, amongst $(x_\ell)_{\ell \leq M}$  for which $E_2$ holds.  
Since their partition elements are singletons, it must be the case that $i_s,j_s \leq  T_0 \sqrt R$ for all $s\leq M/12$; thus for each $s \leq M/12$, the event $\Upsilon_{i_s,j_s}^{s,\sf{stop}}$ holds for some $i_s+j_s \leq k_0 := T_0 \sqrt R$. We further union bound over the at most $\binom{M/3}{M/12} k_0^{M/12}$ many choices for $(i_s, j_s)_{s\leq M/12}$. 

For every $(x_{\ell_s})_{s\leq M/12}$ and $(i_s,j_s)_{s\leq M/12}$, on $(\spiky_{i_s,j_s}^\downarrow)^c \cap \Upsilon_{i_s,j_s}^{s,\sf{stop}}$, we must have 
$$\{\hgt_{x_{\ell_s}}^+ \geq 4RN^{1/3}\} \cap \bigcup_{I_s\in \{I_{i_s,j_s}^{s}, I_{i_s+1,j_s}^{s}, I_{i_s+1,j_s + 1}^{s}, I_{i_s, j_s+1}^s\}} \mathcal E_{\leq \mathcal H_s}^{I_s,+}\,.$$
 There are $4^{M/12}$ such choices of $(I_s)_{s\leq M/12}$, and they are all handled in the same manner, so without loss of generality, let us bound the probability
 $$\mu_{\lambda,N}^{\pm} \bigg(\bigcap_{s\leq M/12} \hgt_{x_{\ell_s}}^+>4RN^{1/3}, \bigcap_{s\leq M/12} \mathcal E_{\leq \mathcal H_s}^{I_{i_s,j_s}^{s},+}\bigg)\,.
 $$
  For this, we can apply Proposition~\ref{prop:multi-strip-minus-enlargement} with the choices $\mathfrak R = (R_s)_{s\leq M/12} = (I_{i_s,j_s}^s)_{s \leq M/12}$, $\mathfrak h = (h_s)_{s\leq M/12} =  (\mathcal H_s)_{s\le M/12}$, and $r = N^{2/3}$. Since $\bigcap_{s\leq M/12} \hgt_{x_{\ell_s}}^+ >4RN^{1/3}$ is increasing and measurable with respect to $\mathfrak R \cap \mathfrak E^\uparrow$ (where we recall that $\mathfrak E^\uparrow = (\enlarge^\uparrow_{r,h_s}(R_s))_{s\le M/12}$), this is at most 
  \begin{align*}
\mu_{\lambda, \mathfrak E^\uparrow}^{\pm,\mathfrak h} \Big(\bigcap_{s \leq M/12} \hgt_{x_{\ell_s}}^+ > 4RN^{1/3} \Big)  + e^{- N^{2/3}/C} \leq \prod_{s \leq M/12} \mu_{\lambda, \enlarge^\uparrow_{r,\mathcal H_s} (I_{i_s,j_s}^{s})}^{\pm, \mathcal H_s} \Big(\hgt_{x_{\ell_s}}^+ >4RN^{1/3}\Big) + e^{ - N^{2/3}/C}
  \end{align*}
  Since $\mathcal H_\ell  = 2RN^{1/3}$ (using that $i+j \leq k_0$), by~\eqref{eq:max-height-fluctuation-with-field}, this is at most 
  \begin{align*}
  \prod_{s \leq M/12} \exp \big( - R^{2} /(C (i_s + j_s)) \big) + e^{ - N^{2/3}/C} \leq \exp ( - R^{3/2} M/C T_0) + e^{ - N^{2/3}/C}\,.
  \end{align*} 
  Multiplying by the factors from the union bound of $2^M \binom{M/3}{M/12} (4 T_0 \sqrt R)^{M/12}$,  
  as long as $R$ is sufficiently large (depending on $T_0$), we deduce the desired bound on $\mu_{\lambda,N}^{\pm}(E_2)$.

\medskip 
\noindent \textbf{Proof of~\eqref{eq:wts-cases}, event $E_3$:} 
We union bound over the at most $2^{2M}$ choices of partition $\mathscr P$, and at most $2^M$ choices of the subset of $\mathscr P$, which abusing notation 
we denote $(P_1,...,P_t)$ with representatives $(x_{\ell_s})_{s\leq t}$, such that $\sum_{s\leq t} |P_s|\geq M/12$, and such that none of $(P_s)_{s\leq t}$ are singletons. 

For such a collection, we union bound over the choices of $(i_s,j_s)_{s \le t}$; since $(P_s)_{s\le t}$ are not singletons, each must have $i_s + j_s \geq k_0$ . Towards that, fix any choice of $(i_s,j_s)_{s\leq t}$ compatible with $\mathscr P$, and consider the probability that  $\bigcap_{s\leq t}(\spiky_{i_s,j_s}^\downarrow)^c$ holds while $\Upsilon_{i_s,j_s}^{s,\sf{stop}}$ holds.  

As in~\eqref{eq:four-upsilons}, for each $s\leq t$, the intersection of $(\spiky_{i_s,j_s}^\downarrow)^c$ and $\Upsilon_{i_s,j_s}^{s,\sf{stop}}$ implies 
one of the events
$$\widetilde \Upsilon_{i_s,j_s}^s, \widetilde \Upsilon_{i_s+1,j_s}^s, \widetilde \Upsilon_{i_s,j_s+1}^s, \widetilde \Upsilon_{i_s+1,j_s + 1}^s\,,$$
where,
$$\widetilde \Upsilon_{i_s,j_s}^{s} := \mathcal E_{\leq \mathcal H_{i_s+j_s}}^{I_{i_s,j_s}^s,+} \cap \Big\{\min_{x\in I_{i_s,j_s}^s} \hgt_x^- \geq RN^{1/3}\Big\}\,.$$
As with~\eqref{eq:four-upsilons}, for the other index pairs, e.g., $\widetilde \Upsilon_{i_s+1,j_s+1}^s$, we change the interval $I_{i_s,j_s}^s$ to $I_{i_s+1,j_s+1}^s$ in $\mathcal E$ but leave the other instances of $i_s,j_s$ unchanged. There are $4^{t}$ choices over which of $\widetilde \Upsilon_{i_s,j_s}^s$, $\widetilde \Upsilon_{i_s+1,j_s}$, $\widetilde \Upsilon_{i_s,j_s+1}$, and $\widetilde \Upsilon_{i_s+1,j_s + 1}$ is used for each $s$, over which we union bound. Since they are all handled the same way, we take them all to be $\widetilde \Upsilon_s = \widetilde \Upsilon_{i_s,j_s}^s$ and let $I_s = I_{i_s,j_s}^s$ and $\mathcal H_s = \mathcal H_{i_s+j_s}$. 

It then remains to consider the probability of 
\begin{align*}
    \mu_{\lambda,N}^\pm \Big(\bigcap_{s\leq t} \widetilde \Upsilon_{s}\Big) \leq  \mu_{\lambda,N}^{\pm} \Big(\bigcap_{s\leq t} \{\Lambda^- \supset \Lambda_{I_s}^H\} \mid \bigcap_{s\leq t} \mathcal E_{\leq \mathcal H_s}^{I_s,+}\Big)\quad \mbox{where}\quad \Lambda_{I_s}^H = I_s \times \llb 0,RN^{1/3}\rrb\,. 
\end{align*}
Since the domains $(I_s)_{s\le t}$ are separated by $T_0 \sqrt R N^{2/3}$, and $(\{\Lambda^- \supset \Lambda_{I_s}^H\})_{s\le t}$ are decreasing events, we can apply Proposition~\ref{prop:multi-strip-minus-enlargement} with the choices $\mathfrak R =(R_s)_{s\leq t}= (I_s\times \llb 0,N\rrb)_{s\leq t}$, $\mathfrak h= (h_s)_{s\leq t} = (\mathcal H_s)_{s\leq t}$, and $r= N^{2/3}$ to bound 
\begin{align*}
    \mu_{\lambda,N}^{\pm}\Big(\bigcap_{s\leq t} \widetilde \Upsilon_{s}\Big)  \leq \mu_{\lambda,\mathfrak T}^{\pm,\mathfrak h} \Big(\bigcap_{s\leq t} \{\tshape^-_{s}\supset \Lambda_{s}^H\}\Big) + e^{ - N^{2/3}/C}\,.
\end{align*}
where $\mathfrak T = (\tshape_{r,h_s}(R_s))_{s\le t}$ and  $\tshape^-_s= \tshape^-_{r,h_s}(R_s)$ is the set of sites below the interface induced by $\mu_{\lambda,\mathfrak T}^{\pm,\mathfrak h}$ in the $s$'th element of $\mathfrak T$. Evidently, this decomposes into a product, so the first term is at most 
\begin{align*}
    \prod_{s\leq t} \mu_{\lambda,\tshape^-_s}^{(\pm,h_s)} \big(\tshape^-_s \supset \Lambda_s^H\big) \leq \prod_{s\leq t} \mu_{\lambda,\tshape^-_s}^{(\pm,2h_s)} \big(\tshape^{-}_s \supset \Lambda_s^H\big)\,.
\end{align*}
This latter quantity is exactly handled by Lemma~\ref{lem:tshape-enclosing-large-area}, since for every $s\leq t$, we have  $i_s + j_s \geq T_0 \sqrt R$ as the partition element cannot be a singleton. We deduce that for $T_0$ large enough, 
\begin{align*}
    \mu_{\lambda,N}^{\pm} \Big(\bigcap_{s\leq t} \widetilde \Upsilon_s\Big) &  \leq \prod_{s\leq t} \exp( - (i_s + j_s) R/C)+ e^{ - N^{2/3}/C}\,.
\end{align*}
Using the fact that $i_s +j_s \geq T_0 (|P_s|-1)\sqrt R \geq T_0 \sqrt R$, we altogether find  
\begin{align*}
    \mu_{\lambda,N}^{\pm}(E_3) \leq 2^{3M} \max_{t\leq M} 4^t  \sum_{\substack{(i_s,j_s)_{s\leq t}: \\ i_s+j_s \geq T_0 (|P_s|-1)\sqrt R}} \Big[ \prod_{s\leq t} \exp ( - (i_s + j_s) R/C) + e^{ - N^{2/3}/C} \Big]\,.   
\end{align*}
Using the fact that $M \leq (T_0 \sqrt R)^{-1} N^{1/3}$, and the assumption that $\sum_{s\leq t} |P_s|\geq M/12$, we deduce 
\begin{align*}
    \mu_{\lambda,N}^{\pm} ( E_3) \leq 2^{5M} \Big[\exp(  - T_0 M R^{3/2}/C) + e^{ - N^{2/3}/C}\Big]\,,
\end{align*}
for $R \geq R_0$ where $R_0$  is a sufficiently large constant depending on $\beta, c_\lambda>0$, as desired. 
\end{proof}

\section{Global behavior of the interface}\label{sec:averaged-statistics}
In this section, we prove Theorems \ref{thm:area-under-interface}--\ref{thm:max-tightness}, regarding the global behavior of $\mathscr I$. In Section~\ref{subsec:area-ub} we prove a sharp upper tail on $|\Lambda^-|$ (the area below the interface) beyond some $KN^{4/3}$. 
In Section~\ref{subsec:area-lb} we use the bound of Theorem \ref{thm:effect-of-global-tilt} to lower bound $|\Lambda^-|$, concluding the proof of Theorem~\ref{thm:area-under-interface}. Using a similar lower bound strategy, in Section~\ref{subsec:comparability}, we are also able to prove  a lower bound on the minus component of the bottom boundary, $|\mathcal C^-|$, yielding Theorem~\ref{thm:comparability-whp}. Finally, in Section~\ref{subsec:max-height-of-interface}, we prove Theorem~\ref{thm:max-tightness} showing that the maximum height of $\mathscr I$ is $\Theta(N^{1/3} (\log N)^{2/3})$. 

\subsection{Upper bound on the area contained in $\Lambda^-$}\label{subsec:area-ub}
Whereas the one-point upper tail bound of Section~\ref{sec:upper-tail} already gives the tightness for $N^{-4/3} |\Lambda^-|$, to obtain the sharp tail estimate on $|\Lambda^-|$ of Theorem~\ref{thm:area-under-interface}, we need to understand the correlations in the height oscillations. Towards this, we use the large deviations of the multi-point height oscillations, as proved in Theorem~\ref{thm:multipoint-height-bounds}. 

\begin{proof}[\textbf{\emph{Proof of Theorem~\ref{thm:area-under-interface}: upper bound}}]
First of all, we observe that for every $\mathscr I$, 
\begin{align}\label{eq:Lambda-area-height}
|\Lambda^-| \leq \sum_{x\in \partial_\south \Lambda_N} \hgt_x^+ \leq R_0 N^{4/3}+ \sum_{R\geq R_0} \sum_{x\in \partial_\south \Lambda_N}  N^{1/3} \mathbf 1\{\hgt_x^+ \geq R N^{1/3}\}\,,
\end{align}
Fix $T_0$ and  $R_0$ sufficiently large as per Theorem~\ref{thm:multipoint-height-bounds}; for each $R\geq R_0$, we bound the  number of sites typically attaining height $RN^{1/3}$. Towards this, fix any integer $R_0 \leq R \leq N^{2/9}$, let 
\begin{align*}
M_R = \alpha_R \mathcal N = \frac {\alpha_R}{T_0\sqrt R} N^{1/3}\,,
\end{align*}
for a sequence $\alpha_R$ to be chosen later, and consider the event
\begin{align*}
E_R = \bigcap_{j < T_0 \sqrt R N^{2/3}} E_R^{j} = \bigcap_{j< T_0 \sqrt R N^{2/3}} \bigg\{\sI: \sum_{k\leq \mathcal N} \mathbf 1\{\hgt_{j+k T_0 \sqrt R N^{2/3}}^+ >RN^{1/3} \}  \leq  M_R \bigg\}\,.
\end{align*}
Notice that under the event $\bigcap_{R\in \llb R_0 , N^{2/9}\rrb} E_R$, we have 
\begin{align*}
\sum_{R \in \llb R_0, N^{2/9}\rrb} \sum_{x\in \partial_\south \Lambda_N} N^{1/3} \mathbf 1\{\hgt_x^+ \geq RN^{1/3}\} \leq  \sum_{R\in \llb R_0, N^{2/9}\rrb}  T_0 \sqrt R N M_R \leq \sum_{R\in \llb R_0, N^{2/9}\rrb} \alpha_R N^{4/3}\,.
\end{align*}
From this we see that if $\sum_{R\geq R_0} \alpha_R <\infty$, there exists $C(\beta, c_\lambda)>0$ so that except with probability 
\begin{align}\label{eq:bad-probability-area-tail}
\mu_{\lambda,N}^{\pm} \Big( \bigcup_{R\in \llb R_0 ,N^{2/9}\rrb} E_R^c \Big)
\end{align}
we have $|\Lambda^-|\leq C N^{4/3}$ as desired (if $\sum_{R} \alpha_R <\infty$, then necessarily $M_{N^{2/9}} =o(1)$, so no site makes it beyond height $N^{2/9}$ under $E_{N^{2/9}}$). 
The quantity in~\eqref{eq:bad-probability-area-tail} can be bounded by a union bound as 
\begin{align*}
\sum_{R\in \llb R_0 ,N^{2/9}\rrb} \mu_{\lambda,N}^{\pm} ( E_R^c) \leq \sum_{R\in \llb R_0 ,N^{2/9}\rrb} \sum_{j< T_0 \sqrt R N^{2/3}} \mu_{\lambda,N}^{\pm} \bigg(\sum_{k\leq \mathcal N} \mathbf 1\{\hgt_{j+k T_0 \sqrt R N^{2/3}}^+ >RN^{1/3}\} >M_R\bigg)\,.
\end{align*} 
By  Theorem~\ref{thm:multipoint-height-bounds}, this latter probability is at most 
\begin{align*}
 \sum_{R \in \llb R_0, N^{2/9}\rrb} C T_0 \sqrt R N^{2/3}  \exp ( - R \alpha_R N^{1/3} /C) \,.
\end{align*}
For every $\eta>0$, choosing $\alpha_R = R^{-(1+\delta)}$ for $0<\delta< 9\eta/2$ arbitrarily small, we see that the above is at most $\exp ( - N^{\frac 13 - \eta} /C)$ for some other $C(\beta, c_\lambda)>0$, concluding the proof.  
\end{proof}

\subsection{Lower bounds on $|\Lambda^{-}|$}\label{subsec:area-lb}
We now prove a corresponding lower bound on the area below the interface,
denoted by $\Lambda^{-}$. With Theorem~\ref{thm:effect-of-global-tilt} in hand, the lower bound on $|\Lambda^-|$ comes from showing that the probability that the Ising interface with no external field confines an area of less than $\varepsilon N^{4/3}$ is exponentially decaying in $C_\varepsilon N^{1/3}$ for $C_\varepsilon \uparrow \infty$ as $\varepsilon\downarrow 0$. 

Following~\cite{Velenik}, the argument is to bound the probability of confining a small area by comparing the area in $\Lambda^-$ with the cumulative area under  $N^{1/3}$ many independent copies of ``critically sized" strips of width $N^{2/3}$ as explained in Section~\ref{sec:ideas-of-proofs}. The area confined on each strip was bounded in~\cite{Velenik} using the two-point estimates from the random line representation;  for conciseness, we choose to import this latter fact from the weak convergence of Ising interfaces in infinite strips to Brownian bridges~\cite{DH97,Hryniv98,GrIo05}. We note that if one did not want to rely on this convergence, with the refined Theorem~\ref{thm:effect-of-global-tilt} in hand, one could modify the proof of~\cite{Velenik}, replacing the bounds on maximal fluctuations of random lines with those of Proposition~\ref{prop:max-height-fluctuation-no-field} to obtain the desired.

\begin{proof}[\textbf{\emph{Proof of Theorem~\ref{thm:area-under-interface}: lower bound}}]
As we will apply Theorem~\ref{thm:effect-of-global-tilt}, we start by bounding the probability of $|\Lambda^-({\mathscr I})|<\eta N^{4/3}$ under $\mu_{0,N}^\pm$.
Observe that the event $|\Lambda^-({\mathscr I})|<\eta N^{4/3}$ is an increasing event in the spin configuration, so that one only increases its probability if, for some mesh $(x_i)_i \subset \partial_\south \Lambda$, we set all spins along $\{x_i\} \times \llb 0,\infty \rrb$ to be all-plus. One further increases the configuration by sending the minus boundary conditions along the floor down to $-\infty$, so that we are considering the Ising model on the  strips $S_{i}^{\downarrow} := \llb x_i,x_{i+1} \rrb \times \llb -\infty,N\rrb$ with $\pm$ boundary conditions. 
Let $\gamma_i$  be the interface of $S_i^\downarrow$ under the $\pm$ boundary conditions described above, and let $\Lambda^-(\gamma_i)$ be the set of sites under the interface $\gamma_i$ that are in the upper half plane. Then, there is a coupling such that $|\Lambda^-|$ is lower bounded by a sum of i.i.d.'s given by $\sum_i |\Lambda^-(\gamma_i)|$.

Let $(x_i)_i$ be a mesh with $x_0 = 0$ and $x_{i+1} - x_i = \varepsilon N^{2/3}$ for an $\varepsilon>0$ to be chosen later. We will lower bound the probability that $|\Lambda^-(\gamma_i)|$ is of size order $N$.  In particular, fix $i$ and let $$x_\west = \frac {x_{i}+x_{i+1}}2 - \varepsilon \delta N^{2/3}\,,\qquad \mbox{and}\qquad x_\east = \frac {x_{i}+x_{i+1}}2  +\varepsilon \delta N^{2/3}\,,$$  for some $\delta>0$ to be specified. Consider the probability that $\Lambda^-(\gamma_i)$ completely contains 
$$ \Delta := \llb x_\west, x_\east \rrb \times \llb 0, \sqrt \varepsilon \delta N^{1/3}\rrb\,.$$
We bound this probability by pulling it back from the convergence of the interface on the \emph{infinite} strip to a Brownian bridge. To do so, we first couple the measure induced on $\llb x_i, x_{i+1}\rrb \times \llb -\infty,N/2\rrb$ under $\mu_{0,S_i^\downarrow}^\pm$ to that under $\mu_{0,S_i}^\pm$, where $S_i = \llb x_i,x_{i+1}\rrb \times \llb - \infty,\infty\rrb$. 

\begin{claim}\label{clm:coupling-to-infinite-strip}
Let $\beta>\beta_c$. There exists $C(\beta)>0$ such that for every $i$, 
\begin{align*}
\|\mu_{0,S_i}^\pm (\sigma_{\llb x_i,x_{i+1}\rrb \times \llb -\infty,N/2\rrb}\in \cdot ) - \mu_{0,S_i^\downarrow}^{\pm}(\sigma_{\llb x_i,x_{i+1}\rrb \times \llb -\infty,N/2\rrb}\in \cdot )\|_{\textsc{tv}} \leq C e^{ - N/C}\,.
\end{align*}
\end{claim}
\begin{proof}
Observe that $\mu_{0,S_i}^{\pm}\preceq \mu_{0,S_i^\downarrow}^\pm$ and therefore, by monotonicity, there exists a coupling of the two measures such that if $\co_h^+ (\llb x_i,x_{i+1}\rrb \times \llb N/2,N\rrb)$ occurs under the former sample, then that plus crossing is also present under the latter. Revealing the upper-most such plus crossing, by the domain Markov property, the two samples will then agree below that plus crossing, and in particular on $\llb x_{i}, x_{i+1} \rrb \times \llb -\infty, N/2\rrb$. We therefore can bound the total-variation distance by 
\begin{align*}
\mu_{0,S_i}^{\pm}(\max_{x}\hgt_x^+\geq N/2)+ \mu_{0,S_i}^{\pm} (\co_h^+(\llb x_i,x_{i+1}\rrb \times \llb N/2,N\rrb), \max_{x}\hgt_x^+ <N/2)\,.
\end{align*}
By Proposition~\ref{prop:max-height-fluctuation-no-field}, the first probability is at most $Ce^{ - N/C}$ as desired. For the second term, conditionally on any interface $\gamma_i$ having $\max_{x}\hgt_x^+ <N/2$, the distribution induced on $\llb x_i,x_{i+1}\rrb \times \llb N/2,N\rrb$ is that induced by  $\mu_{0,\Lambda^+(\gamma_i)}^+$ which stochastically dominates the measure $\mu_{0,\mathbb Z^2}^+$ on $\llb 0,N \rrb\times \llb N/2,N\rrb$. That is to say, the second term above is at most
\begin{align*}
\sup_{\gamma_i: \max_x \hgt_x^+<N/2} \mu_{0,S_i}^\pm (\co_h^+(\llb x_i,x_{i+1}\rrb \times \llb N/2,N\rrb) \mid {\gamma_i}) \leq \mu_{0,\mathbb Z^2}^+ (\co_h^+(\llb x_i,x_{i+1}\rrb \times \llb N/2,N\rrb))\,,
\end{align*}
which by~\eqref{eq:crossing-probability-bound} is at most $Ce^{-N/C}$ for some $C(\beta)>0$. 
\end{proof}

We now pull back properties of the Brownian excursion to the pre-limiting Ising interface. 
\begin{claim}\label{clm:pullback-BM-convergence}
Fix $\beta>\beta_c$. There exist $\delta, \bar p>0$, such that for every $i$ and any $\varepsilon>0$
\begin{align*}
\mu_{0,S_i}^{\pm}\big(\Delta \subset \Lambda^-(\gamma_i) \subset (\llb x_i,x_{i+1}\rrb \times \llb -\infty,N/2\rrb) \big) > \bar p\,.
\end{align*}
\end{claim}
\begin{proof}
In~\cite[Theorem 1.2]{GrIo05}, it was shown that if we take a certain linear interpolation $\mathcal L_N[\gamma_i]$ of the interface $\gamma_i$ (informally, linearly interpolating between the \emph{break points} of $\gamma_i$) and rescale it appropriately ($\varepsilon^{-1} N^{-2/3}$ in the $x$-direction and $\varepsilon^{-1/2}N^{-1/3}$ in the vertical direction) to obtain $\phi_N:[0,1]\to \mathbb R$, there exists $\mathscr T_\beta>0$ such that $\frac{1}{\mathscr T_{\beta}}\phi_N$ converges weakly to the standard Brownian bridge in $C_0[0,1]$ (continuous functions on $[0,1]$ having $f(0)= f(1)=0$) with the uniform topology. 

Now, consider the subset of paths in $C_0[0,1]$ confined to the open subset of $[0,1]\times \mathbb R$,
\begin{align*}
\Gamma_\sqcap := ([0,1]\times (-\infty, \mathscr T_\beta^{-1}))\setminus ([-\delta,\delta]\times [-\infty,2 \mathscr T_\beta^{-1} \delta])\,;
\end{align*}
i.e., the paths that confine the strip $[-\delta,\delta]\times [-\infty,2\mathscr T_\beta^{-1}\delta]$ below them. 
Evidently $\{f\in C_0[0,1]:\{(x,f(x))\}\subset \Gamma_\sqcap\}$ is open in $C_0 [0,1]$.  As such, there exists $p(\delta,\beta)>0$ such that 
\begin{align*}
\liminf_{N\to\infty} \mu_{0,S_i}^{\pm}( {\mathscr T_{\beta}^{-1}} \phi_N \in \Gamma_\sqcap) \geq p\,,
\end{align*}
using the fact that $\Gamma_{\sqcap}$ is a positive probability event for the standard Brownian bridge. 

Next, using the fact~\cite[Eqs.~(1.10),(1.15)]{GrIo05} that the linear interpolation is at most distance $(\log N)^4$ (microscopic in the rescaling above) from the actual interface, $\gamma_i$, except with probability $o(1)$, we have that for large enough $N$, if $\bar p =p/2$, 
\[\mu_{0,S_i}^{\pm} \big( \Delta \subset \Lambda^-(\gamma_i) \subset (\llb x_i,x_{i+1}\rrb \times \llb -\infty, \varepsilon^{1/2} N^{1/3}\rrb) \big) >\bar p\,. \qedhere
\]\end{proof}

Combining Claims~\ref{clm:coupling-to-infinite-strip}--\ref{clm:pullback-BM-convergence}, we see that for all small enough $\delta>0$, there exists $\bar p>0$ such that 
\begin{align*}
\mu_{0,S_i^\downarrow}^{\pm}(|\Lambda^-(\gamma_i)|\geq \varepsilon^{3/2}\delta^2 N) \geq  \mu_{0,S_i}^{\pm}\big(\Delta \subset \Lambda^-(\gamma_i) \subset (\llb x_i,x_{i+1}\rrb \times \llb -\infty,N/2\rrb) \big) - Ce^{ - N/C} >  \bar p\,.
\end{align*}
Now, let $Z_i$ be the indicators that under $\mu_{0,S_i^\downarrow}^\pm$ the event $|\Lambda^-(\gamma_i)|\geq \varepsilon^{3/2}\delta^2 N$ holds. Recall that the law of $|\Lambda^-(\gamma)|$ under $\mu_{0,N}^{\pm}$ dominates $\sum_i |\Lambda^-(\gamma_i)|$ under $\otimes_i \mu_{0,S_i^\downarrow}^{\pm}$. Therefore, for every $r$,
\begin{align*}
\mu_{0,N}^{\pm}(|\Lambda^-(\gamma)|\leq r)\leq \otimes_i \mu_{0,S_i^\downarrow}^{\pm} \big( \sum_i \varepsilon^{3/2} \delta^2 N  Z_i \leq r\big)\,.
\end{align*}
By standard Binomial concentration, since $\mu_{0,S_i}^{\pm}(Z_i =1) \geq \bar p$, for $\eta< \eta_0= \bar p \varepsilon^{1/2}\delta^2/2$, we have that 
\begin{align*}
\otimes_i \mu_{0,S_i}^{\pm} (\sum_i \varepsilon^{3/2}\delta^2 N Z_i <  \eta N^{4/3}) & \leq \otimes_i \mu_{0,S_i}^{\pm}\Big(\sum_i^{\varepsilon^{-1}N^{1/3}} Z_i < \eta \varepsilon^{-3/2} \delta^{-2} N^{4/3}\Big)  \leq C\exp ( - \varepsilon^{-1} N^{1/3}/C)\,,
\end{align*}
for some $C(\bar p, \delta)$, i.e., $C(\beta,c_\lambda)>0$ (independent of $\varepsilon$). 
Taking $\varepsilon$ sufficiently small, $\varepsilon^{-1}/C$ above will be larger than $C_1$ from Theorem~\ref{thm:effect-of-global-tilt}, and we would find that for such $\varepsilon>0$, for every $\eta<\eta_0$, $$\mu_{0,N}^{\pm}(|\Lambda^-|<\eta N^{4/3})\leq Ce^{ - ((C\varepsilon)^{-1} - C_1)N^{1/3}} + Ce^{- N/C}$$ yielding the desired for some other $C(\beta)>0$. 
\end{proof}

\subsection{Comparability of $|\mathcal C^-|$ and $|\Lambda^-|$}\label{subsec:comparability} Following the proof strategy of the lower bound on $|\Lambda^-|$, we prove the lower bound of~\eqref{componentlowerbound}. By combining~\eqref{componentlowerbound} with the upper tail of $|\Lambda^-|$, Theorem~\ref{thm:comparability-whp} straightforwardly follows, proving Conjecture~1 of~\cite{Velenik}.

\begin{proof}[\textbf{\emph{Proof of \eqref{componentlowerbound} in Theorem~\ref{thm:comparability-whp}}}]
Fix mesh points $(x_i)_{i\leq N^{1/3}} \subset\partial_\south \Lambda$, given by $x_0 = 0$ and $x_{i+1} - x_i = N^{2/3}$ for each $i$. Since the event $|\mathcal C^-|<K^{-1} N^{4/3}$ is a decreasing event, we only increase its probability, by setting all spins along $\{x_i\}\times \llb 0,N\rrb$, and all sites above height $ T N^{1/3}$ (for a $T$ to be taken large later) to be plus. If we denote by $\Lambda_{i,h} = \llb x_i, x_{i+1} \rrb \times \llb 0,h\rrb$, we find that the sum of $|\mathcal C^-_i|$ (the minus clusters of $\partial_\south \Lambda_{i,h}$) under $\mu_{\lambda,\Lambda_{i,TN^{1/3}}}^{\pm}$ lower bounds $|\mathcal C^-|$ under $\mu_{\lambda,N}^{\pm}$.   

Notice that with the plus spins that were added, the measures induced on $\Lambda_{i,TN^{1/3}}$ are independent. Thus, by translation invariance, and standard concentration for binomials, we obtain the desired if there exist choices of $T(\beta,c_\lambda)$, $c(\beta, c_\lambda)>0$ and $K(\beta,c_\lambda)>0$ such that for any $i$,
\begin{align*}
    \mu_{\lambda,\Lambda_{i,TN^{1/3}}}^{\pm} (|\mathcal C^-_i| > K^{-1} N)  >c\,.
\end{align*}
In order to prove this, we first remove the external field using~\eqref{eq:radon-nikodym-bound}.  
Then, by Corollary~\ref{cor:max-height-fluctuation-floor} and a coupling argument as in Claim~\ref{clm:coupling-to-infinite-strip}, for every $T>0$, we can couple $\mu_{0,\Lambda_{i,T N^{1/3}}}^{\pm}$ to $\mu_{0,\Lambda_{i,\infty}}^{\pm}$ except with probability $1-e^{ - T^2 /C} - Ce^{ - N^{1/3}/C}$. Together, we deduce
\begin{align*}
    \mu_{\lambda,\Lambda_{i,TN^{1/3}}}^{\pm} (|\mathcal C^-_i| > K^{-1} N) \geq e^{ - 2 c_\lambda T} \mu_{0,\Lambda_{i,\infty}}^{\pm} (|\mathcal C^-_i| > K^{-1} N) - e^{ - T^2/C}\,,
\end{align*}
for all $T$ sufficiently large in $\beta,c_\lambda$. In fact, taking $T$ sufficiently large, it will suffice to prove for some $c(\beta),K(\beta)>0$ that $\mu_{0,\Lambda_{i,\infty}}^{\pm} (|\mathcal C^-_i|>K^{-1}N)>c$. Towards this, let 
\begin{align*}
    \Delta_{\delta} = \llb \tfrac{1}{2}(x_i + x_{i+1})  - \delta N^{2/3}, \tfrac{1}{2}(x_i + x_{i+1})  + \delta N^{2/3} \rrb \times \llb 0 ,  \delta N^{1/3}\rrb \,.
\end{align*}
By monotonicity w.r.t.\ the (doubly) infinite strip and Claim~\ref{clm:pullback-BM-convergence}, there exist $\bar p (\beta),\delta(\beta)>0$ so that  
\begin{align*}
    \mu_{0,\Lambda_{i,\infty}}^{\pm} (\Lambda^- (\gamma_i) \supset \Delta_{2\delta}) >\bar p\,.
\end{align*}
Under this event, the measure induced on $\Delta_{2\delta}$ is stochastically dominated by that induced by $\mu_{0,\mathbb Z^2}^-$, so that it suffices for us to now consider the intersection of the minus component of the x-axis with the box $\Delta_{2\delta}$ under $\mu_{0,\mathbb Z^2}^-$.  Towards this, define the crossing event $\circuit_{\sqcap}^- (\Delta_{2\delta} \setminus\Delta_{\delta})$ 
that there exists a minus half-circuit in the half-annulus $\Delta_{2\delta} \setminus \Delta_\delta$. Moreover, tile $\Delta_{\delta}$ by $2N^{1/3}$ many disjoint $\delta N^{1/3}\times \delta N^{1/3}$ squares $(\bmB_{\ell})_{\ell \leq  2N^{1/3}}$. Let $\Gamma_{i,\eta}$ be the event that in the $\ell$'th tile, 
\begin{align*}
    \sum_{v\in \bmB_{\ell}} \mathbf 1\{v\in \mathcal C_{\partial \Delta_{2\delta}}^-\} > \eta N^{2/3}\,.
\end{align*}
Note that these are all decreasing events measurable w.r.t.\ the configuration on $\Delta_{2\delta}$. Moreover, on the event $\circuit_{\sqcap}^- (\Delta_{2\delta} \setminus\Delta_{\delta})$, a connection from some $v\in \Delta_\delta$ to $\partial \Delta_{2\delta}$ implies $v\in \mathcal C^-_i$. Thus, we have 
\begin{align*}
    \mu_{0,\Lambda_{i,\infty}}^{\pm} (|\mathcal C^-_i| > \eta N) \geq \bar p \cdot \Big(1- \mu_{0,\mathbb Z^2}^- \Big(\big(\circuit_{\sqcap}^-(\Delta_{2\delta} \setminus\Delta_\delta)\big)^c \cup \bigcup_{i \leq 2N^{1/3}} \Gamma_{i,\eta}^c\Big)\Big)\,.
\end{align*}
We bound the probability on the right-hand side by a union bound. By~\eqref{eq:circuit-probability-bound}, the probability of the half-circuit not existing is at most $N \exp( - \delta N^{1/3}/C)$. By the surface order large deviations,~\eqref{eq:surface-order-ld-component-of-infinity}, if $\eta>0$ is sufficiently small (depending on $\delta$), the probability of $\Gamma_{i,\eta}^c$, is at most $\exp ( -  N^{1/3}/C)$. Taking $\delta$ sufficiently small in $\beta,c_\lambda$ first, then $\eta$ sufficiently small subsequently, the probability above is at least some $c(\beta,c_\lambda)>0$, from which the desired follows by taking $T$ sufficiently large.  
\end{proof}

\begin{proof}[\textbf{\emph{Proof of \eqref{comparability1} in Theorem~\ref{thm:comparability-whp}}}]
By definition, $\mathcal C^- \subset \Lambda^-$, so that deterministically, $|\mathcal C^-|\leq |\Lambda^-|$. On the other hand, for every $K>0$,  
\begin{align*}
    \mu_{\lambda,N}^{\pm} (|\mathcal C^-|<  K^{-2} |\Lambda^-|) \leq \mu_{\lambda,N}^{\pm} (|\mathcal C^-|<K^{-1} N^{4/3}) + \mu_{\lambda,N}^{\pm} (|\Lambda^-| >K N^{4/3})\,.
\end{align*}
There exists $K(\beta, c_\lambda)>0$ such that the first term on the right-hand side is at most $\exp ( - N^{1/3}/C)$ by~\eqref{comparability1}, and the second term is at most $\exp( - N^{1/3 - o(1)})$ by~\eqref{eq:area-tail-bounds} of Theorem~\ref{thm:area-under-interface}. 
\end{proof}

\subsection{Maximum height of the interface}\label{subsec:max-height-of-interface}

We next use the tail estimates of Section~\ref{sec:upper-tail} to prove Theorem~\ref{thm:max-tightness}, and show that
the maximum height of the interface is of size $\Theta(N^{1/3}(\log N)^{2/3})$. 

\medskip
\noindent \textbf{Proof of lower bound.} We begin with the lower bound, which, similar to the proof of the lower bound in Theorem~\ref{thm:area-under-interface}, relies on monotonicity to compare with a no-field model at the correct scale of interface fluctuations, and then apply the Gaussian lower bound of Lemma~\ref{lem:Gaussian-lower-bound}. It suffices for us to show that there exists $\eta_0>0$ such that for every $0<\eta<\eta_0$ and every $c(\beta,c_\lambda)\in (0,\frac 13)$, 
\begin{align}\label{eq:wts-lower-bound-max}
\mu_{\lambda,N}^{\pm}\big(\max_x \hgt_x^-< N^{1/3}(\eta \log N)^{2/3}\big)\leq & \exp\big(-cN^{\frac 13 - c}\big)\,.
\end{align}
To bound the above, fix $R$ to be chosen later (polylogarithmic in $N$) and consider the mesh,
\begin{align*}
x_i = i RN^{2/3} \qquad \mbox{for}\qquad i= 1,...,N^{1/3}/R\,.
\end{align*}
along with the rectangles given by
$$\Lambda_i^R =  \llb x_{i}-\sqrt R N^{2/3},x_{i}+\sqrt R N^{2/3}\rrb\times \llb0,3 RN^{1/3}\rrb\,,$$
with boundary conditions that are $+$ on $\partial_{\east, \north, \west} \Lambda_i^R$ and $-$ on $\partial_\south \Lambda_i^R$.  By monotonicity, the measure $\mu_{\lambda,\Lambda_i^R}^{\pm}$ stochastically dominates the measure induced by $\mu_{\lambda, N}^{\pm}$ on $\Lambda_i^R$. Thus $\max \hgt_{x}^-$ under $\mu_{\lambda, N}^{\pm}$, which is larger than $\max_{x\in \bigcup \partial_\south \Lambda_i^R} \hgt_x^-$, stochastically dominates the maximum over $i$ of  $N^{1/3}/R$ i.i.d. random variables $M_i$ distributed as  $\max_{x\in \partial_\south \Lambda_i^R}\hgt_x^-$ under $\mu_{\lambda,\Lambda_i^R}^{\pm}$. As such,  for any choice of $R$, 
\begin{align*}
\mu_{\lambda,N}^{\pm}\big(\max_{x\in \partial_\south \Lambda_N} \hgt_x^- < R N^{1/3} \big) \leq \mathbb P \Big( \mbox{Binom} \big(N^{1/3}/{R}, \mu_{\lambda,\Lambda_i^R}^{\pm}(M_i >  R N^{1/3} \big) = 0\Big)\,.
\end{align*}
Let us now bound the success probability in the binomial above.  First of all, by~\eqref{eq:radon-nikodym-bound}, 
\begin{align*}
\mu_{\lambda, \Lambda_i^R}^{\pm}(M_i > RN^{1/3}) \geq e^{-12 c_\lambda R^{3/2}} \mu_{0,\Lambda_i^R}^{\pm}(M_i > RN^{1/3})\,.
\end{align*}
If we denote by $ \Lambda_{i,R}^\downarrow$, the box $ \llb x_i  - \sqrt R N^{2/3}, x_i + \sqrt R N^{2/3} \rrb \times \llb -\infty ,3RN^{1/3}\rrb$, we see that  
\begin{align*}
\mu_{0,\Lambda_i^R}^{\pm}\big(M_i > RN^{1/3}\big) \geq \mu_{0, \Lambda_{i,R}^\downarrow}^{\pm}\big(M_i > RN^{1/3}\big) \geq \mu_{0, \Lambda_{i,R}^\downarrow}^{\pm} (\hgt_{x_i}^{-} >RN^{1/3})\,.
\end{align*}
By Lemma~\ref{lem:Gaussian-lower-bound}, as long as $R$ is sufficiently large and $R \leq N^{2/9}$ we see that for some $C(\beta)>0$,  
\begin{align*}
\mu_{0, \Lambda_{i,R}^\downarrow}^{\pm} \big(\hgt_{x_i}^{-} > RN^{1/3} \big) \geq e^{- C R^{3/2}} \qquad \mbox{so that} \qquad \mu_{\lambda, \Lambda_i^R}^{\pm} \big(M_i >RN^{1/3}\big) \geq  e^{-(12 c_\lambda + C)R^{3/2}}\,,
\end{align*}   
for all $R = o(N^{2/9})$. 
If we then take 
\begin{align*}
R = (\eta \log N)^{2/3}\,, \qquad \mbox{for}\qquad \eta <\eta_0 : = \frac{1}{6}(12c_\lambda + C)^{-1}\,,
\end{align*}
then for every $i$, the probability that $M_i >RN^{1/3}$ will be at least $\Omega(N^{-\frac 13 + c})$ for some $c(\beta,c_\lambda)>0$, and we recall that there are $N^{1/3}/R= \Omega(N^{1/3}/(\log N)^{2/3})$ many such attempts. As such, by standard concentration for binomial random variables, we will obtain 
\begin{align*}
\mu_{\lambda,N}^{\pm} \Big(\max_{x\in \partial_\south \Lambda_N} \hgt_x^- < N^{1/3} (\eta \log N)^{2/3}\Big) \leq C\exp (- N^{\frac 13}/(C(\log N)^{2/3}))\,,
\end{align*}
for some $C(\beta, c_\lambda)>0$. This yields the desired sufficiently large $N$. 

\medskip\smallskip
\noindent \textbf{Proof of upper bound.}
To upper bound $\max_x \hgt^+_x$, we use the upper tail of Proposition~\ref{prop:right-tail-ub} above any $x\in\partial_\south \Lambda_N$ and then union bound over the $N$ such choices.  Namely, for every $K>0$, 
\begin{align*}
\mu_{\lambda,N}^\pm (\max_x \hgt^+_x>KN^{1/3}(\log N)^{2/3}) & \leq N \cdot \max_{x\in \partial_\south \Lambda_N} \mu_{\lambda,N}^\pm(\hgt^+_x>KN^{1/3}(\log N)^{2/3})\,.
\end{align*}
Since the upper bound in Proposition~\ref{prop:right-tail-ub} holds uniformly over all $x\in \partial_\south \Lambda_N$ for all $R\leq N^{2/9}$,  
\begin{align}\label{eq:max-fluctuation-upper-bound}
\mu_{\lambda,N}^\pm (\max_x \hgt^+_x>KN^{1/3}(\log N)^{2/3}) \leq C N \exp\big(-K^{3/2} (\log N)/C\big)\,,
\end{align}
from which it follows that if  $K>K_0 = C^{2/3}$, the above is at most $N^{1-(K^{3/2}/C)}=o(1)$ as desired. 
\qed

\section{Proofs of domain enlargement couplings}\label{sec:proofs-domain-enlargements}
In this section, we prove the decreasing and increasing enlargement couplings used throughout the paper. Observe that the single-strip domain enlargements, namely, Propositions~\ref{prop:minus-enlargement}--\ref{prop:plus-enlargement} are instances of the more involved Propositions~\ref{prop:multi-strip-minus-enlargement}--\ref{prop:multi-strip-plus-enlargement}; we thus prove the latter two. 

Recall from Section~\ref{subsec:enlargements}, that $\mathfrak R = (R_1, ..., R_k)$ with $R_i = \llb x_\west^{(i)}, x_\east^{(i)}\rrb \times \llb 0,N\rrb$ as in~\eqref{eq:R-def}. Fixing $r$ and the sequence $\mathfrak h = (h_1 ,..., h_k)$ we recall the definitions $\mathcal E_{\leq \mathfrak h}$ and $\mathcal E_{\geq \mathfrak h}$ from~\eqref{eq:entry-exit-conditioning-event}, as well as the enlargements $ \enlarge_r(R_i)$, $\enlarge^\uparrow _{r,h}(R_i)$, $ \enlarge^\downarrow_{r,h}(R_i)$ and $ \tshape_{r,h}(R_i)$ from~\eqref{eq:enlargement}--\eqref{eq:tshape-enlargement}.  From those we deduce the unions of the T-shaped enlargements, denoted $\mathfrak T$ and the decreasing and increasing enlargements $\mathfrak E^{\uparrow}$ and $\mathfrak E^{\downarrow}$ as in~\eqref{eq:multistrip-enlarge-minus}--\eqref{eq:multistrip-enlarge-plus}. 
For ease of notation, additionally let
\begin{align*}
 \mathfrak R^{\uparrow} = \mathfrak R \cap \mathfrak E^\uparrow\qquad \mbox{and}\qquad \mathfrak R^\downarrow  = \mathfrak R \cap \mathfrak E^\downarrow\,,
\end{align*}
where these intersections are taken element by element.

\begin{proof}[\textbf{\emph{Proof of Proposition~\ref{prop:multi-strip-minus-enlargement}}}]
We prove the inequality using a resampling procedure to couple the boundary conditions induced on $\mathfrak R$ from interfaces in $\mathcal E_{\leq \mathfrak h}$ with those induced by the enlargement. For ease of notation, fixing $\mathfrak h$ let $\cE:=\mathcal E_{\leq \mathfrak h}$, 
Consider $\mathfrak E  = \bigcup_i \enlarge_i$ where $\enlarge_i = \enlarge_{r}(R_i)$, each with ``boundary conditions" denoted $(+, \mathsf E_i, -,h_i)$, which are all-plus on $\partial_{\north,\east,\west} \mathsf E_i$,  minus on $\partial_\south \enlarge_i$ and additionally by minus along the sites in $\partial R_i$ at and below height $h_i$ (i.e., on $\{x_\west^{(i)}\}\times \llb 0,h_i\rrb$, $\{x_\east^{(i)}\}\times \llb 0,h_i\rrb$; see the middle diagram in Figure~\ref{fig:enlargement-proof}). Since the strips $(R_i)_{i\leq k}$ are separated by distances of at least $2r+2$, there is a configuration on $\Lambda \setminus \mathfrak R$ such that, together with the minus spins on $\partial R_i$ at and below height $h_i$, it induces the boundary conditions $(+,\enlarge_i, -,h_i)$ for each $i\leq k$: denote these boundary conditions by $(+, \mathfrak E, -, \mathfrak h)$. Then, by monotonicity in boundary conditions,  
\begin{align*}
\mu_{\lambda,\mathfrak T}^{(\pm,\mathfrak h)}(\sigma_{\mathfrak R} \in \cdot ) \preceq \mu_{\lambda,\fE}^{(+,\fE,-, \fh)} (\sigma_\fR \in \cdot) \qquad \text{and}\qquad  \mu_{\lambda, \fE^\uparrow}^{(\pm,\fh)}(\sigma_{\fR^\uparrow}\in \cdot) \preceq \mu_{\lambda,\fE}^{(+,\fE,- , h)} (\sigma_{\fR^\uparrow} \in \cdot) \,.
\end{align*}
As such, it suffices for us to show that for every decreasing $A\subset \{\pm 1\}^\fR$ (any decreasing event in $\{\pm 1\}^{\fR^{\uparrow}}$ is also a decreasing event on $\{\pm 1\}^\fR$),
\begin{align}\label{eq:WTS-decreasing-enlargement}
\mu_{\lambda, N}^{\pm}(\sigma_\fR \in A \mid \mathcal E) \leq \mu_{\lambda, \fE}^{(+,\fE,- , \fh)} (\sigma_\fR \in A) + CN k e^{ - r/C}\,.
\end{align}

By the domain Markov property, to obtain the distribution induced by $\mu_{\lambda, N}^{\pm}(\cdot \mid \mathcal E)$ on $\mathfrak R$, we wish to understand the law of the configuration induced on $\Lambda \setminus \fR$, then view that as a boundary condition on $\fR$. To that end, observe that 
\begin{align*}
\mu_{\lambda,N}^{\pm}(\sigma_\fR \in \cdot \mid \mathcal E) = \mathbf E_{\mu_{\lambda,N}^{\pm}(\eta_{\Lambda \setminus \fR} \in \cdot \mid \mathcal E)} \big[ \mu_{\lambda,R}^{\eta_{\Lambda \setminus \fR}}( \sigma_\fR \in\cdot \mid \mathcal E_{\eta_{\Lambda \setminus \fR}}) \big]\,,
\end{align*}
where for each fixed $\eta_{\Lambda\setminus \fR}$, the set $\mathcal E_{\eta_{\Lambda\setminus \fR}}$ is the set of configurations on $\fR$ such that concatenated with $\eta_{\Lambda\setminus \fR}$, the full configuration is in $\mathcal E$.  
Now notice that for every $\eta_{\Lambda\setminus \fR}$, the event $\mathcal E_{\eta_{\Lambda\setminus \fR}}$ is an increasing event (more minuses in $\fR$ can only increase $\hgt_{x_\west}^+$ and $\hgt_{x_\east}^+$), and therefore by the FKG inequality, we have for every decreasing event $A \subset \{\pm 1\}^\fR$, 
\begin{align*}
\mu_{\lambda, N}^{\pm}(\sigma_\fR\in A \mid \mathcal E) \leq \mathbf E_{\mu_{\lambda,N}^{\pm}(\eta_{\Lambda \setminus \fR}\in \cdot \mid \mathcal E)} \big[ \mu_{\lambda,\fR} ^{\eta_{\Lambda\setminus \fR}} (\sigma_\fR \in A )\big] = \mathbf E_{\mu_{\lambda,N}^{\pm}(\eta_{\partial \fR}\in \cdot \mid \mathcal E)} \big[ \mu_{\lambda,\fR} ^{\eta_{\partial \fR}} (\sigma_\fR \in A )\big]\,.
\end{align*}
We now wish to understand the induced measure over boundary configurations $\eta_{\partial \fR}$ and couple these to the measure induced by $(+,\mathfrak E, -,  \mathfrak h)$. Fix \emph{any} $\sI\in \mathcal E$. Conditionally on $\sI$, the distribution induced on $\Lambda$ is then given by a combination of the Ising distribution on $\Lambda^+(\sI)$ with its plus-boundary conditions (i.e., on $\partial \Lambda$ and on those sites adjacent to $\sI$ frozen to be plus), and the Ising distribution on $\Lambda^-(\sI)$ with minus-boundary conditions (sites adjacent to $\sI$ frozen minus). 

\begin{figure}
\centering
\includegraphics[width=.95\textwidth]{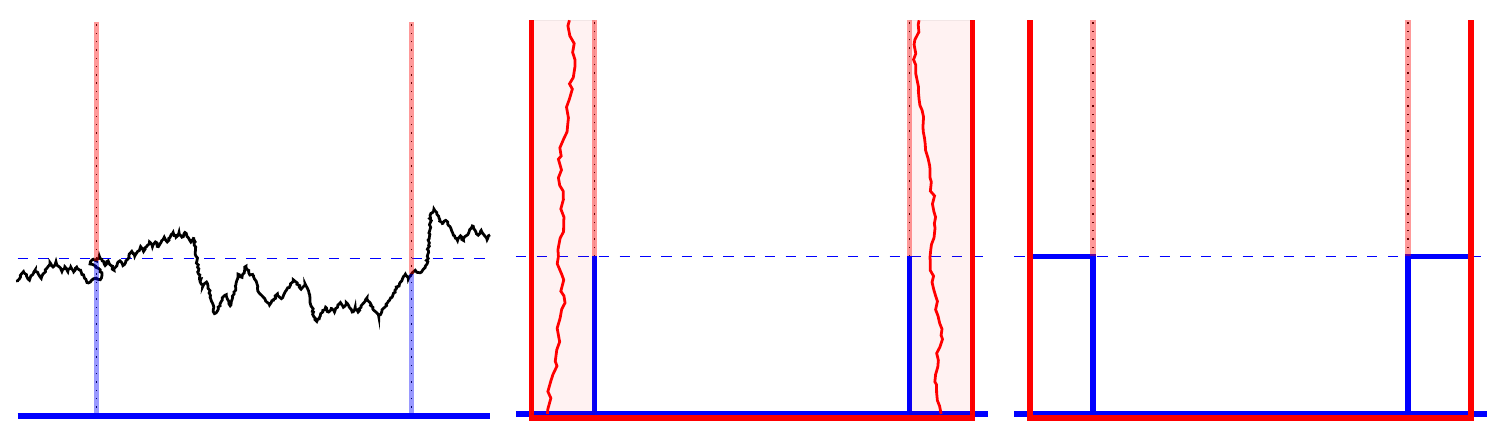}
\caption{Conditionally on an $\mathscr I\in \mathcal E_{\leq \mathfrak h}$, the distribution induced on the east/west boundaries of a strip $R_i$ is mostly plus above $h_i$ and mostly minus below (left). We can couple this induced distribution to the boundary conditions induced by the enlargement $\enlarge_i$ with plus boundary, above height $h_i$ and with minuses up to height $h_i$, in the presence of two vertical crossings in $\enlarge_i \setminus R_i$ (middle). These boundary conditions can be further decreased to the $(\pm,h_i)$ boundary conditions on $\tshape_i$ (right).}\label{fig:enlargement-proof}\vspace{-.2cm}
\end{figure}

Consider first the random configuration induced on the intervals
$$\partial^\uparrow_{\west} R_i:= \partial_\west R_i \times \llb h_i,N\rrb \qquad \text{and} \qquad \partial^\uparrow_\east R: =\partial_\east R_i \times \llb h_i,N\rrb\,,$$ 
collected in the sets $\partial^\uparrow_\west \fR = \bigcup_i \partial_\west^\uparrow R_i$ and $\partial_\east^\uparrow \fR = \bigcup_i \partial_\east^\uparrow R_i$. Define $\partial_\west^\downarrow \fR$ and $\partial_\east^\downarrow \fR$ analogously, with $\llb h_i,N\rrb$ replaced by $\llb 0,h_i\rrb$ above. Since $\partial_\west^\uparrow \fR \cup \partial_\east^\uparrow \fR \subset \Lambda^+$, we have that for every $\sI \in \mathcal E$, 
\begin{align*}
\mu_{\lambda,N}^{\pm} (\sigma_{\partial^\uparrow_\west \fR\cup \partial^\uparrow_\east \fR}\in \cdot \mid \sI) = \mu_{\lambda,\Lambda^+}^{+} (\sigma_{\partial^\uparrow_\west \fR \cup \partial^\uparrow_\east \fR}\in \cdot) \succeq \mu_{\lambda,N}^+  (\sigma_{\partial^\uparrow_\west \fR \cup \partial^\uparrow_\east \fR}\in \cdot)\,.
\end{align*}
Now consider the remainder of $\bigcup_i \partial_\west R_i \cup \partial_\east R_i$, i.e., $\partial_\west^\downarrow \fR \cup \partial_\east^\downarrow \fR$. On these stretches, we can of course only decrease the configuration by setting all sites on $\partial_\west^\downarrow \fR \cup \partial_\east^\downarrow \fR$ to be minus. As such, if we denote by $\partial^\uparrow \fR:= \partial^\uparrow_\west \fR \cup \partial^\uparrow_\east \fR$ and let $\eta_{\partial^\uparrow \fR}$ be a randomly sampled configuration from $\mu_{\lambda,N}^{+}$ on $\partial^\uparrow \fR$, we see that for every $\sI\in \mathcal E$, the measure $\mu_{\lambda,N}^{\pm}(\sigma_{\partial \fR} \in \cdot \mid \sI)$ stochastically dominates the measure on $\partial \fR$ induced by taking $\eta_{\partial^\uparrow \fR}$ on $\partial^\uparrow \fR$, and minus on $\partial^\downarrow_\west \fR \cup \partial^\downarrow_\east \fR$ (as well as $+$ on $\partial_\north \fR$ and $-$ on $\partial_\south \fR$ as induced by $\pm$ on $\Lambda_N$).  
As such, we have that 
\begin{align*}
\mu_{\lambda,N}^{\pm}(\sigma_\fR \in A \mid \mathcal E) & \leq \sup_{\sI \in \mathcal E} \mathbf E_{\mu_{\lambda,N}^{\pm}(\eta_{\uparrow \fR}\in \cdot \mid \sI)} [\mu_{\lambda,\fR}^{(\eta_{\partial^\uparrow \fR}, -,\leq \fh)} (\sigma_\fR\in A)] \leq \mathbf E_{\mu_{\lambda,N}^+(\eta_{\partial^\uparrow \fR}\in \cdot)} [\mu_{\lambda,\fR}^{(\eta_{\partial^\uparrow \fR}, -,\leq \fh)} (\sigma_\fR\in A)]
\end{align*}
where the $(\eta_{\partial^\uparrow \fR}, - ,\leq \fh)$ boundary conditions are plus on $\partial_\north \fR$, distributed according to $\eta_{\partial^\uparrow \fR}$ on ${\partial^\uparrow \fR}$, and are all-minus elsewhere (on $\partial_\south \fR \cup \partial_\west^\downarrow \fR \cup \partial_\east^\downarrow \fR$).
We now wish to couple the induced random boundary conditions $\eta_{\partial^{\uparrow} \fR}$ from $\mu_{\lambda,N}^+$ with the induced random boundary conditions from the collection of enlargements $\fE$ with plus-boundary conditions. See  Figure~\ref{fig:enlargement-proof} for a depiction at a specified strip $R_i$.
Now let 
\begin{align*}
\co_{v}^+ := \bigcap_{i\leq k} \co_{v,i}^+: = \bigcap_{i\leq k} \co_v^+(\llb x_\west^{(i)} - r,x_\west^{(i)} \rrb \times \llb 0,N\rrb) \cap \co_{v}^+ (\llb x_\east^{(i)}, x_\east^{(i)} + r\rrb \times \llb 0,N\rrb )
\end{align*}
where the latter events were defined in~\eqref{eq:crossing-events}.
Since the measure $\mu_{\lambda,\fE}^+$ stochastically dominates the measure $\mu_{\lambda,N}^+$, there is a coupling of the two distributions such that for each $i$, if $\co_{v,i}^+$ holds under the latter, then all sites of that crossing are also plus under the former, and revealing the outer-most such vertical crossings attaining $\co_{v,i}^+$, by the domain Markov property, the two configurations agree everywhere on $\partial^\uparrow_\west R_i \cup \partial^\uparrow_\east R_i$. In particular, since the enlargements $\enlarge_{r,h_i}(R_i)$ are disjoint, this revealing can be done simultaneously for all $i$, and if $\co_{v}^+$ holds under the sample from $\mu_{\lambda ,N}^+$, then the two measures will be coupled to agree on $\partial_\west^\uparrow \fR\cup \partial_\east^\uparrow \fR$ under $\mu_{\lambda ,\fE}^+$. 

Therefore, by monotonicity, for every decreasing event $A_{\partial^\uparrow \fR}\in \{\pm 1\}^{\partial^\uparrow \fR}$, 
\begin{align*}
\mu_{\lambda,N}^{+}(\eta_{\partial^\uparrow \fR}\in A_{\partial^\uparrow \fR}) & \leq  \mu_{\lambda, \fE}^+(\eta_{\partial^\uparrow \fR}\in A_{\partial^\uparrow \fR}) + \mu_{\lambda,N}^+\big((\co_v^+)^c\big) \\
& \leq  \mu_{\lambda, \fE}^+(\eta_{\partial^\uparrow \fR}\in A_{\partial^\uparrow \fR}) + \sum_{i\leq k} \mu_{0,\mathbb Z^2}^+\big( (\co_{v,i}^+)^c\big)\,.
\end{align*} 
By~\eqref{eq:crossing-probability-bound}, each of the summands in the second term is at most $C N e^{ - r /C}$ for some $C(\beta)>0$. 
Combining the above, we deduce that for every decreasing event $A$ on $\fR$, 
\begin{align*}
\mu_{\lambda,N}^{\pm}(\sigma_\fR\in A \mid \mathcal E) & \leq \mathbf E_{\mu_{\lambda, \fE}^{+}(\eta_{\partial^\uparrow \fR}\in \cdot)}\big[\mu_{\lambda,\fR}^{(\eta_{\partial^\uparrow \fR}, -,\leq \fh)} (\sigma_\fR \in A)\big] + C N k  e^{ - r/C}\,.
\end{align*}  
Finally, observe that if instead of the induced configurations on $\partial^\uparrow \fR$ being drawn from an independent sample of $\mu_{\lambda,\fE}^+(\eta_{\partial^\uparrow \fR}\in \cdot)$, they are drawn from the measure $\mu_{\lambda,\fE}^{(+,\fE,-,\fh)}$, this only decreases them (recall $(+,\fE,-,\fh)$ is $+$ on $\partial_{\east,\west} \fE$, $\pm$ on $\partial \Lambda_N$, and additionally minus on the sites along $\partial^\downarrow \fR$). But by the domain Markov property, this would correspond exactly to the distribution $\mu_{\lambda, \mathfrak E}^{(+,\mathfrak E, - \fh)}$ and therefore, by monotonicity, we have the desired inequality, 
\[\mathbf E_{\mu_{\lambda, \fE}^{+}(\eta_{\partial^\uparrow \fR}\in \cdot)}\big[\mu_{\lambda,\fR}^{(\eta_{\partial^\uparrow \fR}, -,\leq \fh)} (\sigma_\fR \in A)\big]  \leq \mu_{\lambda,\fE}^{(+,\fE,-,\fh)}(\sigma_\fR \in A)\,. \qedhere
\]
\end{proof}

\begin{proof}[\textbf{\emph{Proof of Proposition~\ref{prop:multi-strip-plus-enlargement}}}]
Fix $\fh$ and let $\mathcal E = \mathcal E_{\geq \fh}$. We may assume that $\max_{i\leq k} h_i\leq \mathcal M$, as otherwise the third term on the right-hand side of the proposition would be greater than one, since $\max_{x}\hgt_x^+ >\mathcal M$ would be implied by $\mathcal E_{\geq \fh}$. By the domain Markov property, we can express 
\begin{align*}
\mu_{\lambda,N}^{\pm} (\sigma_\fR \in \cdot \mid \mathcal E) = \mathbf E_{\mu_{\lambda, N}^{\pm}(\eta_{\Lambda\setminus \fR} \in \cdot \mid \mathcal E)} \big[\mu_{\lambda, \fR}^{\eta_{\Lambda\setminus \fR}} (\sigma_\fR \in \cdot \mid \mathcal E_{\eta_{\Lambda\setminus \fR}})\big]\,,
\end{align*}
where $\mathcal E_{\eta_{\Lambda\setminus \fR}}$ is the set of configurations on $\fR$ such that concatenated with $\eta_{\Lambda\setminus \fR}$, the full configuration is in $\cE$. For every $\eta_{\Lambda\setminus \fR}$, the event $\mathcal E_{\eta_{\Lambda\setminus \fR}}$ is a decreasing event;  thus for every increasing $A\in \{\pm1\}^\fR$,  
\begin{align*}
\mu_{\lambda, N}^{\pm}(\sigma_\fR \in A \mid \mathcal E) \leq \mathbf E_{\mu_{\lambda, N}^{\pm}(\eta_{\Lambda\setminus \fR} \in \cdot \mid \mathcal E)} \big[\mu_{\lambda,\fR}^{\eta_{\Lambda\setminus \fR}} (\sigma_\fR \in A)\big] = \mathbf E_{\mu_{\lambda, N}^{\pm}(\eta_{\partial \fR} \in \cdot \mid \mathcal E)} \big[\mu_{\lambda,\fR}^{\eta_{\partial \fR}} (\sigma_\fR \in A)\big]\,.
\end{align*}
We can split the above expectation up into interfaces which are in  $\Psi_{\mathcal M} = \{\mathscr I : \max_{x} \hgt_x^+ >\mathcal M\}$, and interfaces not in $\Psi_{\mathcal M}$, so that the above is at most 
\begin{align}\label{eq:splitting-according-to-Psi-enlargement}
\mu_{\lambda,N}^{\pm}(\Psi_{\mathcal M} \mid \mathcal E )+ \mathbf E_{\mu_{\lambda,N}^{\pm}(\eta_{\partial \fR}\in \cdot \mid \mathcal E, \Psi_{\mathcal M}^c)}\big[\mu_{\lambda,\fR}^{\eta_{\partial \fR}}(\sigma_{\fR} \in A)\big] \!\leq \!\frac{\mu_{\lambda,N}^{\pm} (\Psi_{\mathcal M})}{\mu_{\lambda,N}^{\pm}(\mathcal E)} + \mathbf E_{\mu_{\lambda,N}^{\pm}(\eta_{\partial \fR}\in \cdot \mid \mathcal E, \Psi_{\mathcal M}^c)}\big[\mu_{\lambda,\fR}^{\eta_{\partial \fR}}(\sigma_{\fR} \in A)\big]
\end{align}
Consider now the second term, for which we need to understand the boundary conditions induced on $\partial \mathfrak R$ by the conditional distribution under $\mathcal E\cap \Psi_{\mathcal M}^c$. Fix \emph{any} $\sI\in \mathcal E \cap \Psi_{\mathcal M}^c$. Conditionally on $\sI$, the distribution induced on $\Lambda$ is the product of the Ising distribution on $\Lambda^-(\sI)$ with $-\,$-b.c.\ and the distribution on $\Lambda^+(\sI)$ with $+\,$-b.c. Consider first the random configuration induced on 
$$
\partial^\downarrow_\west R_i := \{x_\west^{(i)}\} \times \llb 0,h_i\rrb \qquad \text{and}\qquad \partial^\downarrow_\east R_i: =  \{x_\east^{(i)}\} \times \llb 0,h_i \rrb\,,
$$
whose unions over $i$ are collected in the sets $\partial_\west^\downarrow \fR$ and $\partial_\east^\downarrow \fR$, respectively. 
Since $\partial^\downarrow \fR:= \partial^\downarrow_\west \fR \cup \partial^\downarrow_\east \fR$ is a subset of $\Lambda^-$, we have uniformly over  $\sI\in \mathcal E\cap \Psi_{\mathcal M}^c$, that 
\begin{align*}
\mu_{\lambda, N}^{\pm}(\eta_{\partial^\downarrow \fR} \in \cdot\mid \sI) = \mu_{\lambda, \Lambda^-}^-(\eta_{\partial^\downarrow \fR}\in \cdot) \preceq \mu_{\lambda, \fE ' \cup \Lambda_{N,\mathcal M}}^-(\eta_{\partial^\downarrow \fR}\in\cdot)\,,
\end{align*}
where in this case, we let 
\begin{align*}
\fE':= \bigcup_i \enlarge '_r(R_i) : = \llb x_\west^{(i)} -r, x_\east^{(i)} +r\rrb \times \llb 0, h+r\rrb\,,
\end{align*}
and used the fact that $\Lambda^- \subset \fE' \cup \Lambda_{N,\mathcal M}$ necessarily. Now we wish to couple the distribution induced by this latter measure on $\eta_{\partial^\downarrow \fR}$ with that induced by $\mu_{\lambda, \fE' }^-$. By monotonicity, there is a coupling so that these two measures agree on $\fR^\downarrow=\mathfrak R \cap \mathfrak E^\downarrow$ under the event 
\begin{align*}
\circuit^-_{\sqcap} & : = \bigcap_{i\leq k} \circuit^-_{\sqcap,i}\,, \qquad \mbox{where} \\
\circuit^-_{\sqcap,i} & :  = \co_v^-(\llb x_\west^{(i)} - r,x_\west^{(i)} \rrb \times \llb 0,h_i+r\rrb)\cap \co_h^-( \llb x_\west^{(i)} - r,x_\east^{(i)}+r\rrb \times \llb h_i,h_i+r\rrb ) \\
& \qquad \cap \co_v^-(\llb x_\east^{(i)} - r,x_\east^{(i)} \rrb \times \llb 0,h_i +r\rrb)\,,
\end{align*}
 in the configuration sampled from $\mu_{\lambda, \fE' \cup \Lambda_{N,\mathcal M}}^-$. That is to say, for every increasing $A_{\partial^\downarrow \fR}\in \{\pm 1\}^{\partial^\downarrow \fR}$,
\begin{align*}
\mu_{\lambda, \fE' \cup \Lambda_{N,\mathcal M}}^- (\eta_{\partial^\downarrow \fR} \in A_{\partial^\downarrow \fR} ) & \leq \mu_{\lambda, \fE'}^-(\eta_{\partial^\downarrow \fR} \in A_{\partial^\downarrow \fR}) + \mu_{\lambda, \fE' \cup \Lambda_{N,\mathcal M}}^- \big((\circuit_\sqcap^-)^c\big) \\ 
& \leq  \mu_{\lambda, \fE'}^-(\eta_{\partial^\downarrow \fR}\in A_{\partial^\downarrow \fR}) +  \sum_{i\leq k} e^{2 c_\lambda |\fE' \cup \Lambda_{N,\mathcal M}|/N}\mu_{0, \mathbb Z^2}^- \big((\circuit_{\sqcap,i}^-)^c\big)\,,
\end{align*}
where the final inequality used \eqref{eq:radon-nikodym-bound} along with $\mu_{0, \fE' \cup \Lambda_{N,\mathcal M}}^- \preceq \mu_{0, \mathbb Z^2}^-$,  and the fact that $(\circuit_{\sqcap,i}^-)^c$ are increasing events.
Observe that the area confined in $\fE' \cup \Lambda_{N,\mathcal M}$ is at most 
$$N\mathcal M + r \sum_{i\leq k} |x_\west^{(i)} - x_\east^{(i)}|+ 2k r^2\,.$$ 
As such, we have using~\eqref{eq:crossing-probability-bound}--\eqref{eq:circuit-probability-bound} that for increasing $A_{\partial^\downarrow \fR}\in \{\pm 1\}^{\partial^\downarrow \fR}$, 
\begin{align*}
\mu_{\lambda, \fE' \cup \Lambda_{N,\mathcal M}}^{-}(\eta_{\partial^\downarrow \fR} \in A_{\partial^\downarrow \fR}) \leq \mu_{\lambda, \fE'}^-(\eta_{\partial^\downarrow \fR}\in A_{\partial^\downarrow \fR} ) + 3C N k e^{ 4c_\lambda (\mathcal M+\frac{r \sum_i |x_\west^{(i)} - x_\east^{(i)}|+ k r^2}N)}  e^{- r/C}  \,.
\end{align*}
At this point the proof is almost complete, we simply have to combine the above ingredients together. 
Towards this, we observe using the above that for every increasing event $A\in \{\pm 1\}^{\fR^\downarrow}$, 
\begin{align}\label{eq:R-downarrow-event-probability}
\mathbf E_{\mu_{\lambda,N}^{\pm}(\eta_{\partial \fR}  \in \cdot \mid \mathcal E, \Psi_{\mathcal M}^c)}\big[ & \mu_{\lambda,\fR}^{\eta_{\partial \fR}}  (\sigma_{\fR^\downarrow} \in A)\big]  \\
 & \leq \sup_{\mathscr I \in \mathcal E\cap \Psi_{\mathcal M}^c} \mathbf E_{\mu_{\lambda,N}^{\pm}(\eta_{\partial \fR} \in \cdot \mid \mathscr I)} \big[ \mu_{\lambda,\fR}^{\eta_{\partial \fR}} (\sigma_{\fR^\downarrow} \in A)\big] \nonumber \\
& \leq \mathbf E_{\mu_{\lambda, \fE'}^- (\eta_{\partial^\downarrow \fR} \in \cdot)} \big[\mu_{\lambda, \fR}^{(\eta_{\partial^\downarrow \fR},+,>\fh)} (\sigma_{\fR^\downarrow} \in A)\big] +  C N k  e^{ 4c_\lambda (\mathcal M+\frac{r\sum_i |x_\west^{(i)} - x_\east^{(i)}|+ k r^2}N)}  e^{- r/C}\,, \nonumber
\end{align}
where the boundary conditions $(\eta_{\partial^\downarrow \fR},+,>\fh)$ denote those which are distributed according to $\eta_{\partial^\downarrow \fR}$ on ${\partial^\downarrow \fR}$,  are all-plus on $\partial^\uparrow \fR$, and are $\pm$ on $\partial \Lambda_N$. Evidently, the law of $\eta_{\partial^\downarrow \fR}$ is only increased by adding pluses  above heights $h_i$ in each strip $\enlarge'_r (R_i)$ to the measure $\mu_{\lambda,\fE'}^-$ from which $\eta_{\partial^\downarrow \fR}$ is drawn  (yielding the measure $\mu_{\lambda,\fE^\downarrow}^{(\pm,\fh)}$). By the domain Markov property and monotonicity, we have  
\begin{align*}
\mathbf E_{\mu_{\lambda, \fE'}^- (\eta_{\partial^\downarrow \fR} \in \cdot)} \big[\mu_{\lambda, \fE'}^{(\eta_{\partial^\downarrow \fR},+,>\fh)} (\sigma_{\fR^\downarrow} \in A)\big]  \leq  \mu_{\lambda,\fE^\downarrow}^{(\pm,\fh)} (\sigma_{\fR^\downarrow} \in A)\,.
\end{align*}
Plugging this into~\eqref{eq:R-downarrow-event-probability} and combining with~\eqref{eq:splitting-according-to-Psi-enlargement}, we obtain the desired. 
\end{proof}
\vspace{-.2cm}
\bibliographystyle{abbrv}
\bibliography{references}

\end{document}